\documentclass[12pt]{amsart}

\usepackage[margin=1in]{geometry} 
\usepackage{caption}
\usepackage{microtype}
\usepackage{subcaption}
\usepackage{tikz} 
\usepackage{fourier}
\usepackage{pst-poly}
\usepackage{multido}
\usetikzlibrary{shapes.geometric}
\usepackage{pgfplots}
\usepackage{mathrsfs}
\usetikzlibrary{arrows}
 \usepackage{fancyhdr}

\usepackage[utf8]{inputenc}
\usepackage{graphicx}
\graphicspath{{Figures/}}
\usepackage{amssymb}
\usepackage{bbm}
\usepackage{bm}
\usepackage{mathtools}
\usepackage{amsthm}
\usepackage{amsbsy}
\usepackage{amstext}
\usepackage{amscd}
\usepackage{nccmath}
\usepackage{enumitem}
\usepackage{amsmath}

\allowdisplaybreaks
\newcommand{\ab}{\allowbreak}

\newcommand{\subscript}[2]{$#1 _ #2$}

\newtheorem{theorem}{Theorem}[section]
\newtheorem{proposition}[theorem]{Proposition}
\newtheorem{corollary}[theorem]{Corollary}
\newtheorem{lemma}[theorem]{Lemma}

\theoremstyle{definition}
\newtheorem{definition}[theorem]{Definition}
\newtheorem{notation}[theorem]{Notation}
\newtheorem{example}[theorem]{Example}
\newtheorem{remark}[theorem]{Remark}

\newcommand{\E}{\mathbb{E}}
\newcommand{\cP}{\mathcal{P}}
\newcommand{\C}{K}

\begin{document}\thispagestyle{empty}



\begin{center}
\Large\bf
Ultra high order cumulants and quantitative CLT for polynomials in  Random Matrices
\end{center}

\vspace{0.5cm}
\renewcommand{\thefootnote}{\fnsymbol{footnote}}
\hspace{5ex}	
\begin{center}
 \begin{minipage}[t]{0.4\textwidth}
\begin{center}
Zhigang Bao\footnotemark[1]  \\
\footnotesize {University of Hong Kong}\\
{\it zgbao@hku.hk}
\end{center}
\end{minipage}
\begin{minipage}[t]{0.4\textwidth}
\begin{center}
Daniel Munoz George\footnotemark[2]  \\ 
\footnotesize {University of Hong Kong}\\
{\it dmunozgeorge@gmail.com}
\end{center} 
\end{minipage}
\end{center}

\footnotetext[1]{Supported by Hong Kong RGC Grant 16303922, NSFC12222121 and NSFC12271475}
\footnotetext[2]{Supported by  National Key R\&D Program of China (No. 2023YFA1010400)}

\vspace{4ex}
\begin{center}
 \begin{minipage}{0.8\textwidth} {
  Abstract: From the study of the high order freeness of random matrices, it is known that the order $r$ cumulant of the trace of a polynomial of $N$-dimensional GUE/GOE is of order $N^{2-r}$ if $r$ is fixed. In this work, we extend the study along three directions. First, we also consider generally distributed Wigner matrices with subexponential entries. Second, we include the deterministic matrices into discussion and consider arbitrary polynomials in random matrices and deterministic matrices.  Third, more importantly, we consider  the  ultra high order cumulants in the sense that $r$ is arbitrary, i.e., could be $N$ dependent. Our main results are the upper bounds of the ultra high order cumulants, for which not only the $N$-dependence but also the $r$-dependence become significant. These results are then used to derive three types of quantitative CLT for the trace of any given self-adjoint polynomial in these random matrix variables: a CLT with a Cram\'{e}r type correction, a Berry-Esseen bound, and a concentration inequality which captures both the Gaussian tail in the small deviation regime and $M$-dependent tail in the large deviation regime, where $M$ is the degree of the polynomial.  In contrast to the second order freeness  which implies the CLT for linear eigenvalue statistics of polynomials in random matrices, our study on the ultra high order cumulants leads to the quantitative versions of the CLT. 
 
}
\end{minipage}
\end{center}

\vspace{4ex}
\section{Introduction}

Since Voiculescu discovered the deep connection between random matrices and free probability in his seminal work \cite{V}, a substantial amount of research has been devoted to understanding the spectral properties of general polynomials in random matrices and deterministic matrices, with the aid of free probability theory. In particular, using freely independent random variables to characterize the limiting spectral properties of polynomials in random matrices and deterministic matrices has become a general framework. Specifically, Voiculescu revealed that the limit of the expected empirical spectral distribution of a random matrix polynomial can be characterized by the distribution of the corresponding polynomial of free random variables. This is known as the first order freeness; see \cite{VDN92, NS, Dyk93} for more details.
Following this, further research such as \cite{MSpart1, MSS, CMSS} has explored the second-order fluctuation properties of linear eigenvalue statistics(LES) around the first order limit, by looking at the high order correlation functions.  It has been now well understood that for general self-adjoint random matrix polynomials, their LESs are asymptotically Gaussian, i.e., CLT holds. A notable feature of this type of CLT is that  for random matrix polynomials formed from $N$-dimensional random matrices, the CLT for their trace does not require additional normalization in terms of $N$ factor, unlike the classical iid sum which requires an $N^{-1/2}$ scaling. When the random matrices are from a Gaussian ensemble such as GUE/GOE or more generally unitarily/orthogonally invariant ensembles,  the variance of this CLT has been  precisely characterized by using the limiting free random variables; see \cite{MSpart1, MSS, CMSS, MP13, R14}. These results are referred to as the second order freeness.
When Gaussian matrices are replaced by Wigner matrices with general distributions, the variance depends on more parameters, but the CLT still holds in general cases; see \cite{MMPS,BCDM} for instance. Since the variance can also be viewed as the second-order cumulant, and the proof of the CLT requires demonstrating that higher-order cumulants can be neglected for any fixed order, research into higher-order cumulants has led to the introduction of the concept of high order freeness in \cite{MSpart1, MSS, CMSS}. In particular, it is revealed that for a polynomial of GUE matrices, the $r$-th cumulant of the trace of the polynomial is of order $N^{2-r}$ as long as $r$ is fixed. A precise expression of the leading order of this fixed $r$-order cumulant has also been given in \cite{MSpart1}. A recent  extension  of this result to generally distributed Wigner matrices but with Gaussian diagonal entries has been given in \cite{MM}. We emphasize that the studies in \cite{MSpart1} and \cite{MM} do not include deterministic matrices into the polynomials of Wigner matrices. We also refer to \cite{BS09, S06} for other studies related to high order freeness.

In this work, we extend this line of research along three directions. First, we also consider generally distributed Wigner matrices with subexponential entries. Second, we include the deterministic matrices into discussion and consider arbitrary polynomials in random matrices and deterministic matrices.  Third, more importantly, we consider  the  ultra high order cumulants in the sense that $r$ is arbitrary, i.e., could be $N$ dependent. Our main results are the upper bounds of the ultra high order cumulants, for which not only the $N$-dependence but also the $r$-dependence becomes significant. Including the $r$-dependence in the general $(r,N)$ setting is actually the major task of this work. But we must notice that even the $N$ order in case $r$ is fixed is not known in the literature when deterministic matrices are also involved. In particular, we will show that for generally distributed Wigner matrices, the $N$ factor is no longer $N^{2-r}$. Instead, one has $N^{1-r/2}$ in general. 

A primary application of the upper bounds on the ultra high order cumulants is to obtain various versions of quantitative CLTs for the LES of polynomials of random matrices. Apart from the aforementioned work on second order freeness, the CLT for the LES of random matrix models has been widely studied, and a vast amount of research has been devoted to this topic. For instance, we refer to \cite{KKP, LP1, BH, SW13, LS22, BY05, C01, Ch09, Sh11, HK17} for studies on a single Wigner matrix model and \cite{DF19, JL20, LSX} on deformed Wigner matrix which could be regarded as a rather simple polynomial of Wigner matrices and deterministic matrices. We also refer to the work \cite{CES23, CES22, Re23} for the CLT of the trace of random matrices of the form $f_1(X)D_1\cdots f_k(X)D_k$ for Wigner matrix $X$, deterministic matrices $D_i$'s and regular functions $f_i$'s. This model goes beyond the polynomial of Wigner matrix $X$ and deterministic matrices.    Despite the fact that the CLT for LES has been a classic topic in random matrix theory, the study of quantitative versions of this type of CLT is more recent, except for the results on circular ensembles \cite{DS94, Stein, Joh97, CJ21, JL20, DS11}. For Hermitian ensembles, the study on the quantitative CLT of LES emerged only very recently. In \cite{LLW19}, the authors obtained a convergence rate in the quadratic Kantorovich distance, for the CLT of the LES of the $\beta$-ensembles with one-cut potentials. In \cite{BB21}, the authors obtained a convergence rate of order $N^{-1/5}$ in Kolmogorov distance, for the CLT of LES of GUE/LUE/JUE. Both the works \cite{LLW19} and \cite{BB21} consider the invariant ensembles and the convergence rates do not seem to be optimal in general. For non-invariant ensembles such as the generally distributed Wigner matrices, in a recent work \cite{BH}, a nearly optimal convergence rate in Kolmogorov distance was obtained. Particularly, for all test functions $f\in C^5(\mathbb{R})$, it is shown that the convergence rate is either $N^{-1/2+\varepsilon}$ or $N^{-1+\varepsilon}$, depending on the first Chebyshev coefficient of $f$ and the third moment of the diagonal matrix entries.  We also refer to a recent work \cite{CES23}, a CLT for the generalized linear statistics $\text{Tr} f(D)X$ with deterministic matrix $D$ was established for $f\in H_0^2(\mathbb{R})$ on both macroscopic and mesoscopic scale. Especially, on macroscopic scale, the result in \cite{CES23} indicates a $N^{-1/2}$ convergence rate in the sense of moments.  When extending the study from classical single random matrix models, such as the Wigner matrix, to arbitrary self-adjoint polynomials of a collection of Wigner matrices and deterministic matrices, research has focused on a more specialized class of test functions, such as polynomials, which are more amenable to combinatorial approaches. As mentioned earlier, the research on the second order freeness can indicate the CLT for LES of polynomials of unitarily/orthogonally invariant matrices such as GUE/GOE \cite{MSS}. When one considers self-adjoint polynomials of generally distributed Wigner matrices and deterministic matrices, the second order freeness no longer holds in the sense of \cite{MSpart1} but CLT still applies to LES under mild conditions.   So far, no quantitative result on the CLT of LES for a general self-adjoint polynomial is known. In this work, applying the bound on ultra high order cumulants, we obtain three types of quantitative results in the spirit of the cumulant method for quantitative CLT \cite{DJ,SS12}: a CLT with a Cram\'{e}r type correction, a Berry-Esseen bound for the CLT and a concentration inequality which captures both the Gaussian tail in the small deviation regime and $M$-dependent tail in the large deviation regime, where $M$ is the degree of the polynomial. 

We start with some preliminary definitions in Section \ref{s.pre} and state our main results in Section \ref{Section: Main results}. 

\subsection{Preliminaries} \label{s.pre}

We first define the cumulants following \cite[Chapter 1]{MS}. For a collection of random variables $(Y_m)_m$ we define the $n$-mixed cumulant of $Y_1,\dots,Y_n$, denoted by $K_n(Y_1,\dots,Y_n)$, via the recursive equation,
\begin{equation}\label{Equation_Moment cumulant relation 1}
\E(Y_1\cdots Y_n)=\sum_{\pi \in \cP(n)}K_{\pi}(Y_1,\dots,Y_n).
\end{equation}
Equivalently, the M\"{o}bius inversion theorem says,
\begin{equation}\label{Equation_Moment cumulant relation 2}
K_n(Y_1,\dots,Y_n)=\sum_{\pi\in \cP(n)}(-1)^{\#(\pi)-1}(\#(\pi)-1)!\E_{\pi}(Y_1\cdots Y_n).
\end{equation}
Here in the above notations $\cP(n)$ represents the set of partitions of $[n] \vcentcolon = \{1,\dots,n\}$ and for a partition $\pi$, the number of blocks of $\pi$ is denoted by $\#(\pi)$ and
$$K_{\pi}(Y_1,\dots,Y_n) = \prod_{B=\{i_1,\dots,i_k\}\text{ block of }\pi} K_{|B|}(Y_{i_1},\dots,Y_{i_k}),$$
$$\E_{\pi}(Y_1\cdots Y_n) = \prod_{B=\{i_1,\dots,i_k\}\text{ block of }\pi} {{\E(Y_{i_1}\cdots Y_{i_k})}}.$$
The Equations (\ref{Equation_Moment cumulant relation 1}) and (\ref{Equation_Moment cumulant relation 2}) are usually called the moment-cumulant relations. When the variables $Y_n$ are all the same, say $Y$, then, $K_n(Y) \equiv K_n(Y,\dots,Y)$ recovers the definition of the $n$-cumulant of $Y$.


\begin{example}\label{Example_Gaussian cumulants} 
If $Z\sim \mathcal{N}(0,1)$ has standard real normal distribution then its cumulants are given by $K_n(Z)=0$ for all $n$ except $K_2(Y)=1$. If $Z\sim \mathbb{C}\mathcal{N}(0,1)$ has  standard complex Gaussian distribution, i.e.,  $Z= a+\mathrm{i} b$ where $a$ and $b$ are independent $\mathcal{N}(0,1/2)$, then any cumulant of $Z$ is $0$ except $K_2(Z,\overline{Z})=1$. This can be easily derived from the multi-linearity of the cumulants and the cumulants of $a$ and $b$.
\end{example}

We will consider  real/complex $N\times N$ random Hermitian matrices $X$ for various models, $\mathcal{X}\sim GUE,GOE,Wigner$. For the sake of clarity, we make the following definition for the Wigner model considered in our paper, and GUE/GOE are special examples of it.

\begin{definition}\label{Definition: Wigner Matrix}
By a Wigner matrix $X\equiv X(N)=\frac{1}{\sqrt{N}}(x_{i,j})\in \mathbb{C}^{N\times N}$ we mean a self-adjoint random matrix, i.e., $x_{i,j}=\overline{x_{j,i}}$, with independent (up to symmetry) entries which satisfy the moment conditions $\mathbb{E}x_{i,j}=0$ for all $i,j$ and $\mathbb{E}|x_{i,j}|^2=1$ for $i\neq j$. We further make the following assumptions. 

\begin{itemize}


 

\item[$\diamond$]
the entries above the diagonal, $\{x_{i,j}\}_{i<j}$, are
identically distributed as a complex random variable $x_o$ whose odd  cumulants are all $0$, that is, $\C_{2n+1}(x_o^{\epsilon_1},\dots,x_o^{\epsilon_{2n+1}})=0$ for any $\epsilon_i\in \{1,-1\}$ and $n\in\mathbb{N}$. Here $x_o^{1}=x_o$ while $x_o^{-1}=\overline{x_o}$ means its conjugate. We further assume $x_o\stackrel{d}= \overline{x_o}$. 

\item[$\diamond$] 
the diagonal entries, $\{x_{i,i}\}_{i}$, are identically
distributed as a real random variable $x_d$ whose odd cumulants are all $0$, that is, $\C_{2n+1}(x_d)=0$ for any $n\in\mathbb{N}$;



\item[$\diamond$] For any $r$, any $i,j$ and $\epsilon_i\in \{1,-1\}$, 
$|\C_r(x_{i,j}^{\epsilon_1},\dots,x_{i,j}^{\epsilon_r})|\leq \mathsf{C}^rr!$ for some constant $\mathsf{C}>0$. 

\end{itemize}
We call a Wigner matrix a  GUE if its off-diagonal entries are $\mathbb{C}\mathcal{N}(0, N^{-1})$ and the diagonal entries are $\mathcal{N}(0, N^{-1})$. We call a Wigner matrix a GOE if its off-diagonal entries are $\mathcal{N}(0, N^{-1})$ and the diagonal entries are $\mathcal{N}(0, 2N^{-1})$.
\end{definition}

\begin{remark}
We remark that the condition that the odd cumulants are zero is satisfied by all symmetrically distributed random variables. Furthermore, it is known that the cumulant bound condition is equivalent to the entries being subexponential, that is, $\mathbb{E}(\exp(\delta |x_{i,j}|))<\infty$ for some $\delta>0$.   We refer to Lemma 2.8 of \cite{DJ} for this fact. The subexponential assumption is natural, as we will consider cumulants of arbitrary order. Nevertheless, it is possible to consider distributions with heavier tails of the form $\mathbb{P}(|x_{ij}|>t)\sim \exp(-t^{\gamma})$ with $\gamma\in (0,1)$. We leave this potential extension to interested readers. Further, note that the vanishing of all odd cumulants together with a subexponential tail implies symmetry of the distribution for a real random variable. But we prefer to state our conditions in terms of cumulants as they will be directly used in future discussions. We believe the vanishing odd cumulants condition can be weakened, but this is beyond our current techniques, as it will be crucial when we reduce our combinatorial argument for the Wigner case to the Gaussian case; see the proof of Theorem \ref{Theorem: Order of Wigner case for t()} for instance.
\end{remark}

Let $\mathcal{X}=(X_i)_{i\in I}$ be a collection of independent random matrices. We write
$$X_p= \frac{1}{\sqrt{N}}(x_{i,j}^{(p)})_{i,j=1}^N$$
for any $p\in I$. In the sequel, we say that $\mathcal{X}\sim GUE, GOE, Wigner$ if it is a collection of independent GUE, GOE, or Wigner matrices, respectively. In addition,  let $\mathcal{D}=(D_j)_{j\in J}$ be a collection of $N\times N$ deterministic matrices $D_j\equiv D_j(N)$ such that
$$\sup_N ||D_j||<\infty$$ for any $j\in J$, where $||\cdot||$ denotes the operator norm. We consider a polynomial in $\mathcal{X}$ and $\mathcal{D}$ variables as follows
\begin{align}
P(\mathcal{X},\mathcal{D}):=X_{i^{(1)}_1}D_{j^{(1)}_1}\cdots X_{i^{(1)}_{m_1}}D_{j^{(1)}_{m_1}}+\cdots +X_{i^{(t)}_1}D_{j^{(t)}_{1}}\cdots X_{i^{(t)}_{m_t}}D_{j^{(t)}_{m_t}}, \label{def of P}
\end{align}
where we allow repetition of the indices $i$ and $j$. As we will study the centered trace of a polynomial of $\mathcal{X}$ and $\mathcal{D}$ variables, we do not need to consider a monomial with deterministic matrices only as it will contribute nothing. Further, due to the cyclicity of the trace, we can also take into consideration the models where the first matrix of some monomial is deterministic. Finally, since the deterministic matrices might be the identity we can also take into consideration the models including powers of the random matrices, for instance. Hence, in terms of the quantity of interest in (\ref{object}) below, our assumption on the form of $P(\mathcal{X},\mathcal{D})$ is fully general. A key parameter in our discussion will be
\begin{align}\label{Definition: M}
M:=\max_{k=1,\ldots,t} m_{k},
\end{align}
which will be called the \textit{degree} of the polynomial. Throughout the paper, we always consider the setting that $M$ and $t$ are fixed.  We further denote 
\begin{align}\label{Definition: K}
K=\vcentcolon \max_{u\in \{j_1^{(1)},\dots,j_{m_t}^{(t)}\}}\sup_{N} ||D_u||.
\end{align}
 
 We study the (possibly complex) random variable,
\begin{align}
\mathbf{X}=\text{Tr} P(\mathcal{X},\mathcal{D}) - \E \text{Tr} P(\mathcal{X},\mathcal{D}), \label{object}
\end{align}
and its normalized version 
\begin{align}
\mathfrak{X}=\frac{\text{Tr} P(\mathcal{X},\mathcal{D}) - \E \text{Tr} P(\mathcal{X},\mathcal{D})}{\sqrt{\text{Var}(\text{Tr} P(\mathcal{X},\mathcal{D}))}}. \label{object_normalized}
\end{align}
 Here the variance of a complex random variable $z$ is given by $\mathbb{E}(|z-\mathbb{E}(z)|^2)$.

For our main result we investigate upper bounds for all cumulants of the random variable $\mathbf{X}$ for an arbitrary polynomial with fixed degree $M$. Later on, we use our results to provide quantitative CLTs and concentration results for $\mathfrak{X}$, for the case when the polynomial is self-adjoint (under cyclic permutation of matrices if necessary). We detail this in Section \ref{Section: Main results} where we state our results.

\subsection{Main results}\label{Section: Main results}

We might consider the following conditions. For each of our results we will state which of them are used.
\begin{enumerate}[label=$($\subscript{\mathbf{C}}{{\arabic*}}$)$]
    \item $\mathcal{D}=(D_j)_{j\in J}$ is a collection of deterministic matrices such that $\sup_N ||D_j||<\infty$ for any $j\in J$, where $||\cdot||$ denotes the operator norm; 

    \item {{$\text{Var}(\text{Tr}P(\mathcal{X},\mathcal{D}))\vcentcolon = \mathbb{E}(|\text{Tr}P(\mathcal{X},\mathcal{D})-\mathbb{E}\text{Tr}P(\mathcal{X},\mathcal{D})|^2)>\mathfrak{c}$ for some constant $\mathfrak{c}>0$.}}

    \item $P$ is a self-adjoint polynomial, that is, the random matrix $P(\mathcal{X},\mathcal{D})$ is self-adjoint.

\end{enumerate}

Instead of $(\mathbf{C}_1)$, if we have the following condition $(\mathbf{C}_1')$, we can derive sharper results in the Wigner case. 

\begin{enumerate}

    \item[$(\mathbf{C}_1')$] $\mathcal{D}=(D_j)_{j\in J}$ is a collection of deterministic matrices such that $\sup_N ||D^{abs}_j||<\infty$ for any $j\in J$, where  $D^{abs}=(|d_{i,j}|)_{i,j=1}^N$ is the matrix of absolute values of entries.
\end{enumerate}
Notice that in case $D_j$'s are diagonal, $\mathbf{C}_1$ and $\mathbf{C}_1'$ are equivalent. 

The following results provide an upper bound for the cumulants. Here we do not require the polynomial to be self-adjoint but instead only  condition $\mathbf{C}_1$ is needed. Thus, our results provide an upper bound for $\mathbf{X}$ which might be a complex-random variable. Recall the parameters defined in (\ref{Definition: M}), (\ref{Definition: K}) and those in Definition \ref{Definition: Wigner Matrix}  and condition $\mathbf{C}_2$.  We emphasize that, throughout the paper, all constants are conventionally independent of  $N$.

\begin{theorem}\label{Theorem: Upper bound of cumulants GUE}
Let $\mathcal{X}\sim $ GUE or GOE. Consider the polynomial $P(\mathcal{X},\mathcal{D})$ in (\ref{def of P}) and the random variable $\mathbf{X}$ defined in (\ref{object}). Suppose that $\mathbf{C}_1$ is satisfied. 
Then there exists a constant $\theta_1\equiv \theta_1 (K, M, t)\in (0, 1)$ such that 
\begin{align}|\C_r(\mathbf{X})|\leq \frac{r!^{\frac{M}{2}}}{(\theta_1 N)^{r-2}}, \label{Bound of GUE}
\end{align}
for any $N$ and $r\geq 3$. 
\end{theorem}

\begin{remark} The optimality of the $N^{2-r}$ factors has been demonstrated in the pure GUE case in  \cite{MSpart1}, when $r$ is fixed. Further, the factor $r!^{\frac{M}{2}}$ would implies a tail of the form $\exp(-c_N x^{\frac{2}{M}})$ for the random variable $\mathbf{X}$, according to Lemma 2.8 of \cite{DJ}. The factor $x^{\frac{2}{M}}$ in the exponent is essential as we are considering a polynomial of degree $M$. A trivial example would be a single GUE matrix $X$, and $\text{Tr} X^2=\sum_{ij} x_{ij}^2$ is considered. Apparently, in this case, $\mathbf{X}$ should be subexponential. 
\end{remark}

\begin{theorem}\label{Theorem: Upper bound of cumulants Wigner}
Let $\mathcal{X}\sim Wigner$.  Consider the polynomial $P(\mathcal{X},\mathcal{D})$ in (\ref{def of P}) and the random variable $\mathbf{X}$ defined in (\ref{object}).  Suppose that $\mathbf{C}_1$ is satisfied. 
Then there exists a constant $\theta_2 \equiv \theta_2 (K, M, \mathsf{C}, t) \in (0, 1)$ such that 
\begin{align}|\C_r(\mathbf{X})|\leq \frac{r!^{3M}}{(\theta_2 \sqrt{N})^{r-2}}
\label{Bound og Wigner}
\end{align}
for any $N$ and $r\geq 3$. If $\mathbf{C}_1'$ is additionally assumed, we get the sharper bound
\begin{align}|\C_r(\mathbf{X})|\leq \frac{r!^{M}}{(\theta_3 \sqrt{N})^{r-2}}.
\label{Bound og Wigner Special case}
\end{align}
with some constant $\theta_3\equiv \theta_3(K,M,\mathsf{C}, t) \in (0,1)$. 
\end{theorem}
\begin{remark} In the case when $\mathbf{C}_1'$ is assumed, instead of $r!^{\frac{M}{2}}$ in the GUE/GOE case, we have $r!^{M}$ here, due to the fact that we are considering general subexponential entries instead of sub-Gaussian ones. Further, the $(\sqrt{N})^{2-r}$ is also essential in general, as we will explain in Example \ref{Example: Order of Wigner case}.  
\end{remark}

If $\mathbf{C}_1-\mathbf{C}_3$ are all assumed, as the variance of $\mathbf{X}$ has a lower bound, the above cumulant bounds still apply to the normalized random  variable $\mathfrak{X}$, by slightly changing the parameters $\theta_1, \theta_2$ and $\theta_3$. Especially, it shows that for any given $r\geq 3$, the $r$-th cumulant of $\mathfrak{X}$ vanishes asymptotically as $N\to \infty$. Hence, we automatically recover the CLTs of $\mathfrak{X}$ in all cases. In addition, as our $r$ could even be  $N$-dependent, we can take a step further to derive the quantitative versions of these CLTs, in light of \cite{DJ}. For the following results we now assume $\mathbf{C}_1-\mathbf{C}_3$. We first recall the Kolmogorov distance of two real random variables $X$ and $Y$,
$$\Delta(X,Y)=\sup_{x\in \mathbb{R}}|P(X\leq x)-P(Y\leq x)|.$$
For the following theorem we may assume $M>1$ as the case $M=1$ is trivial. Let $Z$ be a $\mathcal{N}(0,1)$ variable in the sequel.

\begin{theorem}\label{Theorem: Berry-Esseen bound GUE/GOE Main result 1}
Let $\mathcal{X}\sim GUE$ or GOE. Consider the polynomial $P(\mathcal{X},\mathcal{D})$ in (\ref{def of P}) and the random variable $\mathfrak{X}$ defined in (\ref{object_normalized}). Suppose that $\mathbf{C}_1-\mathbf{C}_3$ are satisfied. Then we have the following three conclusions

(i) ({\it CLT with Cram\'{e}r correction}) For  any $x\in [0,\tilde{\theta}_1 N^{1/(M-1)})$ with sufficiently small constant $\tilde{\theta}_1\equiv \tilde{\theta}_1(K, M, \mathfrak{c},t)>0$,
we have 
\begin{align}
\mathbb{P}(\mathfrak{X}\geq x) = e^{{L(x)}}\mathbb{P}(Z\geq x)\left( 1+ O\left( \frac{x+1}{N^{1/(M-1)}}\right) \right).
\label{Cramer correction}
\end{align}
Here ${L(x)}$ satisfies $$|{L(x)}|= O(x^3/N^{1/(M-1)}). $$

(ii) ({\it Berry-Esseen bound}) There exists a constant $C_{1}\equiv C_1(K,M,\mathfrak{c},t)>0$, such that
$$\Delta(\mathfrak{X},Z)\leq \frac{C_{1}}{N^{1/(M-1)}}.$$

(iii) ({\it Concentration}) There exist constants $C_2\equiv C_2(K, M, \mathfrak{c},t)>0$, $\hat{\theta}_1\equiv \hat{\theta}_1(K,M,\mathfrak{c},t)>0$,  and a constant $A>0$ that does not depend on $K, M, \mathfrak{c}, t$, such that
\begin{align}
\mathbb{P}(\mathfrak{X}\geq x) \leq A \exp\left( -\frac{1}{2}\frac{x^2}{C_2+x^{2-\frac{2}{M}}(\hat{\theta}_1N)^{-\frac{2}{M}}}\right)
\label{Concentration inequality}
\end{align}
for all $x\geq 0$.
\end{theorem}

\begin{theorem}\label{Theorem: Berry-Esseen bound Wigner}
Let $\mathcal{X}\sim Wigner$. Consider the polynomial $P(\mathcal{X},\mathcal{D})$ in (\ref{def of P}) and the random variable $\mathfrak{X}$ defined in (\ref{object_normalized}). Suppose that $\mathbf{C}_1-\mathbf{C}_3$ are satisfied. Then we have the following three conclusions.

(i)({\it CLT with Cram\'{e}r correction}) For any $x\in [0,\tilde{\theta}_2 N^{1/(6M-1)})$ for sufficiently small constant $\tilde{\theta}_2\equiv \tilde{\theta}_2(K, M, \mathsf{C}, \mathfrak{c},t)>0$,
we have 
\begin{align}
\mathbb{P}(\mathfrak{X}\geq x) = e^{{L(x)}}\mathbb{P}(Z\geq x)\left( 1+ O\left( \frac{x+1}{(\sqrt{N})^{1/(6M-1)}}\right) \right).
\label{Cramer correction}
\end{align}
Here ${L(x)}$  satisfies $$|{L(x)}|= O(x^3/(\sqrt{N})^{1/(6M-1)}). $$ 
If $\mathbf{C}_1'$ is additionally assumed, the parameter $6M$ can be improved to $2M$ in all above formulas.

(ii) ({\it Berry-Esseen bound}) There exist a constant $C_1 \equiv C_1(K,M,\mathsf{C}, \mathfrak{c},t)$ such that,
$$\Delta(\mathfrak{X},Z)\leq \frac{C_{1}}{(\sqrt{N})^{1/(6M-1)}}.$$
If $\mathbf{C}_1'$ is additionally assumed, the parameter $6M$ can be improved to $2M$ in the above formula. 

(iii) ({\it Concentration}) There exist constants $C_2\equiv C_2(K, M, \mathsf{C}, \mathfrak{c}, t)>0$, $\hat{\theta}_2\equiv \hat{\theta}_2(K,M,\mathsf{C}, \mathfrak{c},t)>0$ and a constant $A>0$ that does not depend on $K, M, \mathsf{C}, \mathfrak{c}, t$ such that,
\begin{align}
\mathbb{P}(\mathfrak{X}\geq x) \leq A \exp\left( -\frac{1}{2}\frac{x^2}{C_2+x^{2-\frac{1}{3M}}(\hat{\theta}_2\sqrt{N})^{-\frac{1}{3M}}}\right)
\label{Concentration inequality}
\end{align}
for all $x\geq 0$. If $\mathbf{C}_1'$ is additionally assumed, the parameter $3M$ can be improved to $M$ in the above formula. 
\end{theorem}

\begin{remark} We notice that the Berry-Esseen bound can also be viewed as a simple consequence of the CLT with Cram\'{e}r correction. But we also single  it out as it is a widely used distance for quantitative CLT. It is known that the cumulant bounds generally do not provide optimal Berry-Esseen bounds, even if the cumulant bounds themselves are optimal; see \cite{DJ} for instance. However, it is noteworthy that such Berry-Esseen bounds have only been obtained for a single random matrix model in the previous literature. Furthermore, the true power of the cumulant bounds can be seen in the CLTs with Cram\'{e}r corrections and the concentration inequalities, which not only capture the Gaussian nature in the small deviation regime, but also the tail of the LES in the moderate and large deviation regimes. Such quantitative characterization of the law of 
$\mathfrak{X}$ is not available from a Berry-Esseen type bound.  

We would like to further compare our concentration result with the existing result on a single Wigner matrix $X$ whose entries satisfy log-Sobolev inequality in \cite{AW15}. In \cite{AW15}, a nearly optimal concentration inequality was obtained for $\text{Tr} f(X)$ when $f$ is not necessarily Lipschitz. Especially, Proposition 3.7 of \cite{AW15} implies that if $f(x)$ is a quadratic polynomial, the concentration tail  of $\text{Tr} f(X)$ is $A\exp(-\delta x^2\wedge Nx)$ for some constant $\delta>0$, which exactly matches our bound in (\ref{Concentration inequality}) when $M=2$. When $M>3$,   in \cite{AW15} it is claimed that a concentration tail of the form $A\exp (-\delta x^2\wedge Nx^{\frac{2}{M}})$ should be valid, while our bound in (\ref{Concentration inequality}) reads $A\exp (-\delta x^2\wedge N^{\frac2M}x^{\frac{2}{M}})$, and thus is weaker in terms of the $N$-factor, but our result is for arbitrary polynomial in random matrices.  We emphasize that similarly to the case of Berry-Esseen bound,  this suboptimality is due to the nature of the cumulant method in \cite{DJ},  even if the cumulant bounds are optimal. For the general Wigner case,  when $\mathbf{C}_1'$ is assumed, the concentration bound in Theorem \ref{Theorem: Berry-Esseen bound Wigner} essentially reads $A\exp(-\delta x^2\wedge N^{\frac{1}{2M}} x^{\frac{1}{M}})$. We believe that the factor $x^{\frac{1}{M}}$ is optimal as now we are considering subexponential entries rather than sub-Gaussian ones as in \cite{AW15}.  
\end{remark}

\begin{remark} Finally, we remark that the constant-order lower bound on $\text{Var}(\text{Tr}(P(\mathcal{X},\mathcal{D}))$ in condition $\mathbf{C}_2$ can be weakened to a lower bound of the form $\mathfrak{c}\equiv \mathfrak{c}_N=N^{-s}$ for some sufficiently small $s>0$. With the cumulant bounds for $\mathbf{X}$ in Theorems \ref{Theorem: Upper bound of cumulants GUE} and \ref{Theorem: Upper bound of cumulants Wigner}, one can easily verify that the mean-0, variance-1 normalized random variable $\mathfrak{X}$ still has negligible high-order cumulants when $r$ is fixed, as long as $s$ is sufficiently small. Hence, a CLT or quantitative CLT can still be derived. In this case, the convergence rates in the quantitative laws will also require modification. We leave such an extension to interested readers. 
\end{remark}

\subsection{Proof Strategy}\label{Subsection: Proof strategy}

We approach first the GUE/GOE case for which we study first the quantities
\begin{equation}\label{Aux: Initial quantities proof strategy}
\C_r(Tr(X_{i_1}D_{j_1}\cdots X_{i_{m_1}}D_{j_{m_1}}),\dots, Tr(X_{i_{m_1+\cdots+m_{r-1}+1}}D_{i_{m_1+\cdots+m_{r-1}+1}}\cdots X_{i_{m_1+\cdots+m_{r}}}D_{j_{m_1+\cdots+m_{r}}})),
\end{equation}
for an arbitrary choice of $r$ and deterministic and random matrices. The only assumption that we impose so far is $\mathbf{C}_1$. We will first prove that the last quantity can be rewritten as a double sum over partitions,
\begin{equation}\label{Equation: sumation over double set}
N^{-m/2}\sum_{\pi\in \cP(\pm[m])}\sum_{\substack{\tau\in \cP(m) \\ \tau\vee\gamma=1_m}}\left[\sum_{\substack{\psi: \pm[m]\rightarrow [N] \\ ker(\psi)=\pi}}\mathbf{D}(\psi)\right]\C_\tau(x_{\psi^\pi(1),\psi^\pi(-1)}^{(i_1)},\dots,x_{\psi^\pi(m),\psi^\pi(-m)}^{(i_m)}),
\end{equation}
where $m=\sum_i m_i$, $\pm[m]=\{-1,1,\dots,-m,m\}$, $[m]=\{1,\dots,m\}$ and $ker(\psi)$ is the partition of $\cP(\pm [m])$ determined by $u$ and $v$ being in the same block whenever $\psi(u)=\psi(v)$. Here $1_m$ is the partition with a single block and $\tau\vee\gamma$ is the smallest partition greater than or equal to $\tau$ and $\gamma$ with respect to the partial order $\leq$ defined in Section \ref{Section_Graph theory and combinatorics}. With this expression the  $\C_\tau(\cdot)$ term is the mixed cumulant of the random variables and the sum of $\mathbf{D}(\psi)$ terms determines the $N$-order of each summand in the sum over partitions. In the GUE case, thanks to the fact  that the cumulants vanish except for the second cumulant $\C_2(x_{i,j},x_{j,i})$, we are able to rewrite (\ref{Equation: sumation over double set}) as
\begin{equation}\label{Equation: sumation over single set}
N^{-m/2}\sum_{\substack{\tau\in \cP_2(\pm[m]) \\ \tau\vee\gamma=1_m}}\left[\sum_{\substack{\psi: \pm[m]\rightarrow [N] \\ ker(\psi)\geq \pi_\tau}}\mathbf{D}(\psi)\right],
\end{equation}
where $\pi_\tau$ is a pairing of $\pm [m]$ determined by $\tau$. This permits us to associate to each partition $\pi_\tau$ a graph $D^{\pi_\tau}$. The graph $D^{\pi_\tau}$ determines the $N$-order from its number of two-edge connected components (see \cite{MSBounds}). In the GUE case we observe that these graphs are all disjoint cycles and the $N$-order is given by its number of cycles. 
Thus we reduce the problem to determining the number of cycles of the graphs $D^{\pi_\tau}$ for each pairing $\tau$. We prove that the number of cycles of these graphs is at most $m/2+r-2$. To achieve this, we observe that these graphs can be seen as quotient graphs (see Definition \ref{Definition: Quotient graph}) of a graph consisting of $r$ cycles. Naturally, counting the number of cycles for an arbitrary $r$ might be a tough task. However, we observe that there is a relation between the number of cycles of the quotient graph $D^{\pi_\tau}$ and a quotient graph $(D^{\prime})^{\pi_\tau^\prime}$ where $D^\prime$ is a graph consisting of $r-1$ cycles. More precisely, we prove that both graphs have the same number of cycles. Therefore, to prove that the number of cycles of $D^{\pi_\tau}$ is at most $m/2+r-2$, we first attack the case $r=1$ and then use induction to obtain the upper bound of the number of cycles for any $r$. Let us remark that in order to prove the case $r=1$, we apply another induction on $m$. For the induction on $m$, our strategy is to define a graph $D^\prime$ consisting of a cycle with $2m-2$ edges that can be constructed from the graph $D$ which consists of a cycle with $2m$ edges. Then we prove that the number of cycles of $D^{\pi_\tau}$ is less than or equal to the number of cycles of $(D^\prime)^{\pi_\tau^\prime}$ minus $1$. From this observation our induction proof follows.  Once our {\it double induction process} is completed, we finally prove that whenever the term $\C_\tau(x_{\psi^\pi(1),\psi^\pi(-1)}^{(i_1)},\dots,x_{\psi^\pi(m),\psi^\pi(-m)}^{(i_m)})$ is not zero,
\begin{equation}\label{Equation: Order of inner summation}
N^{-m/2-r+2}\sum_{\substack{\psi: \pm[m]\rightarrow [N] \\ ker(\psi)\geq \pi_\tau}}\mathbf{D}(\psi)=O(1).
\end{equation} 

From the estimate (\ref{Equation: Order of inner summation}), we are able to provide upper bounds for the quantities of the form ($\ref{Aux: Initial quantities proof strategy}$). Under the GUE setting, we go one step further, to prepare the proof for the GOE case. Again, we estimate the upper bounds for the quantities of the form (\ref{Aux: Initial quantities proof strategy}), but  now we allow $X_{i_j}$ and its transpose $X_{i_j}^{\top}$ to be both in the polynomial. We call these type of quantities the \textit{$\epsilon$-cumulants}, where $\epsilon$ is a vector indicating whether or not the transpose is taken to each GUE variable. We show that the $\epsilon$-cumulants admit the same upper bound as the regular cumulants described in (\ref{Aux: Initial quantities proof strategy}). We are then able to provide upper bounds for the quantities (\ref{Aux: Initial quantities proof strategy}) when $X_{i_j}\sim GOE$ via the observation that a GOE matrix $X$ can be written as $X= (Z+Z^{\top})/\sqrt{2}$, where $Z\sim GUE$. Then by the multi-linearity of the cumulants, getting an upper bound for (\ref{Aux: Initial quantities proof strategy}) in the GOE case can be done by getting the corresponding upper bounds of the $\epsilon$-cumulants
\begin{equation}\label{Aux: Initial quantities proof strategy epsilon case}
\C_r(Tr(X_{i_1}^{\epsilon(1)}D_{j_1}\cdots X_{i_{m_1}}^{\epsilon(m_1)}D_{j_{m_1}}),\dots, Tr(X_{i_{m_1+\cdots+m_{r-1}+1}}^{\epsilon(m_1+\cdots +m_{r-1}+1)}D_{i_{m_1+\cdots+m_{r-1}+1}}\cdots X_{i_{m_1+\cdots+m_{r}}}^{\epsilon(m_1+\cdots +m_r)}D_{j_{m_1+\cdots+m_{r}}})),
\end{equation}
where $\mathcal{X}\sim GUE$ and $\epsilon_i\in \{1,{\top}\}$ describes whether we consider $X_i$ or $X_i^{\top}$.

Further, thanks to the multi-linearity of the cumulants and the vanishing of mixed cumulants property (\cite[Theorem 11.32]{NS}),  our bounds can be extended to any polynomial in $\mathcal{X}$ and $\mathcal{D}$, and thus the random variable $\mathbf{X}$.

When we turn to the general Wigner case, unlike the Gaussian case, the cumulants admit  distinct $N$-orders. We illustrate this fact with Example \ref{Example: Order of Wigner case}. We observe that the order of the quantity (\ref{Aux: Initial quantities proof strategy}) for a polynomial in general Wigner and deterministic matrices is $N^{1-r/2}$ in general. Due to the nature of the cumulants $\C_\tau(\cdot)$ that appear in (\ref{Equation: sumation over double set}) in the Wigner case, we are no longer able to simplify (\ref{Equation: sumation over double set}) to the form (\ref{Equation: sumation over single set}). Hence we turn our attention back to (\ref{Equation: sumation over double set}) and provide an upper  bound for the quantity
$$\left|\sum_{\substack{\psi: \pm [m]\rightarrow [N] \\ ker(\psi)=\pi}}\mathbf{D}(\psi)\right|$$
in terms of the quantity of the form (\ref{Equation: sumation over single set}) using the M\"{o}bius inversion theorem. As in the Gaussian case, the bound of the above quantity is ultimately determined by the number of two-edge connected components of the graph $D^\pi$. Our major idea to handle the general Wigner case is to reduce the more involved combinatoric structures into those simple structures existing in the  GUE case. Specifically, our aim is to prove that the number of two-edge connected components of the graph $D^{\pi}$ is at most $m/2+1-r/2$. Due to the nature of the cumulants of the Wigner case, the partitions $\pi$ are way more involved than those of the Gaussian case, and hence the graph $D^{\pi}$ has a more complicated structure. However, we prove that the graph $D^{\pi}$ can be regarded as the union of graphs associated with the Gaussian case. Here by union we specifically mean that the graphs are merged along their edges (by identifying the two pair of vertices of the edges). Therefore the number of two edge-connected components of the resulting graph can be deduced from those of the individual graphs associated with the Gaussian case. Thanks to the previous argument and results for the Gaussian case, we are then allowed to obtain an upper bound for the number of two-edge connected components of the graphs associated with the Wigner case. In order to formally define the meaning of two graphs being merged along their edges we introduce a new concept called  {\it quotient graphs induced by graphs} (see Definition \ref{Definition: Quotient graph induces by graph}). Roughly speaking, we introduce a new graph, $T$. This graph will encode the information of how the graphs are merged along their edges. This new notion permits us to describe the topological properties of the resulting graph in terms of the topological properties of the individual ones via the properties of the encoding graph $T$. A more precise definition with examples and proofs can be found in Section \ref{SubSection: Graphs induced by graphs}. We are then  able to provide upper bounds for any polynomial of Wigner and deterministic matrices under condition $\mathbf{C}_1$.

After establishing the upper bounds of $K_r(\mathbf{X})$ for all $r$ in Theorems \ref{Theorem: Upper bound of cumulants GUE}-\ref{Theorem: Upper bound of cumulants Wigner}, we then apply them to derive the quantitative laws in Theorems \ref{Theorem: Berry-Esseen bound GUE/GOE Main result 1}-\ref{Theorem: Berry-Esseen bound Wigner}, by the method of cumulants. We refer to \cite{DJ} for a rather comprehensive survey of this approach to the quantitative laws. Actually, all the bounds in Theorems \ref{Theorem: Upper bound of cumulants GUE}-\ref{Theorem: Upper bound of cumulants Wigner} can be viewed as the so-called Statulevi$\check{c}$ius type bound for cumulants. The quantitative laws in Theorems \ref{Theorem: Berry-Esseen bound GUE/GOE Main result 1}-\ref{Theorem: Berry-Esseen bound Wigner} then follow by using the conclusions in \cite{DJ}.


\subsection{Organization}

The rest of the paper is organized as follows. In Section \ref{Section_Graph theory and combinatorics}, we introduce the necessary notions and derive preliminary lemmas for the combinatorial tools. In Section \ref{Section_Genus expansion of the cumulants}, we introduce the general formula for the expansion of the cumulants.  In Section \ref{Upper bounds for t}, we derive the bounds for a function $t(\cdot)$ which will be the key input for the bounds of the mixed cumulants of the traces of matrix monomial proved in  Section \ref{Section: The order of GUE}.  In Section \ref{Section: The Statulevicius condition}, we prove our main theorems.

\section{Quotient graphs and partitions}\label{Section_Graph theory and combinatorics}

In this section, let us introduce some notions on combinatorics that will be used for the proof of our main results. First we introduce the set of partitions. Second, we recall some general graph theory. Inspired by \cite{MSBounds},   we then turn our attention to the concept of quotient graphs, which has also been used in some later works such as  \cite{M, MMPS}.  Finally in Subsection \ref{SubSection: Graphs induced by graphs},  we introduce a new concept called  {\it quotient graphs induced by graphs} that will be used to prove our main results in the Wigner case.

\subsection{The set of partitions}

By a partition of a set $A$, we mean a collection of disjoint subsets of $A$ whose union is $A$. We call the subsets the \textit{blocks} of the partition. We denote the set of partitions of the set $A$ by $\cP(A)$. If $B$ is a block of $\pi\in \cP(A)$, we use the notation $B\in \pi$. We denote the number of blocks of a partition $\pi$ by $\#(\pi)$. When all blocks of the partition have size $2$, we call $\pi$ a pairing. The set of pairings will be denoted by $\cP_2(A)$. For a partition $\pi\in\cP(A)$ and an element $v\in A$, we denote by $[v]_{\pi}$ the block of $\pi$ that contains $v$.

We consider the partial order $\leq$ in $\cP(A)$ given by $\pi\leq \sigma$ if any block of $\pi$ is contained in a block of $\sigma$. Equipped with this partial order $\cP(A)$ becomes a POSET (partially ordered set).  With this partial order, the largest element, denoted by $1_A$, is the partition that consists of a single block.  And the smallest element, denoted by $0_A$, is the partition that consists of $|A|$ blocks of singleton.

The \textit{join} of two partitions $\pi$ and $\sigma$, denoted by $\pi\vee\sigma$, is the smallest partition that is greater than or equal to $\pi$ and $\sigma$. When the set $A=[m]\vcentcolon =\{1,\dots,m\}$,  we simply use the notation $\cP(m)$ to denote $\cP([m])$.

The following proposition will be used in the derivation of the upper bound for the cumulants of $\mathbf{X}$. Despite the fact that it will be needed until Section \ref{s. Wigner case}, we prefer to state and prove it in this section, as it constitutes a purely theoretical result on partitions.

\begin{proposition}\label{Proposition: the existence of the pairing sigma}
Let $r\geq 2$ and $m_1,\dots,m_r\in\mathbb{N}$. Let $m=\sum_{i=1}^r m_i$. Let $\gamma\in \cP(m)$ be the partition whose blocks are
\begin{align*}\{m_1+\cdots+m_{j-1}+1,\dots,m_1+\cdots+m_j\},\qquad 1\leq j\leq r
\end{align*}
with the convention $m_0=0$. Let $\tau\in \cP(m)$ be a partition such that each block of $\tau$ has even size and $\gamma\vee\tau=1_m$. Then there exists a pairing $\sigma\leq \tau$ such that $\#(\sigma\vee\gamma)\leq r/2$.
\end{proposition}
\begin{proof}
We use induction on $r$. The case $r=2$ follows from the fact that $\tau$ must have a block that contains an element from $\{1,\dots,m_1\}$ and an element from $\{m_1+1,\dots,m_1+m_2\}$. Hence, we can choose the pairing $\sigma$ so that it contains a block with one element from each of these sets. It proves $\gamma\vee\sigma=1_m$. Let us assume that the statement is true for any integer less than or equal to $r$ and then  prove it for $r+1$. If any block of $\tau$ has size $2$ then $\sigma=\tau$ satisfies the requirement. So let $B$ be a block of $\tau$ of size at least $4$. Let $\{u,v\}\subset B$ and let $\tau^\prime$ be the partition whose blocks are $B\setminus \{u,v\}, \{u,v\}$ and any block of $\tau$ distinct from $B$. If $\{u,v\}$ and $B\setminus \{u,v\}$ are contained in the same block of $\tau^\prime\vee\gamma$ then $\tau^\prime\vee\gamma=\tau\vee\gamma=1_m$ otherwise $\tau^\prime\vee\gamma$ has two blocks, $C$ and $D$, such that $\{u,v\}\subset C$ and $B\setminus \{u,v\}\subset D$. If $\tau^\prime\vee\gamma=1_m$ and any block of $\tau^\prime$ has size $2$ then we are done, otherwise we repeat the same process until we either have $\#(\tau^\prime\vee\gamma)=2$ or $\tau^\prime$ has only blocks of size $2$. So we may assume $\#(\tau^\prime\vee\gamma)=2$. Let $\tau^\prime_C$ and $\gamma_C$ be the partitions whose blocks are the blocks of $\tau^\prime$ and $\gamma$ contained in $C$ respectively. We similarly define $\tau^\prime_D$ and $\gamma_D$. If $\#(\gamma_C)=\#(\gamma_D)=1$ then $r=2$ which we already proved. So we may assume at least one is greater than $1$. We are reduced to three possible scenarios. Either $\#(\gamma_C)>1$ and $\#(\gamma_D)>1$, or $\#(\gamma_C)=1$ and $\#(\gamma_D)>1$, or $\#(\gamma_C)>1$ and $\#(\gamma_D)=1$. Moreover in either case $\tau^\prime_C\vee\gamma_C=1_C$ and $\tau^\prime_D\vee\gamma_D=1_D$. In the first scenario by induction hypothesis there exist pairings $\sigma_C$ and $\sigma_D$ such that $\sigma_C\leq \tau^\prime_C$, $\sigma_D\leq \tau^\prime_D$ and $\#(\sigma_C\vee\gamma_C)\leq \#(\gamma_C)/2$ and $\#(\sigma_D\vee\gamma_D)\leq \#(\gamma_D)/2$. If $\sigma$ denotes the pairing whose blocks are the blocks of $\sigma_C$ and the blocks of $\sigma_D$ then $\sigma\leq \tau^\prime\leq \tau$ and
$$\#(\sigma\vee\gamma)=\#(\sigma_C\vee\gamma_C)+\#(\sigma_D\vee\gamma_D)\leq \#(\gamma_C)/2+\#(\gamma_D)/2=(r+1)/2.$$
In the second scenario $\gamma_C$ consists of a single block of $\gamma$ and hence every block $\tau^\prime_C$ is contained in a single block of $\gamma$. Let $\sigma_C$ be a pairing such that $\sigma_C\leq \tau^\prime_C$. Thus, all blocks of $\sigma_C$ are contained in the same block of $\gamma$,  in particular $\{u,v\}$. On the other hand, it is clear that $\#(\gamma_D)= r$, so, by induction hypothesis there exists $\sigma_D$ such that $\sigma_D\leq \tau^\prime_D$ and $\#(\sigma_D\vee\gamma_D)\leq r/2$. Let $\{a,b\}\in \sigma_D$ be a block such that $\{a,b\}\subset B\setminus \{u,v\}$. Let $\sigma$ be the pairing whose blocks are the blocks of $\sigma_C$ and the blocks of $\sigma_D$ except $\{u,v\}$ and $\{a,b\}$ for which instead we consider the blocks $\{u,a\}$ and $\{v,b\}$. In this way we have that $\#(\sigma\vee\gamma)=\#(\sigma_D\vee\gamma_D)\leq r/2<(r+1)/2$. Moreover, in either case $\sigma\leq \tau$ because $\{u,a\}$ and $\{v,b\}$ are contained in the block $B$ of $\tau$. Hence, the proof of this case is complete. The last scenario follows exactly as the second one.
\end{proof}

\subsection{Oriented, unoriented and quotient graphs}

By an { unoriented graph} or simply { graph} we mean a pair $(V,E)$ where $V$ is the set of vertices and $E\subset V\times V$ the set of edges where $\{u,v\}\in E$ means that there is an edge connecting $u$ and $v$. We allow multiple edges and loops so that the same pair of vertices might have more than one edge connecting them. An oriented graph is a pair $(V,E)$ where $V$ is the set of vertices and $E\subset V\times V$ is the set of edges where now the order of the vertices matters. We adopt the notation $(u,v)\in E$ which means there is an edge going from $u$ to $v$ but not necessarily there is an edge going from $v$ to $u$. We call $u$ the source of the edge $e$ which we denote by $src(e)$ and $v$ the target of $e$ which we denote by $trg(e)$. We show examples of these basic concepts in Example \ref{Example: Oriented and unoriented graphs}, which will also be used for the introduction of further concepts.

Given two vertices in an unoriented graph $u,v\in V$, a path from $u$ to $v$ is a collection of vertices $u_1,\dots,u_n$ such that $u=u_1,v=u_n$ and there is an edge $e_i=\{u_i,u_{i+1}\}$ of the graph for any $1\leq i\leq n-1$. For a path in our context we will consider all of the vertices $u_1,\dots,u_n$ distinct, as otherwise we will always consider another path whose vertices are all different by deleting the portion of the path in between the first arrival to a vertex and the second arrival to the same vertex. For brevity we use the notation
$$u_1 \overset{e_1}{-} u_2\overset{e_2}{-} u_3 \cdots u_{n-1}\overset{e_{n-1}}{-}u_n$$
to denote the path that connects $u_1$ with $u_n$ along the edges $e_i=\{u_i,u_{i+1}\}, 1\leq i\leq n-1$.

\begin{definition}\label{Definition: Quotient graph}
For an oriented graph $G=(V,E)$, we let $\underline{G}$ be the graph consisting of the same set of vertices and edges but forgetting the orientation. That is, if $(a,b)\in E$, then $\{a,b\}$ is an edge of $\underline{G}$. We call $\underline{G}$ the \textit{unorientation} of $G$. As the sets of vertices and edges remain the same,  we still use the notation $\underline{G}=(V,E)$. See Example \ref{Example: Oriented and unoriented graphs} for an example of an oriented graph and its unorientation.
\end{definition}

\begin{example}\label{Example: Oriented and unoriented graphs}
In Figures \ref{fig:image1} and \ref{fig:image2}, we show an example of an oriented graph $G$ and its unorientation. We use arrows to represent the oriented edges. Notice that, in the unorientation, we keep the multiplicity of the edges.
\end{example}

\begin{figure}[h!]
    \centering
    \begin{minipage}{0.48\textwidth}
        \centering
        \includegraphics[width=\linewidth]{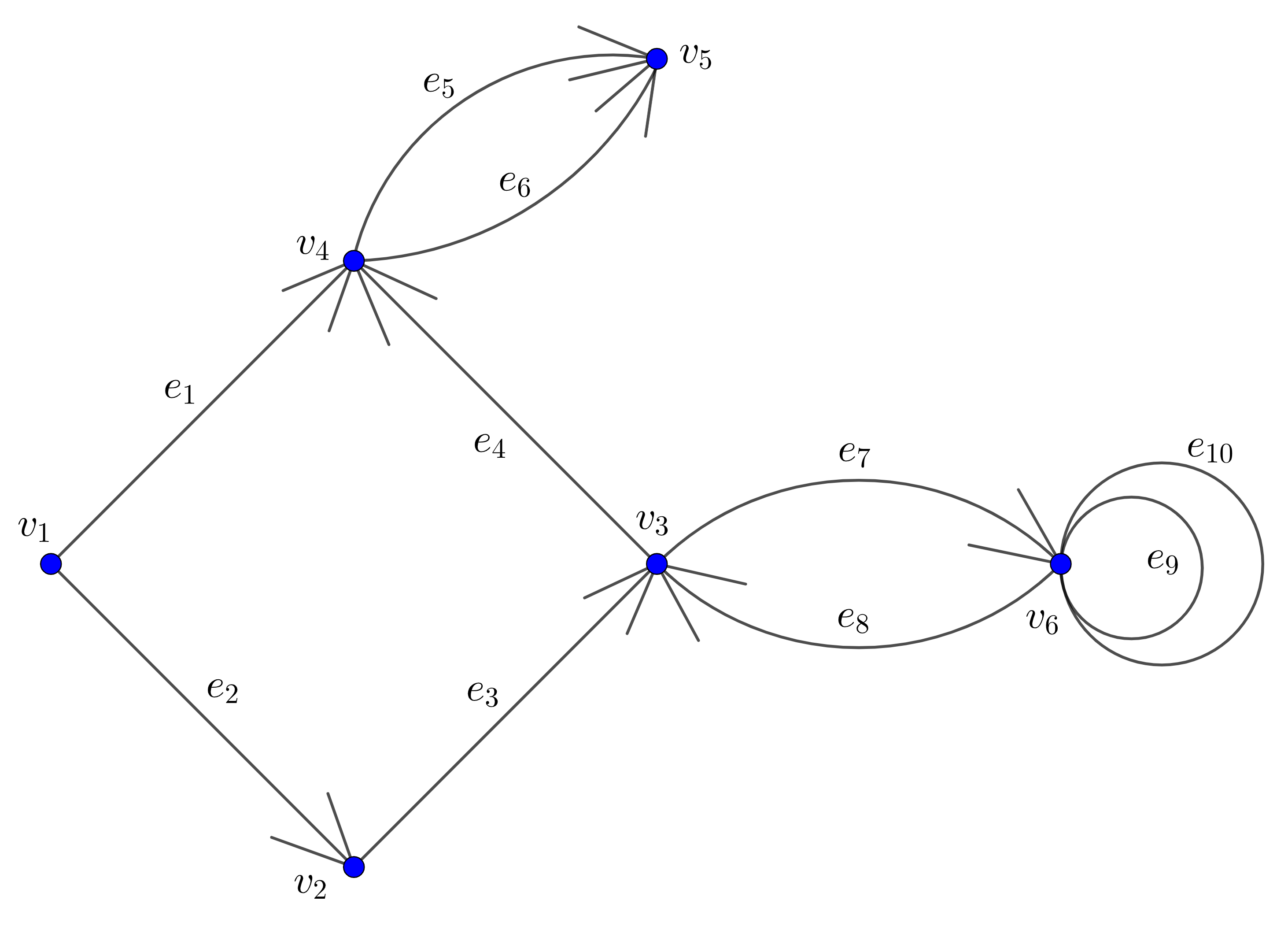} 
        \caption{\text{ Graph $G$}}
         \label{fig:image1}
    \end{minipage}
    \hfill 
    \begin{minipage}{0.48\textwidth}
        \centering
        \includegraphics[width=\linewidth]{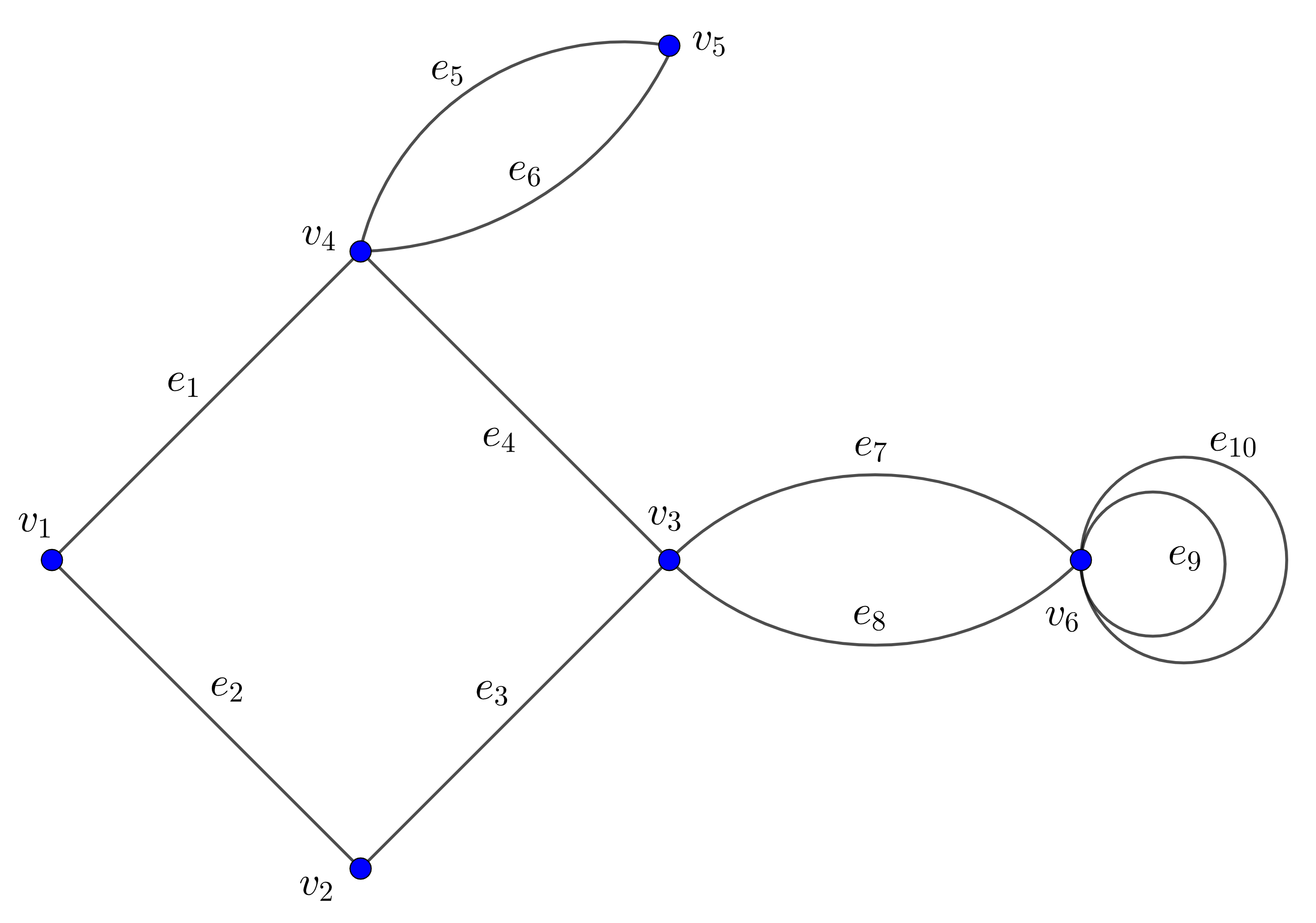}
        \caption{\text{ Unorientation of $G$}}
         \label{fig:image2}
    \end{minipage}
\end{figure}


\begin{definition}
Given an oriented (respectively unoriented) graph $G=(V,E)$ and a partition $\pi\in \cP(V)$. We let $G^{\pi}=(V^{\pi},E^{\pi})$  be the oriented (respectively unoriented) graph obtained after identifying vertices of $G$ into the same block of $\pi$. The vertices of $G^{\pi}$ are the blocks of $V$ while its edges are given by $([u]_{\pi},[v]_{\pi})$(respectively $\{[u]_{\pi},[v]_{\pi}\}$) whenever $(u,v)$(respectively $\{u,v\}$) is an edge of $G$. If $e=(u,v)$ is an edge of $G$, we denote by $e^\pi=([u]_\pi,[v]_\pi)$  the corresponding edge of $G^\pi$. We call $G^{\pi}$ the \textit{quotient graph} of the graph $G$ with respect to $\pi$. See Example \ref{Example: Quotient graph} for an example of a quotient graph.
\end{definition}

\begin{example}\label{Example: Quotient graph}
Let $\underline{G}=(V,E)$ be the unoriented graph in Example \ref{Example: Oriented and unoriented graphs}. Let $\pi\in \cP(V)$ be the partition,
$$\{v_1\},\{v_2,v_4\},\{v_3\},\{v_5,v_6\}.$$
The quotient graph $\underline{G}^{\pi}$ is the unoriented graph with set of vertices $\{v_1\},\{v_2,v_4\},\{v_3\},\{v_5,v_6\}$ and edges $e_1^{\pi}=\{\{v_1\},\{v_2,v_4\}\}, e_2^{\pi} =\{\{v_1\},\{v_2,v_4\}\}, e_3^{\pi}=\{\{v_2,v_4\},\{v_3\}\}, e_4^{\pi}=\{\{v_2,v_4\},\{v_3\}\}, e_5^{\pi}=\{\{v_2,v_4\},\ab \{v_5,v_6\}\}, e_6^{\pi}=\{\{v_2,v_4\},\{v_5,v_6\}\}, e_7^{\pi}=\{\{v_3\},\{v_5,v_6\}\}, e_8^{\pi}=\{\{v_3\}, \{v_5,v_6\}\}, e_9^{\pi}=\{\{v_5,v_6\},\{v_5,v_6\}\}$ and $e_{10}^{\pi}=\{\{v_5,v_6\},\{v_5,v_6\}\}$. This quotient graph is presented in Figure \ref{Figure: Quotient graph}. Note that, again, we keep the multiplicity of the edges. 
\end{example}

\begin{figure}
    \centering
    \includegraphics[width=0.6\textwidth]{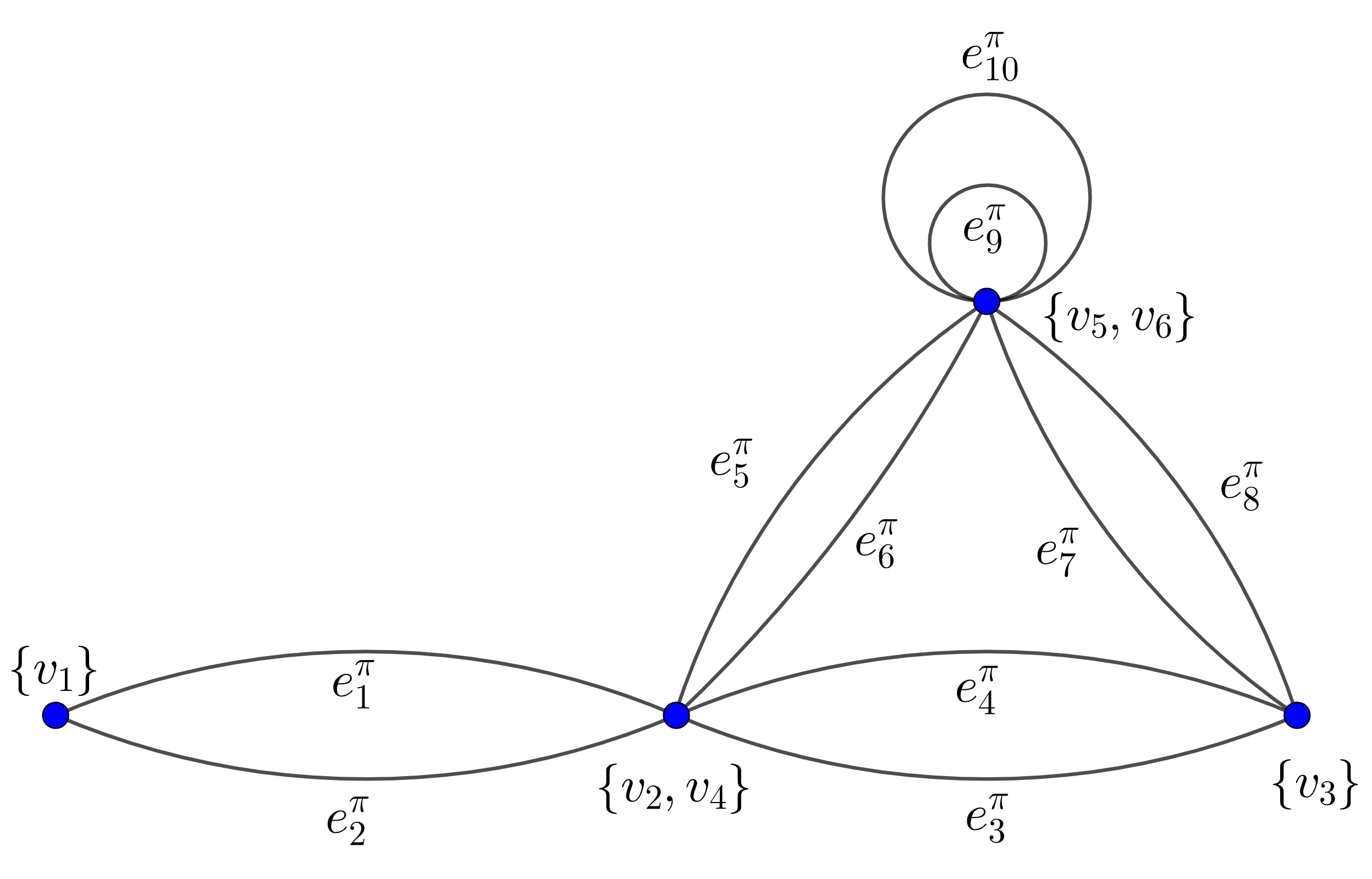}
    \caption{The quotient graph $\underline{G}^{\pi}$ of Example \ref{Example: Quotient graph}.}
    \label{Figure: Quotient graph}
\end{figure}

A \textit{cutting edge} of a graph (either oriented or unoriented) is an edge whose removal increases the number of connected components of the graph. A subgraph is \textit{two edge connected} if it has no cutting edges. A \textit{two edge connected component} of a graph is a subgraph which is two edge connected and it is not contained in a larger subgraph that  is also two edge connected. 

\begin{definition}
For a given graph $G$ (either oriented or unoriented), its graph of two edge connected components, denoted by $\mathcal{F}(G)$, is the unoriented graph whose vertices are the two edge connected components of $G$, say $U$ and $V$, and there is an edge from $U$ to $V$ whenever there is a cutting edge $\{u,v\}$ of $G$ with $u\in U$ and $v\in V$.
\end{definition}

\begin{remark}\label{Remark: The forest of two edge connected components doesnt care about orientation}
For an oriented graph $G$, observe that its cutting edges and two edge connected components are the same as the ones for its unorientation $\underline{G}$. In this sense, $\mathcal{F}(G)$ and $\mathcal{F}(\underline{G})$ are the same. So, from now on whenever we write $\mathcal{F}(G)$,  we can forget the orientation of the graph. 
\end{remark}

\begin{remark}\label{Remark: The graph of two edge connected components is a forest}
The graph $\mathcal{F}(G)$ is a forest as shrinking two edge connected components into a single vertex gets rid of any cycle. For this reason we will refer to $\mathcal{F}(G)$ as the forest of two edge connected components of $G$. 
\end{remark}

An example of a graph $G$ and its forest of two edge connected components can be seen in Figure \ref{Figure: forest of two edge connected components}.

\subsection{Sums associated with graphs}

Let $N,m\in\mathbb{N}$. For a collection of deterministic matrices $D_n=(d^{(n)}_{i,j})_{i,j=1}^N$ for $n=1,\dots,m$,  we will be interested in the $N$-order of the sums of the following form
\begin{equation}\label{Equation: Sum associated to product of deterministic matrices}
S_{\pi}(N) \vcentcolon= \sum_{\substack{j_1,\dots,j_{2m}=1 \\ ker(j)\geq \pi}}^N d_{j_1,j_2}^{(1)}\cdots d_{j_{2m-1},j_{2m}}^{(m)}.
\end{equation}
Here $\pi$ is a partition of $\cP(2m)$ and $ker(j)$ is the partition of $\cP(2m)$ given by $u$ and $v$ being in the same block whenever $j_u=j_v$. These sums appear naturally when considering the cumulants of traces of products of deterministic and permutation invariant random matrices, see for example \cite{MMPS}. The estimate of the quantity in (\ref{Equation: Sum associated to product of deterministic matrices}) has been studied before; see for example \cite[Section 2.1.2]{BS} and  \cite[Theorem 6]{MSBounds}. The method provided in \cite{MSBounds} is to associate a graph to each sum of the form in (\ref{Equation: Sum associated to product of deterministic matrices}). Let $G=(V,E)$ be the oriented graph with set of vertices $[2 m]\vcentcolon =\{1,\dots,2m\}$ and set of edges 
$e_1=(2,1),\dots,e_m=(2m,2m-1)$. 
Let $G^\pi=(V^\pi,E^\pi)$ be the quotient graph of $G$ with respect to $\pi$. We can then rewrite the sum in (\ref{Equation: Sum associated to product of deterministic matrices}) as follows,
\begin{equation}\label{Equation: Sum associated to product of deterministic matrices in terms of graph}
S_\pi(N) = \sum_{\substack{j:V^\pi\rightarrow [N]}}\prod_{i=1}^m d^{(i)}_{j(trg(e_i^\pi)),j(src(e_i^\pi))}.
\end{equation}

As proved in \cite[Theorem 6]{MSBounds}, we have the following Lemma.

\begin{lemma}\label{Lemma: Order of sums associated to deterministic matrices}
Let $\pi\in \cP(2m)$ and let $S_{\pi}(N)$ be the sum in (\ref{Equation: Sum associated to product of deterministic matrices in terms of graph}). Then,
$$|S_\pi(N)|\leq N^{t(G^\pi)}\prod_{i=1}^m ||D_i||.$$
Here  $G^\pi$ is the quotient graph of $G$ defined as before, and $t(G^\pi)$ is determined by the forest of two edge connected components $\mathcal{F}(G^\pi)$ as follows
$$t(G^\pi)=\sum_{L\text{ leaf of }\mathcal{F}(G^\pi)}t(L),$$
where
$$t(L)= \left\{ \begin{array}{lc} 1 & \text{if } L\text{ is an isolate vertex}, \\ \\ 
1/2 & \text{otherwise}. \end{array} \right.$$
\end{lemma}

\begin{example}\label{Example: order of a product of deterministic matrices} 
Let $D_1,\dots,D_{14}$ be $N\times N$ deterministic matrices with entries $D_k=(d_{i,j}^{(k)})_{i,j=1}^N$. We want to find an upper bound of the form $N^{t}\prod_{i=1}^{14} ||D_i||$, 
for the absolute value of the sum
$$S(N)=\sum_{j_1,\dots,j_{28}=1}^N d_{j_1,j_2}^{(1)}\cdots d_{j_{27},j_{28}}^{(14)},$$
subject to the constrains
\begin{align*}
j_1=j_4, \quad j_2=j_3, \quad j_5=j_8, \quad j_7=j_{19}, \quad j_{16}=j_{20}, \quad j_{10}=j_{11}, \quad j_{18}=j_{22} \\
j_6=j_{15}=j_{26}, \quad j_9=j_{13}=j_{25}, \quad j_{17}=j_{21}=j_{27}, \quad j_{12}=j_{14}=j_{23}=j_{28}. 
\end{align*}
To find the optimal choice of $t$ we consider the graph $G=(V,E)$ with set of vertices $\{1, \ldots, 28\}$ and edges $e_i=(2i,2i-1)$ for $i=1,\dots,14$. Our constraints determine the partition $\pi$ with blocks,
\begin{eqnarray*}
&&\{1,4\},\{2,3\},\{5,8\},\{6,15,26\},\{7,19\},\{16,20\},\{9,13,25\},\{10,11\} \\
&&\{12,14,23,28\},\{17,21,27\},\{18,22\},\{24\}.
\end{eqnarray*}
Hence we can rewrite the sum as,
\begin{eqnarray*}
S(N) = S_{\pi}(N)=\sum_{\substack{j_1,\dots,j_{28}=1}}^N d_{j_1,j_2}^{(1)}\cdots d_{j_{27},j_{28}}^{(14)} 
= \sum_{\substack{j:V^\pi\rightarrow [N]}}\prod_{i=1}^{14} d^{(i)}_{j(trg(e_i^\pi)),j(src(e_i^\pi))},
\end{eqnarray*}
where $G^{\pi}=(V^{\pi},E^{\pi})$ is the quotient graph of $G$ under $\pi$. This graph has two connected components; one consisting of a two edge connected component and the other containing three two edge connected components. Its forest of two edge connected components is a forest with two trees, one with a single vertex and the other one with three leaves. These graphs can be seen in Figure \ref{Figure: forest of two edge connected components}. Therefore, from Lemma \ref{Lemma: Order of sums associated to deterministic matrices}, it follows that
$$|S_{\pi}(N)|\leq N^{5/2}\prod_{i=1}^{14}||D_i||.$$
\end{example}

\begin{figure}
    \centering
    \includegraphics[width=0.7\textwidth]{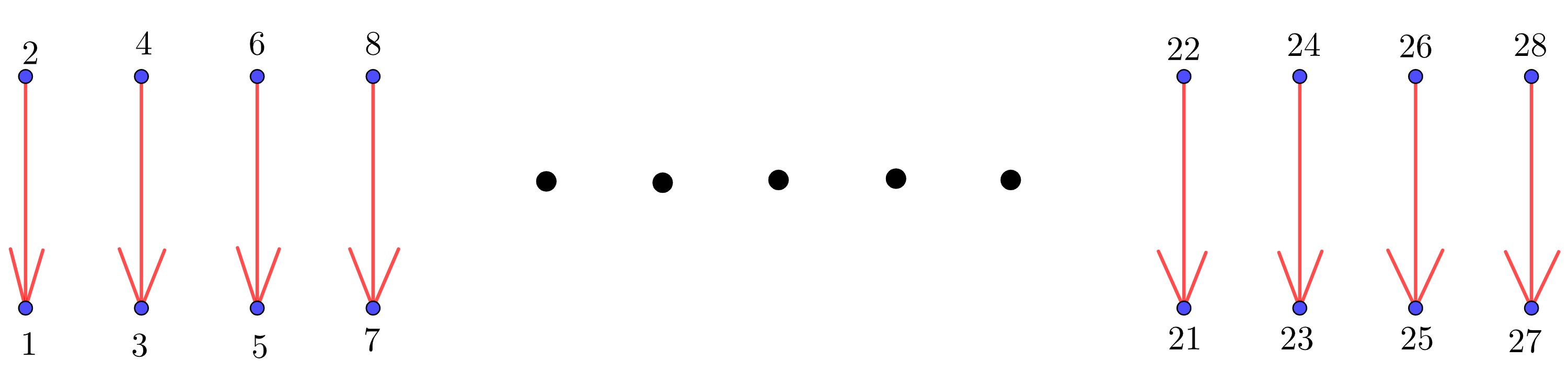}\\
    \text{a) The graph $G$.}\\
    \includegraphics[width=0.7\textwidth]{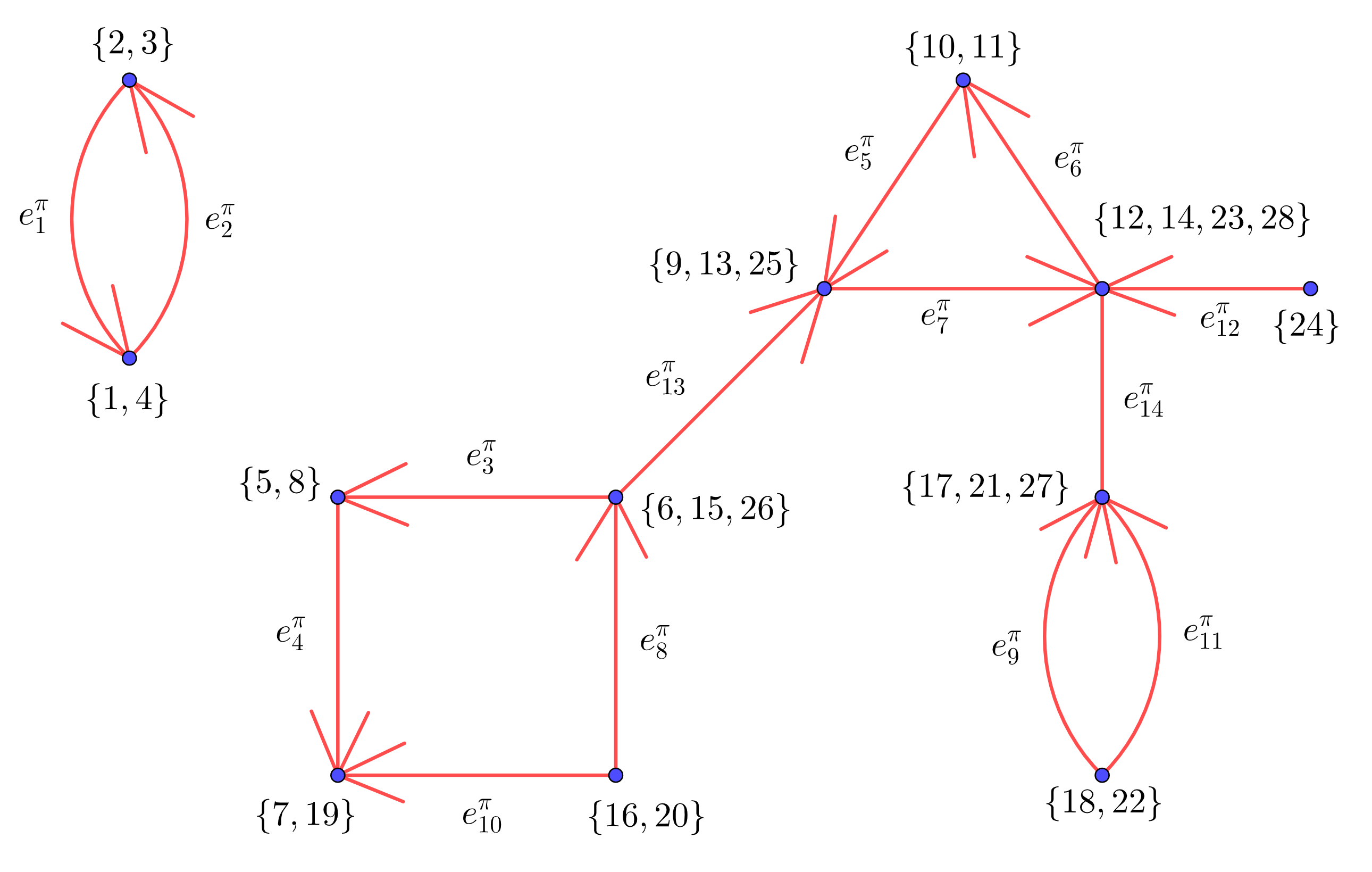}\\
    \text{b) The quotient graph $G^\pi$.} \\
    \includegraphics[width=0.6\textwidth]{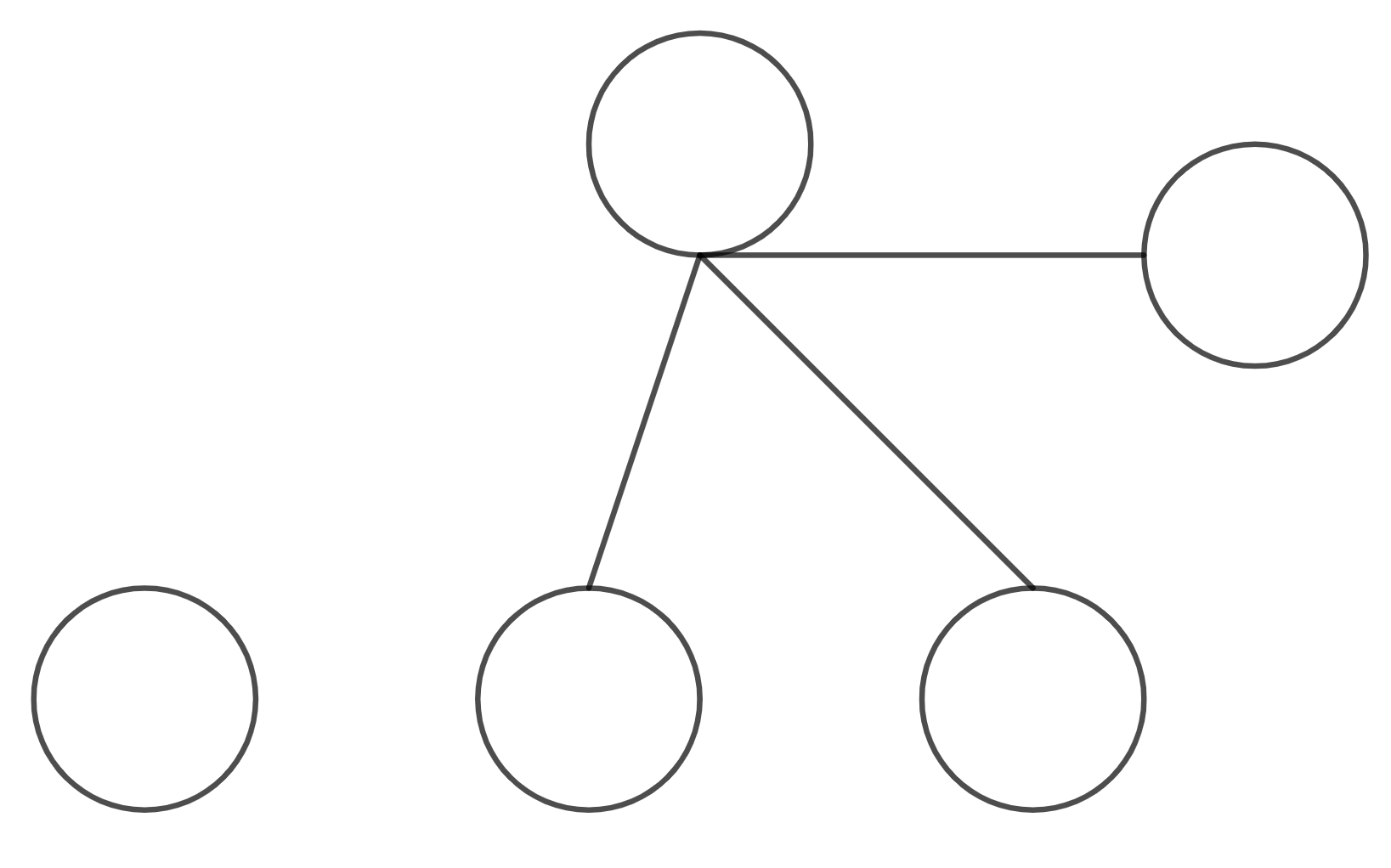}\\
    \text{c) The forest of two edge connected components $\mathcal{F}(G^\pi)$.} \\
    \caption{The graphs of Example \ref{Example: order of a product of deterministic matrices}.}
    \label{Figure: forest of two edge connected components}
\end{figure}

The following simple lemma can be seen from the proof of \cite[Lemma 17]{MMPS}. For the convenience of the reader, we reproduce it here. 

\begin{lemma}\label{Lemma: Order doesn't increase when doing quotient}

Let $\pi,\sigma\in \cP(2m)$ be such that $\pi\leq\sigma$, then,
$$t(G^\sigma)\leq t(G^\pi).$$
\end{lemma}
\begin{proof}
Observe that the graph $G^{\sigma}$ can be seen as a quotient graph of $G^\pi$ by joining vertices of $G^\pi$ that correspond to blocks of $\pi$ in the same block of $\sigma$. It is clear that when we identify vertices the number of two edge connected components of the graphs cannot increase as a two edge connected component in the original graph will still be part of a two edge connected component in the new graph. From this observation the proof follows directly.
\end{proof}

\subsection{Quotient graphs induced by graphs}\label{SubSection: Graphs induced by graphs}

Let us now recall a bit about the proof strategy. For the Gaussian case, we will use Lemmas \ref{Lemma: Order of sums associated to deterministic matrices} and \ref{Lemma: Order doesn't increase when doing quotient} to find the bound of the cumulants. The Wigner case is more involved. The main difficulty is that all even cumulants of $x_{i,j}$ for the Wigner case might be nonzero, unlike the Gaussian case where only the second one is nonzero. It would mean that we need to consider more partitions and hence more complicated graph structures in the Wigner case. Thankfully, we will be able to eventually reduce the Wigner case to the Gaussian case. The main idea is that any graph associated to the Wigner case can be seen as the union of many graphs, each associated to the Gaussian case. Here by union of graphs we mean graphs that are merged by joining vertices. The method that we introduce here permits us to find the number of two edge connected components of the new graph in terms of the number of two edge connected components of the graphs that are merged.

As discussed above, we would like to define a graph $G$ as the union of disjoint graphs $G_1,\dots,G_n$ by merging them through their vertices. Further, we will explore the relation between $t(G)$ and $\sum_i t(G_i)$, where $t(\cdot)$ is the exponent defined in Lemma \ref{Lemma: Order of sums associated to deterministic matrices}. The simplest way to define the graph $G$ will be using another graph $T$ whose vertices will be indexed by the graphs $G_i$ and an edge $\{G_i,G_j\}$ will determine how we glue the graphs $G_i$ and $G_j$. We will call this a quotient graph induced by the graph $T$. Let us give a more rigorous definition.

\begin{definition}\label{Definition: Quotient graph induces by graph}
Let $G_1,\dots,G_n$ be unoriented graphs with set of vertices and edges $(V_i,E_i)$ for $1\leq i\leq n$. Let $T=(V_T,E_T)$ be an unoriented graph with set of vertices $V_T=\{G_i:1\leq i\leq n\}$. For each edge $e=\{G_i,G_j\}\in E_T$ we also consider a pair of vertices $\{v_i,v_j\}$ where $v_i$ and $v_j$ are vertices of $G_i$ and $G_j$ respectively. We denote the vertices $\{v_i,v_j\}$ associated with the edge $e\in E_T$ by $V(e)$.  
We call the triple $T=(V_T,E_T,\{V(e)\}_{e\in E_T})$ a \textit{graph of the graphs} $G_1,\dots,G_n$. We let $\pi_T\in \cP(\cup_{i=1}^n V_i)$ be the partition given by $u$ and $v$ being in the same block of $\pi_T$ whenever there is an edge $e\in E_T$ such that $V(e)=\{u,v\}$. Let $G=\cup_{i=1}^n G_i$ be the union of the graphs, that is, the graph with set of vertices $\cup_{i=1}^n V_i$ and set of edges $\cup_{i=1}^n E_i$. The quotient graph of $G$ under $\pi_T$ denoted by $G^T$ will be called the \textit{ quotient graph induced by the graph} $T$. 
\end{definition}

\begin{remark}
In Definition \ref{Definition: Quotient graph induces by graph},  the graph $T$ is allowed to have multiple edges and loops,  although most of our interesting results will follow directly from the case where $T$ is a tree. For instance, see Example \ref{Example: Graph induced by graphs} for a quotient graph induced by a tree.
\end{remark}

\begin{example}\label{Example: Graph induced by graphs}
Let $G_1=(V_1,E_1),G_2=(V_2,E_2)$ and $G_3=(V_3,E_3)$ be the graphs with set of vertices
\begin{eqnarray*}
V_1 = \{1,2,3\}, \quad
V_2 = \{4,5,6\}, \quad
V_3 = \{7,8,9\},
\end{eqnarray*}
and edges
\begin{eqnarray*}
E_1 & = & \{e_1=\{1,2\},e_2=\{1,2\},e_3=\{2,3\},e_4=\{3,3\}\} \\
E_2 & = & \{e_5=\{4,5\},e_6=\{5,6\},e_7=\{4,6\}\} \\
E_3 & = & \{e_8=\{7,8\},e_9=\{7,8\},e_{10}=\{7,7\},e_{11}=\{9,9\}\}.
\end{eqnarray*}
We let $T=(V_T,E_T,\{V(e)\}_{e\in E_T})$ be the graph of the graphs $G_1,G_2,G_3$ with vertices $V_T=\{G_1,G_2,G_3\}$ and edges $E_T=\{e_1^T,e_2^T\}$ given by
\begin{align*}
e_1^T=\{G_1,G_2\},e_2^T=\{G_1,G_3\},\qquad
V(e_1^T)=\{3,4\},V(e_2^T)=\{2,8\}.
\end{align*}
Therefore $\pi_T$, which we simply denote by $\pi$,  is given by
$$\pi=\{1\},\{2,8\},\{3,4\},\{5\},\{6\},\{7\},\{9\}.$$
If $G=G_1\cup G_2\cup G_3$, then $G^T$ is a graph with vertices given by the blocks of $\pi$ and edges $e_1^\pi=\{\{1\},\{2,8\}\}, e_2^\pi=\{\{1\},\{2,8\}\}, e_3^\pi=\{\{2,8\},\{3,4\}\}, e_4^\pi=\{\{3,4\},\{3,4\}\}, e_5^\pi=\{\{3,4\},\{5\}\}, e_6^\pi=\{\{5\},\{6\}\}, e_7^\pi=\{\{3,4\},\{6\}\}, e_8^\pi=\{\{2,8\},\{7\}\}, e_9^\pi=\{\{2,8\},\{7\}\}, e_{10}^\pi=\{\{7\},\{7\}\}$ and $e_{11}^\pi=\{\{9\},\{9\}\}$. The figures of these graphs can be seen in Figure \ref{Figure: Graph induced by graphs}.
\end{example}

\begin{figure}
    \centering
    \includegraphics[width=0.8\textwidth]{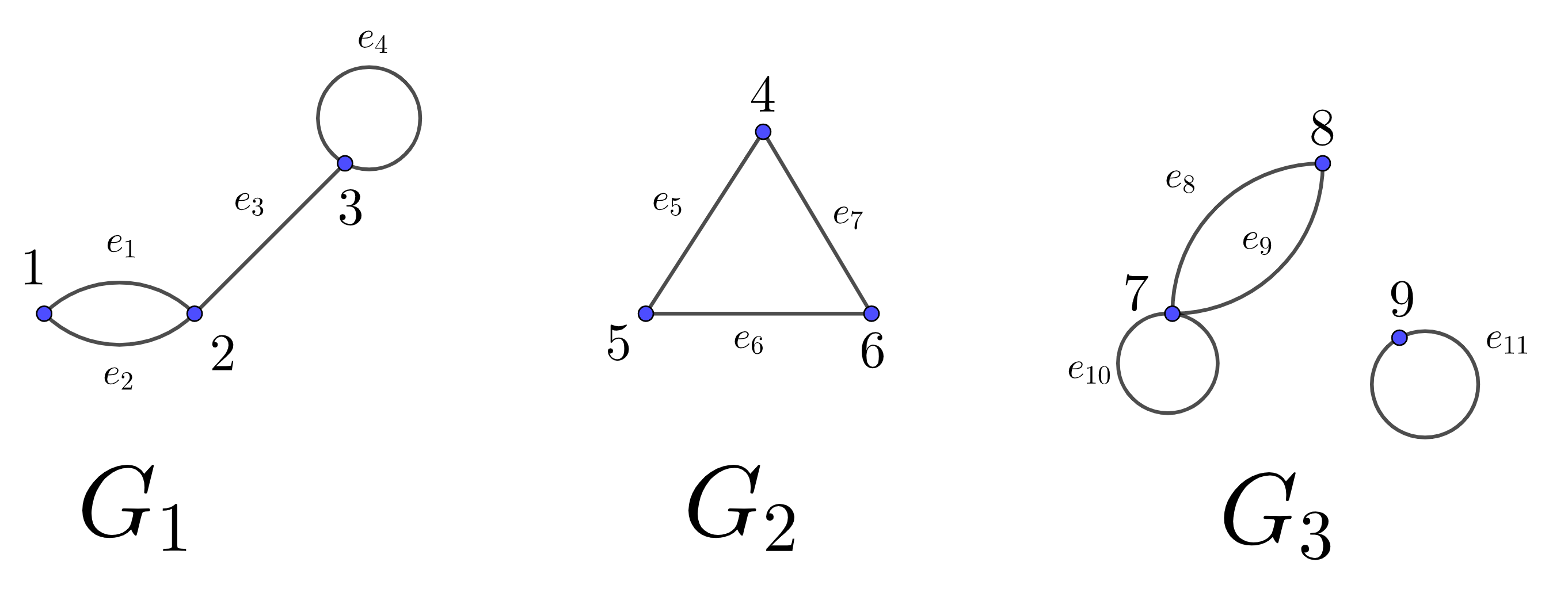}\\
    \text{a) The graphs $G_1,G_2,G_3$.}\\
    \includegraphics[width=0.7\textwidth]{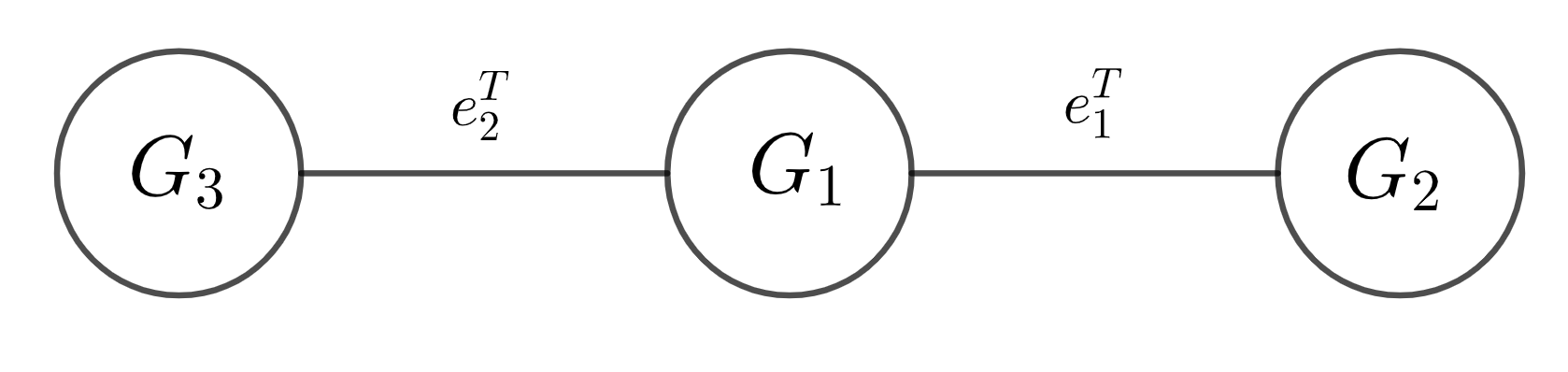}\\
    \text{b) The graph $T$.} \\
    \includegraphics[width=0.5\textwidth]{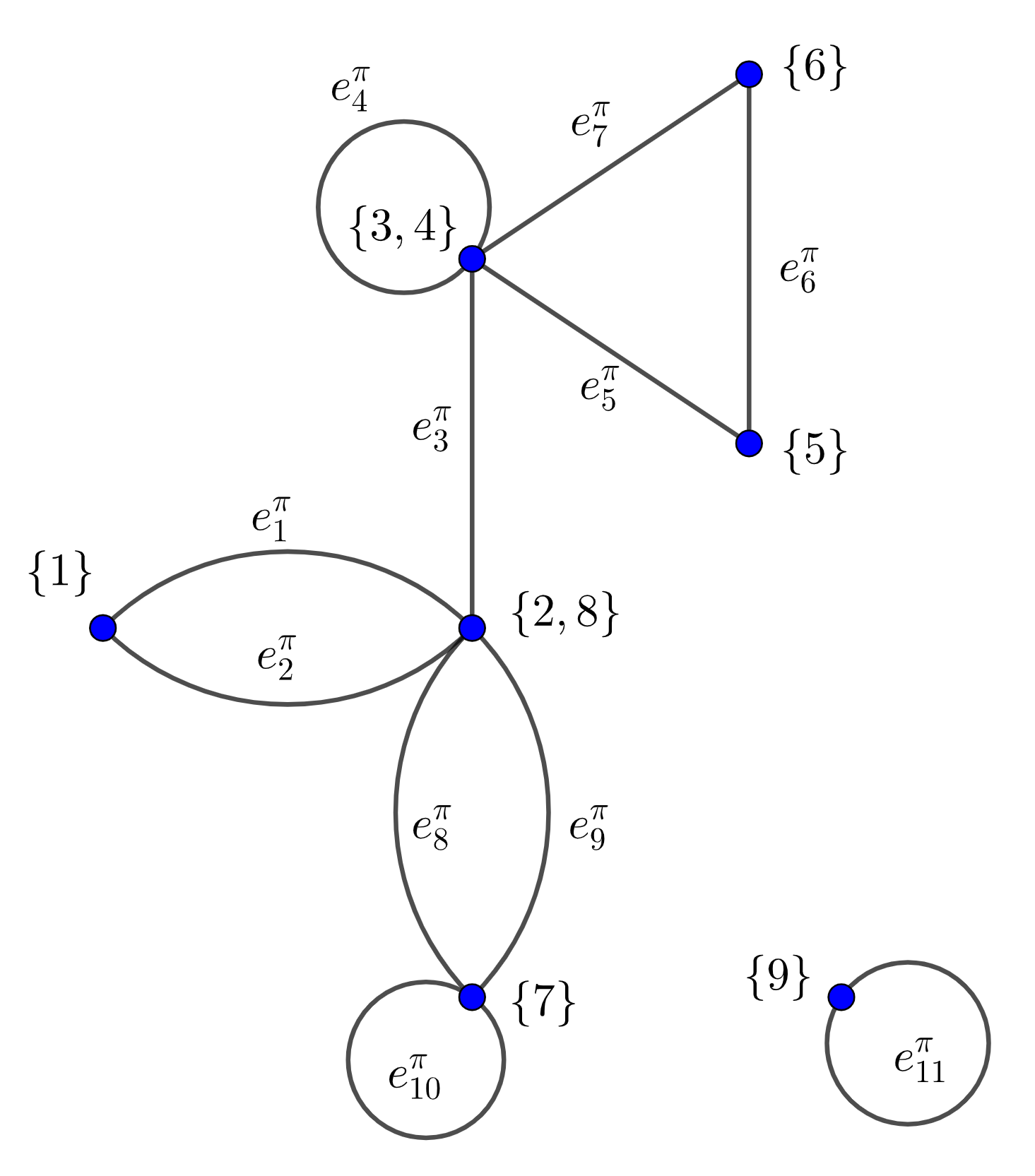}\\
    \text{c) The graph $G^T$.} \\
    \caption{The graphs of Example \ref{Example: Graph induced by graphs}.}
    \label{Figure: Graph induced by graphs}
\end{figure}

\begin{lemma}\label{Lemma: Order of exponent when inducing graph is a tree}
Let $G_1,\dots,G_n$ be unoriented graphs and let $T=(V_T,E_T,\ab\{V(e)\}_{e\in E_T})$ be a graph of the graphs $G_i$. Let $G=\cup_{i=1}^n G_i$ be the graph consisting of the union of the graphs $G_i$. If $T$ is a tree, then
$$t(G^T) = \sum_{i=1}^n t(G_i)-n+1.$$
\end{lemma}

\begin{proof}
We prove this by induction on $n$. In the case $n=1$ we have a single graph $G_1$, and hence $G^T=G$, from which the conclusion follows. Let us assume that it is true for $n$ and  prove it for $n+1$. Since $T$ is a tree, it has a leaf consisting of a vertex, say $G_{n+1}$,  and an edge, say $e=\{G_n,G_{n+1}\}$,  with $V(e)=\{v_n,v_{n+1}\}$. Let $T^\prime$ be the graph $T$ without the vertex $G_{n+1}$ and the edge $e$. So $T^\prime$ is a tree with $n$ vertices. Let $G^\prime=\cup_{i=1}^n G_i$, then by induction hypothesis we know
$$t({G^\prime}^{T^\prime})= \sum_{i=1}^n t(G_i)-n+1.$$
Observe that since $(G_{n+1},e)$ is a leaf, the graph $G^{T}$ can be obtained from merging the vertex $[v_n]_{\pi_{T^\prime}}$ of ${G^\prime}^{T^\prime}$ and the vertex $v_{n+1}$ of $G_{n+1}$. If we merge two graphs by joining two vertices, one from each graph, it is clear that
$$t(G^T)=t({G^\prime}^{T^\prime})+t(G_{n+1})-1,$$
since the two edge connected component of ${G^\prime}^{T^\prime}$ that contains $[v_n]_{\pi_{T^\prime}}$ and the two edge connected component of $G_{n+1}$ that contains $v_{n+1}$ are joined into the same two edge connected component of the new graph. Hence,
\begin{eqnarray*}
t(G^T)&=&t({G^\prime}^{T^\prime})+t(G_{n+1})-1 \\
& = & \sum_{i=1}^n t(G_i)-n+1+t(G_{n+1})-1=\sum_{i=1}^{n+1} t(G_i)-n,
\end{eqnarray*}
where in last equality we use the induction hypothesis.
\end{proof}

\begin{corollary}\label{Corollary: Order of exponent when inducing graph is connected}
Let $G_1,\dots,G_n$ be unoriented graphs and let $T=(V_T,E_T,\ab\{V(e)\}_{e\in E_T})$ be a graph of the graphs $G_i$. Let $G=\cup_{i=1}^n G_i$ be the graph consisting of the union of the graphs $G_i$. If $T$ is connected, then
$$t(G^T) \leq \sum_{i=1}^n t(G_i)-n+1.$$
\end{corollary}
\begin{proof}
Since $T$ is connected it has a spanning tree $T_S$. It is clear by definition that $\pi_{T_S} \leq \pi_{T}$, as  (possibly) more vertices of $G$ are merged since  $T$ may have extra edges apart from those in $T_S$. Hence, from Lemmas \ref{Lemma: Order doesn't increase when doing quotient} and \ref{Lemma: Order of exponent when inducing graph is a tree}, it follows 
$$t(G^T)\leq t(G^{T_S}) = \sum_{i=1}^n t(G_i)-n+1.$$
\end{proof}

\section{Expansion of the cumulants} \label{Section_Genus expansion of the cumulants}

Let $r\geq 1$ and $m_1,\dots,m_r \in\mathbb{N}$. We let $M_i=\sum_{j=1}^im_j$ and $m=M_r=\sum_{j=1}^r m_j$. In order to provide an upper bound for the cumulants of $\text{Tr} P(\mathcal{X},\mathcal{D})$, we get started by providing an upper bound for cumulants of the form

\begin{equation}\label{Equation: Cumulants of monomials}
\C_r(\text{Tr}(X_{i_1}D_{j_1}\cdots X_{i_{m_1}}D_{j_{m_1}}),\dots,\text{Tr}(X_{i_{M_{r-1}+1}}D_{i_{M_{r-1}+1}}\cdots X_{i_m}D_{j_m})). 
\end{equation}

Then an upper bound for the cumulants of $\mathbf{X}$ can be obtained by the multi-linearity of the cumulants. For brevity,  we set the following notation  for the rest of the section.

\begin{notation}\label{Notation: Y_k}
For $1\leq k\leq r$, we let
$$Y_{k}\vcentcolon=X_{i_{M_{k-1}+1}}D_{j_{M_{k-1}+1}}\cdots X_{i_{M_k}}D_{j_{M_k}},$$
with the convention $M_0=0$. We also write $i=(i_1,\dots,i_m)$. Since we allow repetition of the indices,  we let $ker(i)\in \cP(m)$ be the partition given by $u$ and $v$ being in the same block of $ker(i)$ whenever $i_u=i_v$.
\end{notation}

\begin{notation} \label{def of gamma}
We denote by $\pm[m]\vcentcolon=\{-m,-m+1,\dots,-1,1,\dots,m\}$,  and for $1\leq k\leq r$,  we denote by $\pm[M_{k-1}+1,M_k]$ the set $\{M_{k-1}+1,\dots,M_k,-(M_{k-1}+1),\dots,-M_k\}.$
We further let $\gamma \in S_m$ be the permutation with cycle decomposition,
$$(1,\dots,m_1)(m_1+1,\dots,m_1+m_2)\cdots (m_1+\dots +m_{r-1}+1,\dots,m).$$
\end{notation}

Observe that for $1\leq k\leq r$,
\begin{eqnarray*}
\text{Tr}(Y_k)=N^{-m_k/2}\sum_{\psi: \pm[M_{k-1}+1,M_k]\rightarrow [N]}\prod_{l=M_{k-1}+1}^{M_k} x_{\psi(l),\psi(-l)}^{(i_l)}d_{\psi(-l),\psi(\gamma(l))}^{(j_l)}.
\end{eqnarray*}

Further, for $\psi: \pm[m]\rightarrow [N]$ and $1\leq k\leq r$, we denote
$$Z_k(\psi)\vcentcolon=\prod_{l=M_{k-1}+1}^{M_k} x_{\psi(l),\psi(-l)}^{(i_l)}d_{\psi(-l),\psi(\gamma(l))}^{(j_l)}.$$
With these notations, we have 
\begin{eqnarray*}
\C_r(\text{Tr}(Y_1),\dots,\text{Tr}(Y_r)) &=& N^{-m/2}\sum_{\psi: \pm[m]\rightarrow [N]}\C_r(Z_1(\psi),\dots,Z_r(\psi)).
\end{eqnarray*}
Now, using the multi-linearity of the cumulants and the fact that $d_{i,j}^{(k)}$ are all deterministic, we obtain

\begin{eqnarray*}
\C_r(\text{Tr}(Y_1),\dots,\text{Tr}(Y_r)) &=& N^{-m/2}\sum_{\psi: \pm[m]\rightarrow [N]}\mathbf{D}(\psi)\C_r(x_1(\psi),\dots,x_r(\psi)),
\end{eqnarray*}
with
\begin{align*} \mathbf{D}(\psi)=\prod_{l=1}^m d_{\psi(-l),\psi(\gamma(l))}^{(j_l)},\qquad x_k(\psi)=\prod_{l=M_{k-1}+1}^{M_k}x_{\psi(l),\psi(-l)}^{(i_l)}
\end{align*}
for any $1\leq k\leq r$. Since our matrices are permutation invariant according to Definition \ref{Definition: Wigner Matrix}, the quantity $\C_r(x_1(\psi),\dots,x_r(\psi))$ depends only on  $ker(\psi)$, which is the partition given by $u$ and $v$ being in the same block if and only if $\psi(u)=\psi(v)$. Thus,

\begin{eqnarray*}
\C_r(\text{Tr}(Y_1),\dots,\text{Tr}(Y_r)) = N^{-m/2}\sum_{\pi\in \cP(\pm[m])}\left[\sum_{\substack{\psi: \pm[m]\rightarrow [N] \\ ker(\psi)=\pi}}\mathbf{D}(\psi)\right]\C_r(x_1(\psi^\pi),\dots,x_r(\psi^\pi)),
\end{eqnarray*}
where $\psi^\pi$ is any function $\psi: \pm[m] \rightarrow [N]$ such that $ker(\psi)=\pi$.

Now, we invoke \cite[Theorem 11.30]{NS} to get the expression
\begin{eqnarray}\label{Equation: Cumulant expansion}
&&\C_r(\text{Tr}(Y_1),\dots,\text{Tr}(Y_r)) = \nonumber\\
&& N^{-m/2}\sum_{\pi\in \cP(\pm[m])}\left[\sum_{\substack{\psi: \pm[m]\rightarrow [N] \\ ker(\psi)=\pi}}\mathbf{D}(\psi)\right]\sum_{\substack{\tau\in \cP(m) \\ \tau\vee\gamma=1_m}}\C_\tau(x_{\psi^\pi(1),\psi^\pi(-1)}^{(i_1)},\dots,x_{\psi^\pi(m),\psi^\pi(-m)}^{(i_m)}) \nonumber\\
&=& N^{-m/2}\sum_{\pi\in \cP(\pm[m])}\sum_{\substack{\tau\in \cP(m) \\ \tau\vee\gamma=1_m}}\left[\sum_{\substack{\psi: \pm[m]\rightarrow [N] \\ ker(\psi)=\pi}}\mathbf{D}(\psi)\right]\C_\tau(x_{\psi^\pi(1),\psi^\pi(-1)}^{(i_1)},\dots,x_{\psi^\pi(m),\psi^\pi(-m)}^{(i_m)}).
\end{eqnarray}

For each $\pi\in \cP(\pm[m])$,  the term 
$$N^{-m/2}\sum_{\substack{\psi: \pm [m]\rightarrow [N] \\ ker(\psi)=\pi}}\mathbf{D}(\psi)$$ 
determines the $N$-order of each summand.  Hence, with the help of this expression, we do an expansion of the cumulant in terms of the order of $N$. Furthermore, for some cases (like when all deterministic matrices are the identity) this term is a polynomial in $1/N$ and hence $\C_r(\text{Tr}(Y_1),\dots,\text{Tr}(Y_r))$ can be expressed as a polynomial in $1/N$.  This expansion as a polynomial in $1/N$ is usually referred as a genus expansion of the cumulants. 
The leading term in this expansion is of particular interest in free probability; see \cite[Section 2.3]{CMSS}. Some research has been done in this direction, see for example \cite{MM} for the Wigner model, \cite{MN} for the Wishart model and \cite{SP} for the SYK model. We also refer to the recent work \cite{F} on a truncated (up to a fixed order) genus expansion of the expectation of the trace of smooth function of polynomial of GUE and deterministic matrices. 

\begin{remark}
It is important to observe that both $\mathbf{D}(\psi)$ and $\gamma$ depend on $r$. We drop the $r$ dependence to make the notation more compact.
\end{remark}

\section{Upper bounds for $t(\cdot)$}\label{Upper bounds for t}

Let us remind that one of our goals is to find an upper bound for the cumulants of the form (\ref{Equation: Cumulants of monomials}). Thanks to (\ref{Equation: Cumulant expansion}) this can be done by finding upper bounds for the quantities
\begin{equation}\label{Equation: Deterministic sum asociated to Gaussian case depending on tau when equality}
\sum_{\substack{\psi: \pm [m]\rightarrow [N] \\ ker(\psi)=\pi}}\mathbf{D}(\psi),
\end{equation}
for each $\pi\in \cP(\pm[m])$ for which the term $\C_\tau(x_{\psi^\pi(1),\psi^\pi(-1)}^{(i_1)},\dots,x_{\psi^\pi(m),\psi^\pi(-m)}^{(i_m)})$ in (\ref{Equation: Cumulant expansion}) is not zero. The structure of such a partition $\pi$ depends on the model. In the GUE case we will see that $\pi$ has a very simple expression,  while in the Wigner case it is more involved. Further, observe that (\ref{Equation: Deterministic sum asociated to Gaussian case depending on tau when equality}) has a form similar to (\ref{Equation: Sum associated to product of deterministic matrices}). However, here we index the vertices by $\pm [m]$ in order to distinguish $\mathcal{X}$ entries from $\mathcal{D}$ entries more conveniently, and instead of an inequality constraint for $\ker(\cdot)$, we have an equality constraint. Correspondingly, we introduce the graphs $D$ and $G$ below in Notation \ref{Notation: The graphs D and G}, where the former encodes the product of deterministic entries, and the latter encodes the product of both deterministic and random entries. We will also need $D^{\pi}$, i.e., the quotient graph of $D$.  Then, similarly to Lemma \ref{Lemma: Order of sums associated to deterministic matrices}, the upper bounds of (\ref{Equation: Deterministic sum asociated to Gaussian case depending on tau when equality}) and its variants boil down to the upper bounds for various $t(D^{\pi})$. In this section we aim to establish the necessary bounds for $t(D^{\pi})$ for the ensembles of  GUE, GOE and Wigner matrices. We will state the discussion for GUE/GOE in Section \ref{s.the gaussian case}, and state that for the Wigner case in Section \ref{s. Wigner case}. 

\subsection{The Gaussian case} \label{s.the gaussian case}
%

Let us first introduce the graph notations needed. 

\begin{notation}\label{Notation: The graphs D and G}
We let $D=(V_D,E_D)$ be the oriented graph with set of vertices $\pm [m]$ and set of edges $e_k =(\gamma(k),-k)$ for $1\leq k\leq m$. Here $\gamma\in S_m$ is the permutation defined as in Section \ref{Section_Graph theory and combinatorics}. We let $G=(V_G,E_G)$ be the graph with set of vertices $\pm [m]$ and set of edges 
$$\{\gamma(k),-k\},\{k,-k\},$$
for $1\leq k\leq m$. We use the notation
\begin{align*}
e_k^D=\{\gamma(k),-k\},\qquad e_k^X=\{k,-k\},
\end{align*}
as the edges $e_k^D$ label the deterministic matrices while the edges $e_k^X$ label the random matrices.
\end{notation}

\begin{notation}\label{Notation: definition of pi_tau}
For an even number $m$ and a pairing $\tau\in \cP_2(m)$, we let $\pi_\tau$ be the pairing of $\cP_2(\pm [m])$ whose blocks are $\{u,-v\}$ and $\{v,-u\}$ whenever $\{u,v\}$ is a block of $\tau$.
\end{notation}

In the GUE case, we do not have to estimate (\ref{Equation: Deterministic sum asociated to Gaussian case depending on tau when equality}) for each $\pi$ individually. Actually, when we sum over those $\pi$'s corresponding to nonzero $\C_\tau(x_{\psi^\pi(1),\psi^\pi(-1)}^{(i_1)},\dots,x_{\psi^\pi(m),\psi^\pi(-m)}^{(i_m)})$, we only need to consider the sum of the following form 
\begin{equation}\label{Equation: Deterministic sum asociated to Gaussian case depending on tau}
\sum_{\substack{\psi: \pm [m]\rightarrow [N] \\ ker(\psi)\geq \pi_\tau}}\mathbf{D}(\psi).
\end{equation}
 This will be further explained in Section \ref{Section: The order of GUE}. 

Observe that (\ref{Equation: Deterministic sum asociated to Gaussian case depending on tau}) has exactly the same form as (\ref{Equation: Sum associated to product of deterministic matrices}). So, as we discussed before, it is enough to find an upper bound for $t(D^{\pi_\tau})$, in the spirit of Lemma \ref{Lemma: Order of sums associated to deterministic matrices}.

\begin{theorem}\label{Theorem: Order of GUE case for t()}
Let $m$ be even, $\tau\in \cP_2(m)$ be such that $\tau\vee\gamma=1_m$ and $\pi_\tau$ as in Notation \ref{Notation: definition of pi_tau}. Then,
$$t(D^{\pi_\tau})\leq m/2+2-r.$$
\end{theorem}

We will prove Theorem \ref{Theorem: Order of GUE case for t()} by induction on $r$. We start from the initial case $r=1$.

\begin{lemma}\label{Theorem: Order of GUE case for t() base case}
Let $m$ be even, $\tau\in \cP_2(m)$ and $\pi_\tau$ as in Notation \ref{Notation: definition of pi_tau}. Let $D$ be the graph of Notation \ref{Notation: The graphs D and G} with $r=1$ i.e., $\gamma=(1,\dots,m)$. Then,
$$t(D^{\pi_\tau})\leq m/2+1.$$
\end{lemma}

\begin{proof}
For brevity let us denote $\pi=\pi_\tau$ within this proof. Observe that any block $\{u,-v\}$ of $\pi$ which is a vertex of $D^\pi$ is adjacent to the two edges $e_v=(\gamma(v),-v)$ and $e_{\gamma^{-1}(u)}=(u,-\gamma^{-1}(u))$. So every vertex of $D^\pi$ has multiplicity $2$, and hence every connected component of $D^{\pi}$ is a cycle, which means that $t(D^\pi)$ is just the number of cycles of $D^\pi$ (or its number of connected components). As pointed out in Remark \ref{Remark: The forest of two edge connected components doesnt care about orientation}, we can forget the orientation of the edges of $D^\pi$, and instead work with the unoriented graph $\underline{D^\pi}$ which is still denoted by $D^\pi$. 

We will prove that the number of cycles of $D^\pi$ is at most $m/2+1$ by induction on $m$. The base case $m=2$ follows easily, as $D^\pi$ consists of two cycles, and hence $m/2+1=2/2+1=2$. Now let us assume $t(D^\pi)\leq n$ where $m=2(n-1)$. We aim to prove the result  for $m=2n$. Note that the edges $e_k^D$ of the graph $G$ are the same as the edges of the graph $D$. On the other hand, we may regard our pairing $\tau$ as a pairing of the edges $e_k^X$ of $G$. The latter means that, given a pairing $\tau$ of the edges $e_k^X$ of $G$,  the graph $D^\pi$ is the same as the graph $G^\pi$ restricted to its set of edges $(e_k^D)^\pi$ (for an example see Figure \ref{Figure: Quotient graph of G}). Let $B=\{u,v\}\in \tau$ be a block of $\tau$. We define the graph $G_B$ as the graph with vertices $\pm[m]\setminus \{u,-u,v,-v\}$, and set of edges the same as the edges of $G$ except for the edges $e_{u}^X,e_v^X,e_u^D$ and $e_v^D$, which are removed, and the edges $e_{\gamma^{-1}(u)}^D,e_{\gamma^{-1}(v)}^D$,  which are redefined as follows
\begin{eqnarray*}
e_{\gamma^{-1}(u)}^D = \{\gamma(u),-\gamma^{-1}(u)\}, \qquad e_{\gamma^{-1}(v)}^D = \{\gamma(v),-\gamma^{-1}(v)\}
\end{eqnarray*}
whenever $v\neq \gamma(u)$ and $u\neq \gamma(v)$. If $\gamma(u)=v$ then $G_B$ is defined as before except for $e_{\gamma^{-1}(u)}^D$, which is redefined as
\begin{eqnarray*}
e_{\gamma^{-1}(u)}^D &=& \{\gamma(v),-\gamma^{-1}(u)\}.
\end{eqnarray*}
Similarly if $\gamma(v)=u$ then only the edge $e_{\gamma^{-1}(v)}^D$ is redefined as
\begin{eqnarray*}
e_{\gamma^{-1}(v)}^D &=& \{\gamma(u),-\gamma^{-1}(v)\}.
\end{eqnarray*} 
The graph $G_B$ is a cycle with $2n-2$ edges. If $\tau^\prime$ is the pairing $\tau$ without the block $B$, then $\tau^\prime$ is a pairing of the edges $e_k^X$ of $G_B$. The pairing $\tau^\prime$ determines a partition $\pi^\prime$ which is the same as $\pi$ after we removed the blocks $\{u,-v\}$ and $\{v,-u\}$. If $D_B$ is the graph $G_B$ restricted to the set of edges $e_k^D$ then by induction hypothesis we know that the graph $D_B^{\pi^\prime}$ satisfies
$$t(D_B^{\pi^\prime})\leq n.$$
We claim that,
$$t(D^\pi)\leq t(D_B^{\pi^\prime})+1.$$
To prove this, observe that the graph $D^{\pi}$ and the graph $D_B^{\pi^\prime}$ have the same vertices and edges except for the vertices of $D^\pi$ that are removed and the edges that are redefined. This means that we can easily obtain the graph $D_B^{\pi^\prime}$ from the graph $D^\pi$ by removing these vertices and changing the edges that were redefined. If $v=\gamma(u)$, then the graph $D_B^{\pi^\prime}$ can be obtained from the graph $D^\pi$ by removing the vertices $\{u,-v\}$ and $\{v,-u\}$, and the edges $e_u^\pi$ and $e_v^\pi$. Further, the edge $e_{\gamma^{-1}(u)}^\pi$ which connects $\{-\gamma^{-1}(u),\tau(\gamma^{-1}(u))\}$ and $\{u,-v\}$ in $D^\pi$ now instead connects $\{-\gamma^{-1}(u),\tau(\gamma^{-1}(u))\}$ and $\{\gamma(v),-\tau(\gamma(v))\}$ in $D_B^{\pi^\prime}$. Observe that the cycle of $D^\pi$ that contains $e_{\gamma^{-1}(u)}^\pi$ contains the path
$$\{\gamma(v),-\tau(\gamma(v))\}\overset{e_v^\pi}{-} \{u,-v\} \overset{e_{\gamma^{-1}(u)}^\pi}{-} \{-\gamma^{-1}(u),\tau(\gamma^{-1}(u))\}.$$
But as we mentioned earlier, the edge $e_{\gamma^{-1}(u)}^\pi$ of $D_B^{\pi^\prime}$ connects $\{-\gamma^{-1}(u),\ab \tau(\gamma^{-1}(u))\}$ and $\{\gamma(v),-\tau(\gamma(v))\}$, which means that the path described above becomes the path
$$\{\gamma(v),-\tau(\gamma(v))\} \overset{e_{\gamma^{-1}(u)}^\pi}{-} \{-\gamma^{-1}(u),\tau(\gamma^{-1}(u))\}.$$
So the cycle that contains $e_{\gamma^{-1}(u)}^\pi$ in $D^\pi$ is the same as the cycle that contains $e_{\gamma^{-1}(u)}^\pi$ in $D_B^{\pi^\prime}$, with only one vertex and edge less. Finally, observe that $e_u^\pi$ is a loop of $D^\pi$ at the vertex $\{-u,v\}$. Therefore, $D^\pi$ has exactly one cycle more than $D_B^{\pi^\prime}$. The case $u=\gamma(v)$ is analogous, so we are reduced to check the case when $u\neq\gamma(v)$ and $v\neq \gamma(u)$. In this case, observe that the edges $e_v^\pi$ and $e_{\gamma^{-1}(u)}^\pi$ are in the same cycle of $D^\pi$, as they are adjacent to the same vertex $\{u,-v\}$. Moreover, this cycle also contains the path
$$\{\gamma(v),-\tau(\gamma(v))\} \overset{e_v^\pi}{-} \{u,-v\} \overset{e_{\gamma^{-1}(u)^\pi}}{-} \{-\gamma^{-1}(u),\tau(\gamma^{-1}(u))\}.$$
Similarly, $e_u^\pi$ and $e_{\gamma^{-1}(v)}^\pi$ are in the same cycle of $D^\pi$, as they are adjacent to the same vertex $\{v,-u\}$ and this cycle contains the path
$$\{\gamma(u),-\tau(\gamma(u))\} \overset{e_u^\pi}{-} \{v,-u\} \overset{e_{\gamma^{-1}(v)^\pi}}{-} \{-\gamma^{-1}(v),\tau(\gamma^{-1}(v))\}.$$
On the other hand, in the graph $D_B^{\pi^\prime}$, the edge $e_{\gamma^{-1}(u)}$ connects $\{-\gamma^{-1}(u),\ab \tau(\gamma^{-1}(u))\}$ and $\{\gamma(u),-\tau(\gamma(u))\}$. Similarly, the edge $e_{\gamma^{-1}(v)}$ connects $\{-\gamma^{-1}(v),\tau(\gamma^{-1}(v))\}$ and $\{\gamma(v),-\tau(\gamma(v))\}$. 
If the cycles of $D^\pi$ that contain $e_u^\pi$ and $e_v^\pi$ are distinct, then in the graph $D_B^{\pi^\prime}$ these cycles are merged into a single one and therefore $D_B^{\pi^\prime}$ has one cycle less than $D^\pi$. If the cycles are the same, then we have two possible scenarios. Either there is a path from $\{-\gamma^{-1}(u),\tau(\gamma^{-1}(u))\}$ to $\{\gamma(u),-\tau(\gamma(u))\}$, or there is a path from $\{-\gamma^{-1}(u),\tau(\gamma^{-1}(u))\}$ to $\{-\gamma^{-1}(v),\tau(\gamma^{-1}(v))\}$. In the former case $D_B^{\pi^\prime}$ has one cycle more than $D^\pi$ as the cycles that contains $e_{\gamma^{-1}(u)}^\pi$ and $e_{\gamma^{-1}(v)}^\pi$ of $D_B^{\pi^\prime}$ are distinct. In the latter case, $D_B^{\pi^\prime}$ has the same number of cycles than $D^\pi$ as the cycles that contains $e_{\gamma^{-1}(u)}^\pi$ and $e_{\gamma^{-1}(v)}^\pi$ of $D_B^{\pi^\prime}$ are the same. In either case $t(D^\pi)\leq t(D_B^{\pi^\prime})+1$. We conclude by using the induction hypothesis
$$t(D^\pi)\leq t(D_B^{\pi^\prime})+1\leq n+1.$$
\end{proof}
Next, we proceed with the proof of Theorem \ref{Theorem: Order of GUE case for t()} for general $r$. 
\begin{proof}[Proof of Theorem \ref{Theorem: Order of GUE case for t()}.] We proceed by induction on $r$. The case $r=1$ is completed thanks to Lemma \ref{Theorem: Order of GUE case for t() base case}. We assume it is true for $r-1$ and  prove it for $r$. Since $\tau\vee\gamma=1_m$,  there exists a block $B=\{u,v\}$ of $\tau$ such that $1\leq u\leq m_1$ and $m_1<v$. Let us assume without loss of generality that $1\leq u\leq m_1$ and $m_1+1\leq v\leq m_1+m_2$. We proceed similarly as in Lemma \ref{Theorem: Order of GUE case for t() base case}. Let $G_B$ be the graph with vertices $\pm[m]\setminus\{u,v,-u,-v\}$ and set of edges the same as the edges of $G$ except for $e_u^X,e_u^D,e_v^X$ and $e_v^D$ which are removed and the edges $e_{\gamma^{-1}(v)}^D,e_{\gamma^{-1}(u)}^D$ which are defined as
\begin{eqnarray*}
e_{\gamma^{-1}(v)}^D = \{\gamma(u),-\gamma^{-1}(v)\} \qquad
e_{\gamma^{-1}(u)}^D = \{\gamma(v),-\gamma^{-1}(u)\}.
\end{eqnarray*}
In this way, $G_B$ is now a graph with $r-1$ cycles and $2m-4$ edges. If $\tau^\prime$ is the pairing with blocks the same as $\tau$ except for $\{u,v\}$ then $\tau^\prime$ determines a partition $\pi^\prime$ with blocks the same as $\pi_\tau$(which we simply denote by $\pi$) except for the blocks $\{u,-v\}$ and $\{v,-u\}$. By induction hypothesis, we know that,
$$t(D_B^{\pi^\prime})\leq \frac{m-2}{2}+2-r+1=m/2+2-r.$$
To conclude, it is enough to observe that the graph $D_B^{\pi^\prime}$ can be obtained by removing the vertices $\{u,-v\}$ and $\{v,-u\}$ of the graph $D^\pi$, removing the edges $(e_u^D)^\pi$ and $(e_v^D)^\pi$ and changing the edges $(e_{\gamma^{-1}(u)}^D)^\pi$ and $(e_{\gamma^{-1}(v)}^D)^\pi$ which connect now the vertices $\{\gamma(v),-\tau(\gamma(v))\}$ to $\{-\gamma^{-1}(u),\tau(\gamma^{-1}(u))\}$ and the vertices $\{\gamma(u),-\tau(\gamma(u))\}$ to $\{-\gamma^{-1}(v),\tau(\gamma^{-1}(v))\}$ respectively. However in the graph $D^\pi$ these edges connect the vertices $\{u,-v\}$ to $\{-\gamma^{-1}(u),\tau(\gamma^{-1}(u))\}$ and the vertices $\{v,-u\}$ to $\{-\gamma^{-1}(v),\tau(\gamma^{-1}(v))\}$ respectively. The latter means that the cycle of $D^\pi$ that contains the edge $(e_v^D)^\pi$ has a path of the form
$$\{\gamma(v),-\tau(\gamma(v))\}\overset{(e_v^D)^\pi}{-} \{-v,u\} \overset{(e_{\gamma^{-1}(u)}^D)^\pi}{-} \{-\gamma^{-1}(u),\tau(\gamma^{-1}(u))\}.$$
While in the graph $D_B^{\pi^\prime}$ this path becomes
$$\{\gamma(v),-\tau(\gamma(v))\} \overset{(e_{\gamma^{-1}(u)}^D)^\pi}{-} \{-\gamma^{-1}(u),\tau(\gamma^{-1}(u))\}.$$
This proves that this cycle is unchanged, and we just decreased the number of vertices in the cycle by $1$. The same reasoning applies to the other cycle containing $(e_v^D)^\pi$ in $D^\pi$. So the number of cycles is unchanged which means $t(D_B^{\pi^\prime})=t(D^\pi)$, therefore,
$$t(D^\pi)=t(D_B^{\pi^\prime})\leq m/2+2-r.$$
\end{proof}

\begin{example}\label{Example: The quotient graph of D}
To illustrate the graphs that appear in the proof of Theorem \ref{Theorem: Order of GUE case for t()}, let us consider the following example. Let $r=2$ and $m_1=8$ and $m_2=6$ so that $m=14$. Let $\tau$ be the pairing
$$\{1,8\},\{2,13\},\{3,5\},\{4,7\},\{6,10\},\{9,14\},\{11,12\}.$$
Therefore,
\begin{eqnarray*}
\pi & \vcentcolon & = \pi_\tau=\{1,-8\},\{-1,8\},\{2,-13\},\{-2,13\},\{3,-5\},\{-3,5\},\{4,-7\}, \\
&& \{-4,7\},\{6,-10\},\{-6,10\},\{9,-14\},\{-9,14\},\{11,-12\},\{-11,12\}.
\end{eqnarray*}
The graph $G$ consists of two cycles while the graph $G^{\pi_\tau}$ has $14$ vertices indexed by the blocks of $\pi_\tau$. The graph $D^{\pi_\tau}$, which is the same as the graph $G^{\pi_\tau}$ restricted to the set of edges that label our deterministic matrices, has $5$ cycles given by
\begin{eqnarray*}
&& 1) \{1,-8\} \overset{e_8^{\pi_\tau}}{-} \{1,-8\}, \qquad  2) \{9,-14\} \overset{e_{14}^{\pi_\tau}}{-} \{9,-14\}, \qquad  3) \{-11,12\} \overset{e_{11}^{\pi_\tau}}{-} \{-11,12\}, \\
&& 4) \{-2,13\} \overset{e_{2}^{\pi_\tau}}{-} \{3,-5\} \overset{e_{5}^{\pi_\tau}}{-} \{6,-10\} \overset{e_{10}^{\pi_\tau}}{-} \{11,-12\} \overset{e_{12}^{\pi_\tau}}{-} \{-2,13\},\\
&& 5) \{-3,5\} \overset{e_{3}^{\pi_\tau}}{-} \{4,-7\} \overset{e_{7}^{\pi_\tau}}{-} \{-1,-8\} \overset{e_{1}^{\pi_\tau}}{-} \{2,-13\} \overset{e_{13}^{\pi_\tau}}{-} \{-9,14\} \overset{e_{9}^{\pi_\tau}}{-} \{-6,10\} \overset{e_{6}^{\pi_\tau}}{-} \{-4,7\} \overset{e_{4}^{\pi_\tau}}{-} \{-3,5\}.
\end{eqnarray*}
The graphs $G,D$ and $G^{\pi_\tau}$ and $D^{\pi_\tau}$ can be seen in Figure \ref{Figure: Quotient graph of G}.
\end{example}

\begin{figure}
    \centering
    \includegraphics[width=0.6\textwidth]{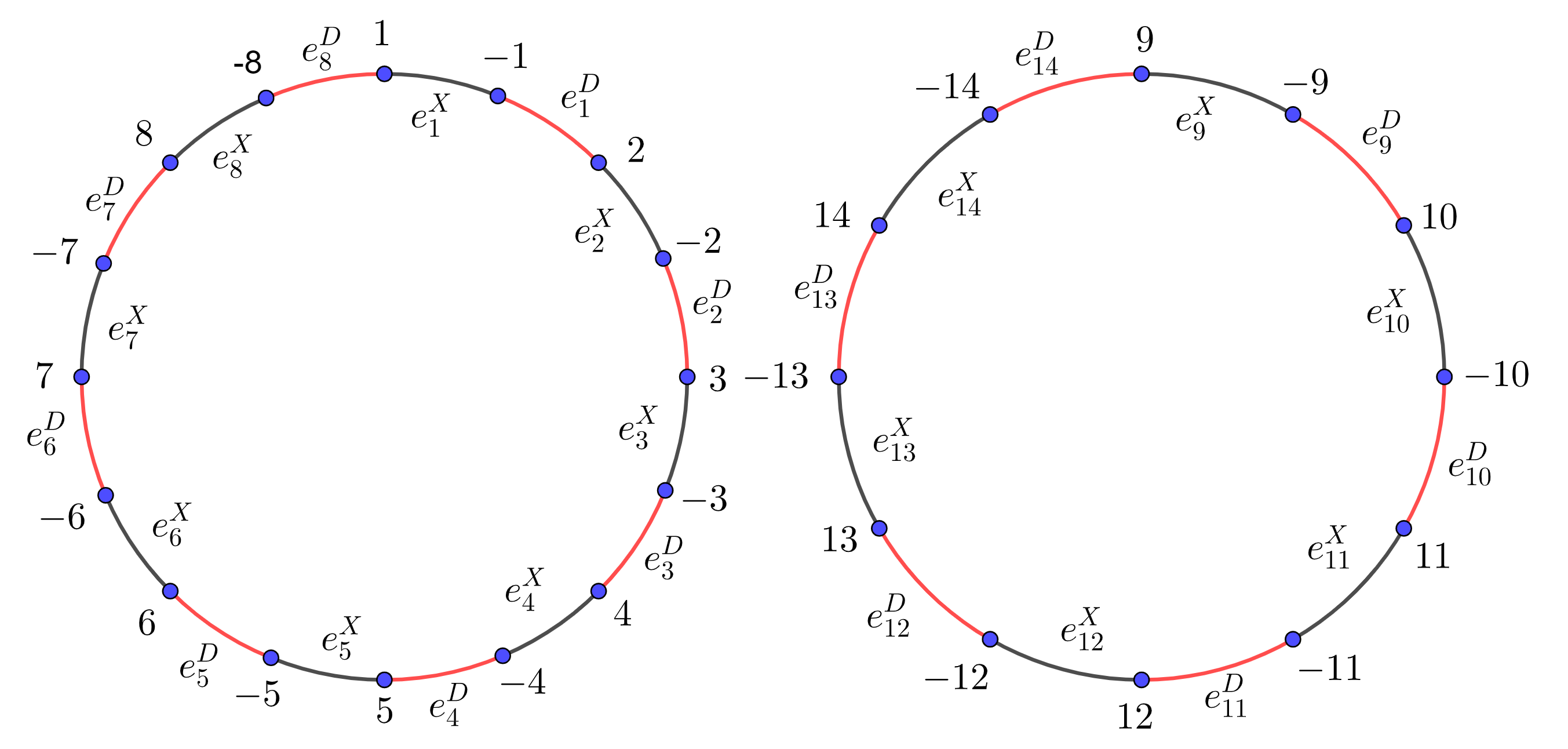}\\
    \text{a) The graph $G$.}\\
    \includegraphics[width=0.6\textwidth]{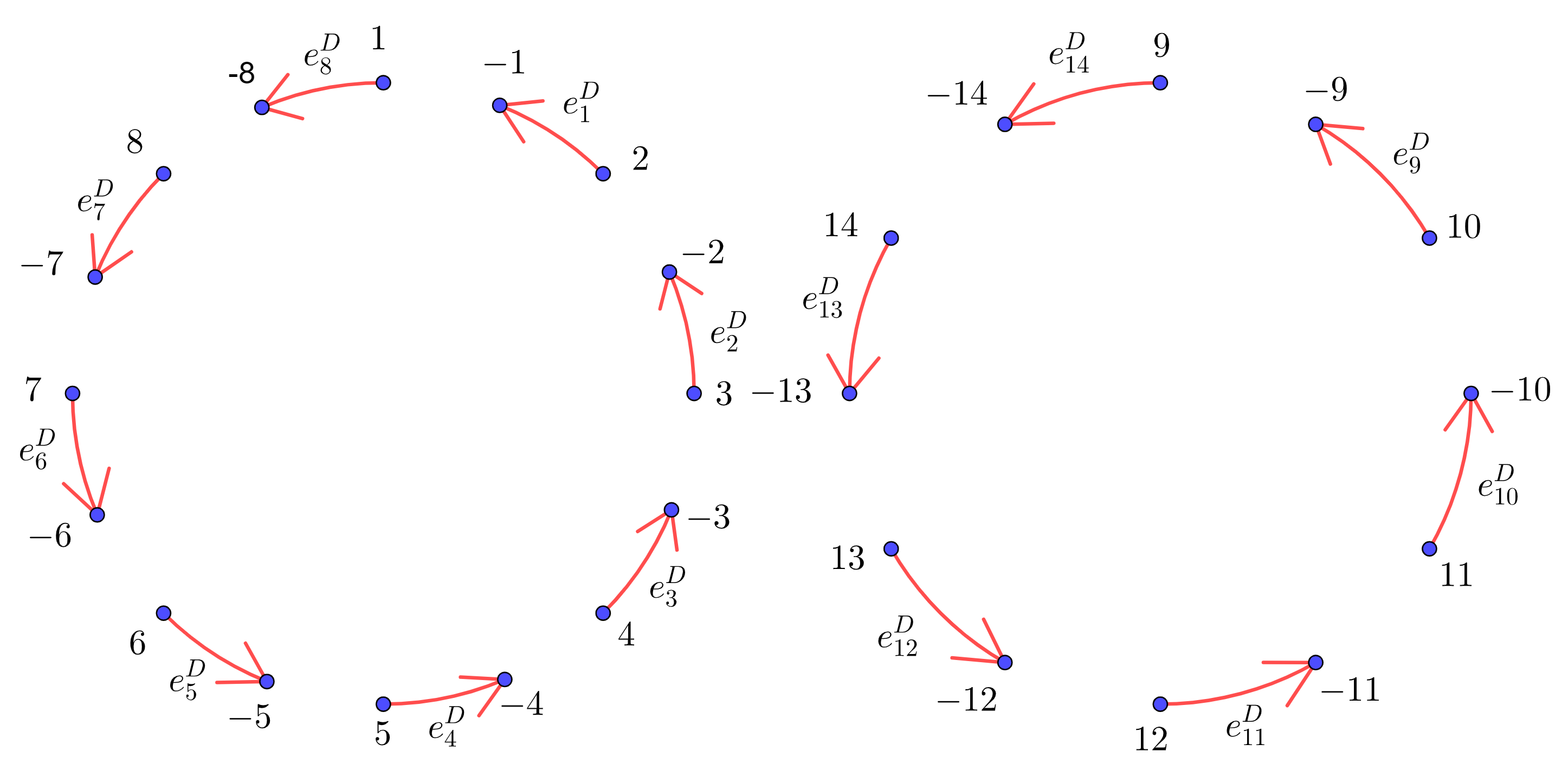}\\
    \text{b) The graph $D$.} \\
    \includegraphics[width=0.4\textwidth]{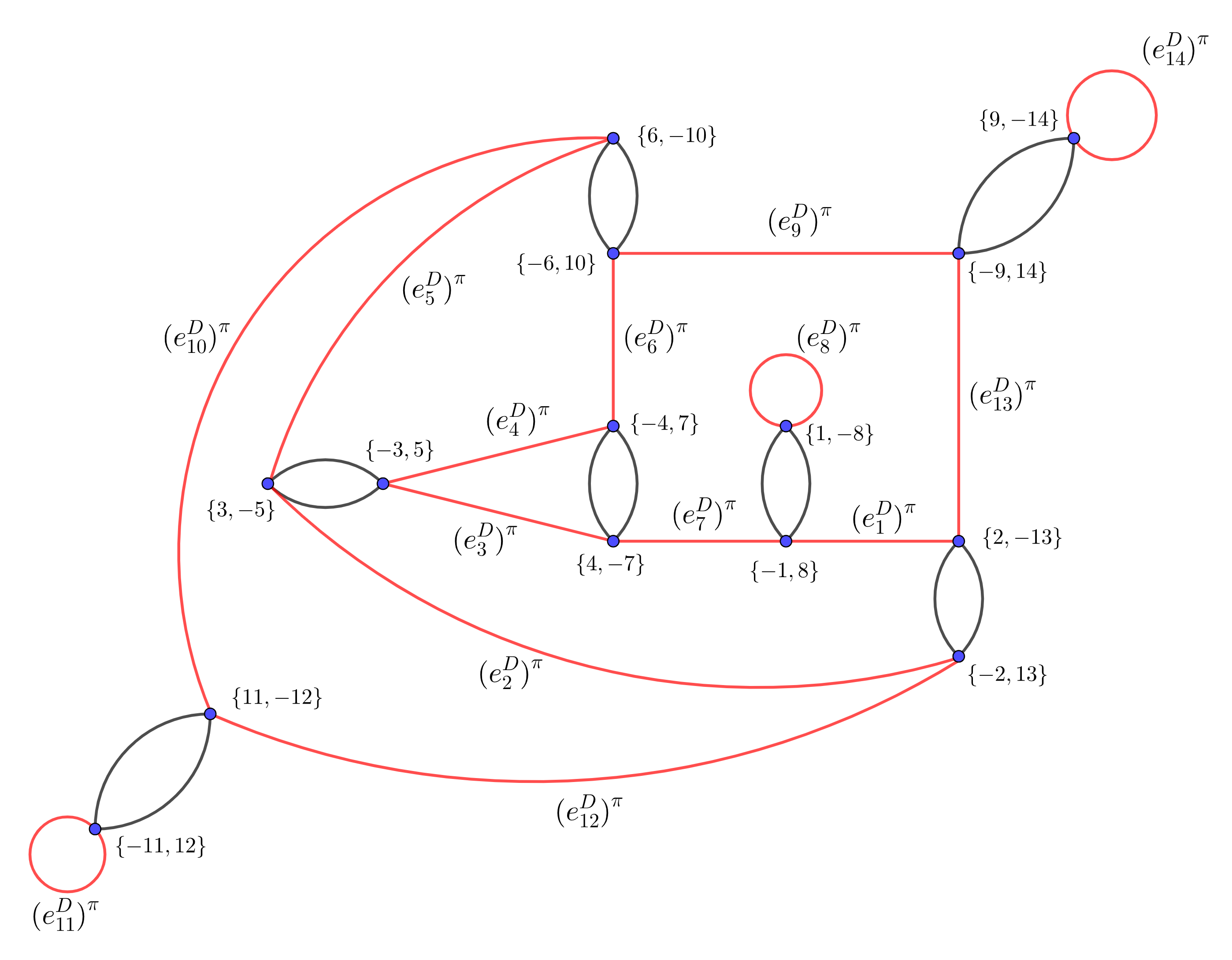}
    \includegraphics[width=0.4\textwidth]{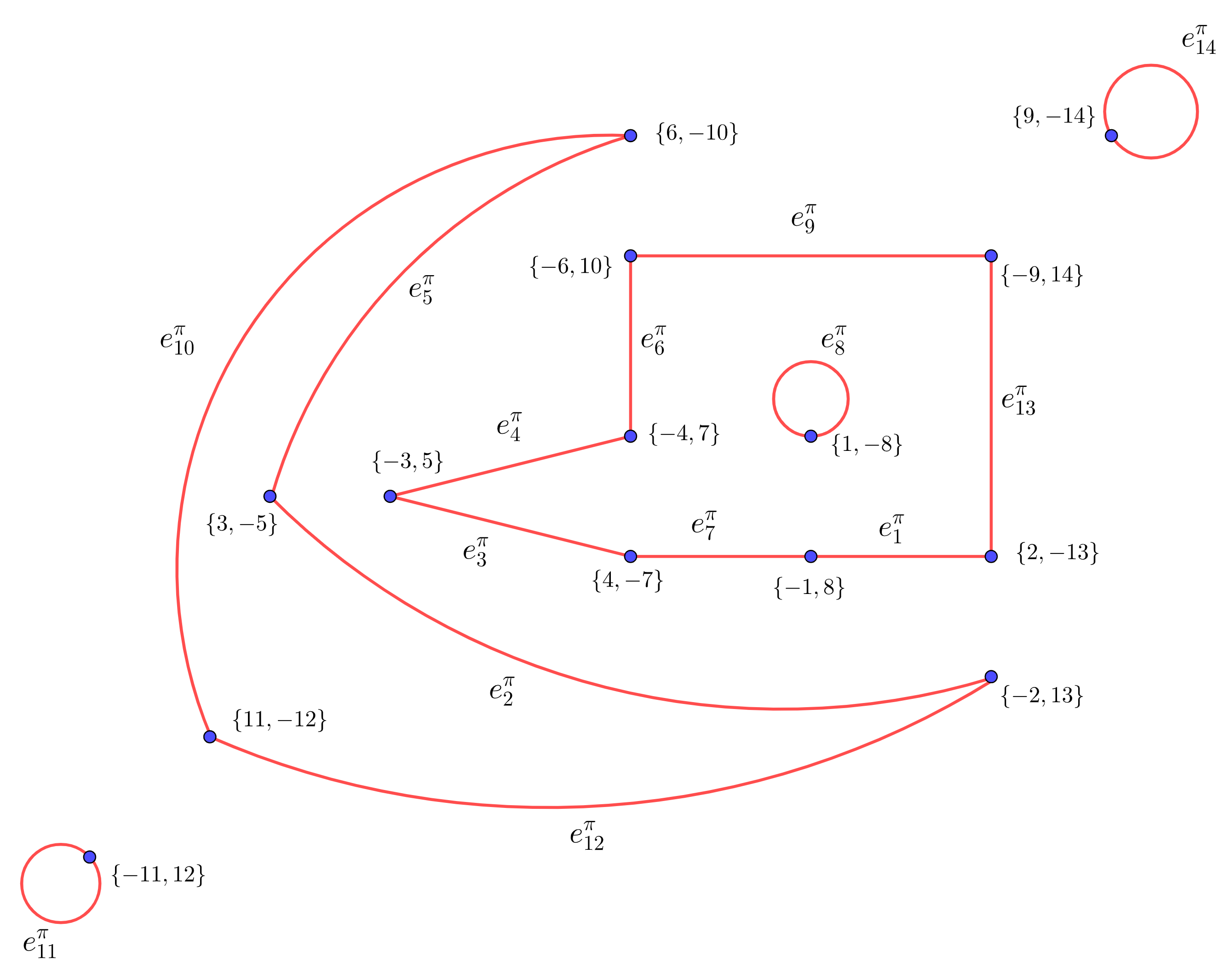}
    \text{c) The graphs $G^{\pi_\tau}$ and $D^{\pi_\tau}$.} \\
    \caption{The graphs of Example \ref{Example: The quotient graph of D}.}
    \label{Figure: Quotient graph of G}
\end{figure}

In the GOE case we use the fact that it can be written as $(Z_i+Z_i^{\top})/\sqrt{2}$ where $Z_i\sim GUE$. Because of this, we need to consider cumulants as in (\ref{Equation: Cumulants of monomials}) where we allow $X_i$ and $X_i^\top$. In the GUE case we are reduced to considering the partitions $\pi\geq \pi_\tau$, and therefore it is enough to bound (\ref{Equation: Deterministic sum asociated to Gaussian case depending on tau}). In this case, since we also have transposes, we need to consider partitions $\pi\geq \pi_\tau^\epsilon$ for some $\pi^\epsilon_\tau$ which will be defined below. Here $\epsilon$ is a vector parameter that describes whether we consider the transpose of a GUE. We call cumulants of the form (\ref{Equation: Cumulants of monomials}) but involving both $X_i$'s and the transposes of them the $\epsilon$-cumulants. We leave the detailed discussion on the reason of considering these partitions $\pi_\tau^\epsilon$ in the estimate of $\epsilon$-cumulants to Section \ref{Section: The order of GUE}. For this section we focus on the  combinatorics of the graphs. To this end, let us first formally introduce the definition of the partitions $\pi_\tau^\epsilon$.

\begin{notation}\label{Notation: definition of pi_tau for epsilon cumulants}
For an even number $m$, a pairing $\tau\in \cP_2(m)$ and $\epsilon\in \{1,\top\}^{m}$, we let $\pi_\tau^\epsilon$ be the pairing of $\cP_2(\pm [m])$ whose blocks are determined as follows: Whenever $\{u,v\}$ is a block of $\tau$ we let
\begin{enumerate}
    \item $\{-u,v\}$ and $\{v,-u\}$ be blocks of $\pi_\tau^\epsilon$ if either $\epsilon(u)=\epsilon(v)={\top}$ or $\epsilon(u)=\epsilon(v)=1$.
    \item $\{u,v\}$ and $\{-v,-u\}$ be blocks of $\pi_\tau^\epsilon$ if $\{\epsilon(u),\epsilon(v)\}=\{1,{\top}\}$.
\end{enumerate}
\end{notation}

The proof of Theorem \ref{Theorem: Order of GUE case for epsilon cumulants for t()} below will follow  the same ideas as the proof of Theorem \ref{Theorem: Order of GUE case for t()}, with the only difference that in this case our partition $\pi_\tau^\epsilon$ might be distinct to the partition $\pi_\tau$ and therefore further exploration is required for the quotient graph $D^{\pi_\tau^\epsilon}$. We will again use the double induction process.  For the induction on $r$, we again start from the case $r=1$. Then in order to prove the case $r=1$,  we proceed by induction on $m$. 

\begin{lemma}\label{Lemma: Order of GUE case base case for epsilon cumulants for t()}
Let $m$ be even, $\tau\in \cP_2(m)$, $\epsilon\in \{1,{\top}\}^m$ and $\pi_\tau^\epsilon$ as in Notation \ref{Notation: definition of pi_tau for epsilon cumulants}. Let $D$ be the graph of Notation \ref{Notation: The graphs D and G} with $r=1$ i.e., $\gamma=(1,\dots,m)$. Then,
$$t(D^{\pi_\tau^\epsilon})\leq m/2+1.$$
\end{lemma}
\begin{proof}
For brevity let us write $\pi=\pi_\tau^\epsilon$ within this proof. First, we apply induction on $m$. For a block $\{u,v\}\in \tau$ we know $\pi_\tau^\epsilon$ has blocks $\{u,-v\}$ and $\{v,-u\}$ or $\{u,v\}$ and $\{-u,-v\}$. The case $m=2$ means that either $\pi=\{1,-2\}\{2,-1\}$ or $\pi=\{1,2\}\{-1,-2\}$. In the former case $D^\pi$ has two cycles (each being a loop) while in the later case $D^\pi$ has one cycle. In either case $t(D^\pi)\leq 2=m/2+1$. Now we assume $t(D^\pi)\leq n$ where $m=2(n-1)$. We aim to prove it is true for $m=2n$. We proceed similarly as in Lemma \ref{Theorem: Order of GUE case for t() base case}. Let $B=\{u,v\}$ be a block of $\tau$. We define the graph $G_B$ as in Lemma \ref{Theorem: Order of GUE case for t() base case}. Note that the graph $G_B$ is a cycle with $2n-2$ edges. We let $\tau^\prime$ and $\pi^\prime$ be the pairing $\tau$ without the block $\{u,v\}$ and the pairing $\pi$ without the blocks that contains $\{u,v,-u,-v\}$, respectively. If $D_B$ is the graph $G_B$ restricted to the set of edges $e_k^D$, we know by induction hypothesis that
$$t(D_B^{\pi^\prime})\leq n.$$
As in Lemma \ref{Theorem: Order of GUE case for t() base case}, we claim $t(D^\pi)\leq t(D_B^{\pi^\prime})+1$. If $\{u,-v\}$ and $\{v,-u\}$ are blocks of $\pi$ then the proof follows exactly the same as in Lemma \ref{Theorem: Order of GUE case for t() base case}. So let us assume $\{u,v\}$ and $\{-u,-v\}$ are blocks of $\pi$. If $v=\gamma(u)$, observe that the cycle of $D^\pi$ that contains $e_{u}^\pi$ also contains $e_{\gamma^{-1}(u)}^\pi$ and $e_v^\pi$ and it has a path of the form
$$[-\gamma^{-1}(u)]_\pi\overset{e_{\gamma^{-1}(u)}^\pi}{-}\{u,v\}\overset{e_u^\pi}{-}\{-u,-v\}\overset{e_v^\pi}{-}[\gamma(v)]_\pi.$$
That is, $e_{\gamma^{-1}(u)}^\pi$ is adjacent to $e_u^\pi$ through the vertex $\{u,v\}$, and $e_u^\pi$ is adjacent to $e_v^\pi$ through the vertex $\{-u,-v\}$. Moreover, the edge $e_v^\pi$ is adjacent to the other vertex $[\gamma(v)]_{\pi}$. So when removing these vertices and the edges $e_u^\pi,e_v^\pi$ in $D_B^{\pi^\prime}$, we still have a cycle as the remaining edge $e_{\gamma^{-1}(u)}^\pi$ is by definition now connected to the vertex $[\gamma(v)]_{\pi}$. Thus $t(D_B^{\pi^\prime})=t(D^\pi)$. The case $u=\gamma(v)$ is analogous, so it remains to check the case $u\neq \gamma(v)$ and $v\neq \gamma(u)$. However, observe that this case proceeds exactly as in Lemma \ref{Theorem: Order of GUE case for t() base case}. Hence,
$$t(D^\pi)\leq t(D_B^{\pi^\prime})+1\leq n+1.$$
\end{proof}
Then we continue with the general $r$ case.
\begin{theorem}\label{Theorem: Order of GUE case for epsilon cumulants for t()}
Let $m$ be even, $\tau\in \cP_2(m)$ be such that $\tau\vee\gamma=1_m$, $\epsilon\in \{1,\top\}^m$ and $\pi_\tau^\epsilon$ as in Notation \ref{Notation: definition of pi_tau for epsilon cumulants}. Then,
$$t(D^{\pi_\tau^\epsilon})\leq m/2+2-r.$$
\end{theorem}
\begin{proof}
For brevity we write $\pi=\pi_\tau^\epsilon$ within this proof. We apply induction on $r$. The case $r=1$ is proved by Lemma \ref{Lemma: Order of GUE case base case for epsilon cumulants for t()} so we assume it is true for $r-1$ and prove it for $r$. Since $\tau\vee\gamma=1_m$, there exists a block of $B=\{u,v\}$ of $\tau$ such that $1\leq u\leq m_1$ and $v>m_1$. Let us assume that $m_1+1\leq v\leq m_1+m_2$. Again, given a block $\{u,v\}\in\tau$,  we know $\pi$ has blocks either $\{u,-v\}$ and $\{v,-u\}$ or $\{u,v\}$ and $\{-u,-v\}$. If $\{u,-v\}$ and $\{v,-u\}$ are blocks of $\pi$, we proceed as in Theorem \ref{Theorem: Order of GUE case for t()}. Then we are reduced to proving it for the case that $\{u,v\}$ and $\{-u,-v\}$ are blocks of $\pi$. In this case, we let $G_B$ be the graph with vertices $\pm[m]\setminus \{u,v,-u,-v\}$ and set of edges the same as $G$ except $e_u^X,e_v^X,e_v^D$ and $e_{\gamma^{-1}(u)}^D$ which are removed and the edges $e_u^D$ and $e_{\gamma^{-1}(v)}^D$ which are redefined as
\begin{eqnarray*}
e_u^D = \{\gamma(u),\gamma(v)\}, \qquad
e_{\gamma^{-1}(v)}^D = \{-\gamma^{-1}(u),-(v-1)\}.
\end{eqnarray*}
$G_B$ is now a graph with $r-1$ cycles and $2m-4$ edges. If $\tau^\prime$ is the pairing with blocks the same as $\tau$ except for $\{u,v\}$ then $\tau^\prime$ determines a partition $\pi^\prime$ whose blocks are the blocks of $\pi$ except for the blocks $\{u,v\}$ and $\{-u,-v\}$. By induction hypothesis we have,
$$t(D_B^{\pi^\prime})\leq \frac{m-2}{2}+2-r+1=m/2+2-r.$$
Now we conclude as in Theorem  \ref{Theorem: Order of GUE case for t()}. To get the graph $G_B^{\pi^\prime}$, we remove the vertices $\{u,v\}$ and $\{-u,-v\}$ of $G_B^\pi$ and the edges $e_{\gamma^{-1}(u)}^\pi$ and $e_v^\pi$. And we also change the edges $e_u^\pi$ and $e_{\gamma^{-1}(v)}^\pi$ according to how they were redefined. In $G^{\pi}$ the edge $e_u^\pi$ connects the vertices $[\gamma(u)]_{\pi}$ and $\{-u,-v\}$, however in $G_B^{\pi^\prime}$ it connects $[\gamma(u)]_{\pi}$ and $[\gamma(v)]_{\pi}$ instead. The cycle that contains $e_u^\pi$ in $G^\pi$ also contains $e_v^\pi$, and these are consecutive edges in the cycle both connected to the same $\{-u,-v\}$ vertex, while the other adjacent vertex of $e_v^\pi$ is precisely $[\gamma(v)]_{\pi}$. Hence this cycle is still a cycle in $G_B^{\pi^\prime}$. The same applies to the cycle that contains $e_{\gamma^{-1}(u)}^\pi$ and $e_{\gamma^{-1}(v)}^\pi$. Therefore $t(D_B^{\pi^\prime})=t(D^{\pi})$ which implies
$$t(D^\pi)=t(D_B^{\pi^\prime})\leq m/2+2-r.$$
\end{proof}

\subsection{The Wigner case} \label{s. Wigner case}

Unlike the Gaussian case for which it is enough to bound quantities of the form (\ref{Equation: Deterministic sum asociated to Gaussian case depending on tau}),  in the Wigner case,  we need to look into quantities like (\ref{Equation: Deterministic sum asociated to Gaussian case depending on tau when equality}) for each individual $\pi$. In order to find upper bounds for (\ref{Equation: Deterministic sum asociated to Gaussian case depending on tau when equality}) we appeal to Mobius inversion theorem and Lemma \ref{Lemma: Order of sums associated to deterministic matrices}. We explain the details in Section \ref{Section: The order of Wigner}, for this subsection we focus in finding upper bounds for $t(D^{\pi})$ which will be needed in the bound of (\ref{Equation: Deterministic sum asociated to Gaussian case depending on tau when equality}).

\begin{theorem}\label{Theorem: Order of Wigner case for t()}
Assume $r\geq 2$. Let $\pi\in \cP(\pm[m])$ and $\tau\in \cP(m)$ be such that $\tau\vee\gamma=1_m$ and
$$\C_\tau(x_{\psi^\pi(1),\psi^\pi(-1)}^{(i_1)},\dots,x_{\psi^\pi(m),\psi^\pi(-m)}^{(i_m)})\neq 0.$$
Let $D$ be the graph of Notation \ref{Notation: The graphs D and G}. Then,
$$t(D^{\pi})\leq m/2+1-r/2.$$
\end{theorem}

\begin{proof}
Let $G$ be the graph defined in Notation \ref{Notation: The graphs D and G}. Let $V=\{j_1,\dots,j_{2n}\}$ be a block of $\tau$. Since $\C_\tau(x_{\psi^\pi(1),\psi^\pi(-1)}^{(i_1)},\dots,\ab x_{\psi^\pi(m),\psi^\pi(-m)}^{(i_m)})\neq 0$, the variables $x^{i_{j_1}}_{\psi^\pi(j_1),\psi^\pi(-j_1)},\dots,x^{i_{j_{2n}}}_{\psi^\pi(j_{2n}),\psi^\pi(-j_{2n})}$ must  be all the same or the complex conjugate of each other, otherwise by independence of the entries and vanishing of mixed cumulants the cumulant is zero. The latter means that either $\pi$ has a block that contains $j_1,\dots,j_{2n},-j_1,\dots,-j_{2n}$, or $\pi$ has two blocks $A_V$ and $B_V$ such that for each $1\leq u\leq 2n$ either $j_u\in A_V$ and $-j_u\in B_V$ or $j_u\in B_V$ and $-j_u\in A_V$. In either case we want to prove that $t(D^\pi)\leq m/2+1-r/2$. We may assume that we are in the latter case, as the former case will follow from Lemma \ref{Lemma: Order doesn't increase when doing quotient}. Let us assume that for any block $V\in \tau$, $A_V$ always contains the positive values $j_1,\dots,j_{2n}$ while $B_V$ contains the negative ones $-j_1,\dots,-j_{2n}$. We will explain at the end why the same proof applies in other cases. By Proposition \ref{Proposition: the existence of the pairing sigma}, there exists a pairing $\sigma$ such that $\sigma\leq \tau$ and $s\leq r/2$, where $s=\#(\sigma\vee\gamma)$. Let $\pi_\sigma$ be the partition defined by the following: if $\{u,v\}\in \sigma$ is a block of $\sigma$ then we let $\{u,v\}$ and $\{-u,-v\}$ be blocks of $\pi_\sigma$. It is clear $\pi_\sigma \leq \pi$. 

Let $B_1,\dots,B_s$ be the blocks of $\sigma\vee\gamma$. For $1\leq i\leq s$, let $D_i$ be the graph $D$ restricted to the set of vertices $\pm B_i=:\{a,-a: a\in B_i\}$ and set of edges $e_k$ with $k\in B_i$. Each block $B_i$ is a union of blocks of $\sigma$ and $\gamma$, so we let $\sigma_i$ and $\gamma_i$ be $\sigma$ and $\gamma$ restricted to $B_i$ respectively. For each block $B_i$, we have $\sigma_i \vee \gamma_i =1_{|B_i|}$. The pairing $\sigma_i$ is a pairing of the matrix edges $e_k^X$ of $G_i$, where $G_i$ is the restriction of $G$ to $\pm B_i$. For each block $\{u,v\}$ of $\sigma_i$, the corresponding blocks of $\pi_{\sigma_i}$ are $\{u,v\}$ and $\{-u,-v\}$. Hence, thanks to Theorem \ref{Theorem: Order of GUE case for epsilon cumulants for t()}, we know 
\begin{equation}\label{Aux: Equation 1}
t(D_i^{\pi_{\sigma_i}})\leq m_i/2+2-r_i,
\end{equation}
where $r_i$ is the number of  cycles of $B_i$ and $m_i=|B_i|$. Let $T=(V_T,E_T,\{V(e)\}_{e\in E_T})$ be the graph of the graphs $D_1^{\pi_{\sigma_1}},\dots,D_s^{\pi_{\sigma_s}}$ given as follows:
\begin{enumerate}
    \item If $\{u_1,v_1\}$ and $\{u_2,v_2\}$ are blocks of $\sigma$ in the same block of $\tau$ with $\{u_i,v_i\}\in B_{j_i}$ for $i=1,2$, we let $e_1$ and $e_2$ be edges of $T$ connecting the vertices $D_{j_1}^{\pi_{\sigma_{j_1}}}$ and $D_{j_2}^{\pi_{\sigma_{j_2}}}$ of $T$.
    \item $V(e_1)=\{\{u_1,v_1\},\{u_2,v_2\}\}$, which is well defined  as $\{u_1,v_1\}\in \pi_{\sigma_{j_1}}$ and hence it is a vertex of $D_{j_1}^{\pi_{\sigma_{j_1}}}$. Similarly $\{u_2,v_2\}$ must be a vertex of $D_{j_2}^{\pi_{\sigma_{j_2}}}$.
    \item $V(e_2)=\{\{-u_1,-v_1\},\{-u_2,-v_2\}\}$,  which is well defined as the previous item. 
\end{enumerate}
Let $D=\cup_{i=1}^s D_i^{\pi_{\sigma_i}}$. Observe that for blocks $\{u_1,v_1\},\dots ,\{u_n,v_n\}$ of $\sigma$ in the same block of $\tau$ the partition $\pi_T$ has blocks $\{u_1,v_1,\cdots ,u_n,v_n\}$ and $\{-u_1,-v_1,\cdots ,-u_n,-v_n\}$, hence $\pi_T\leq \pi$. From Lemma \ref{Lemma: Order doesn't increase when doing quotient} it follows $t(D^\pi)\leq t(D^{\pi_T})=t(D^T)$. Moreover, since $\tau\vee\gamma=1_m$, then $T$ must be connected. Hence it follows from Lemma \ref{Corollary: Order of exponent when inducing graph is connected} that
$$t(D^{\pi})\leq t(D^{T})\leq \sum_{i=1}^s t(D_i^{\pi_{\sigma_i}})-s+1.$$
From (\ref{Aux: Equation 1}) we get
\begin{eqnarray*}
t(D^{\pi}) &\leq &  \sum_{i=1}^s [m_i/2+2-r_i]-s+1 \\
& = & m/2+2s-r-s+1\leq  m/2+1-r/2. 
\end{eqnarray*}
To finish the proof,  we observe that,  if $A_V$ and $B_V$ are more general in the sense that they do not consist of purely negative or positive values,  we only need to define $\pi_\sigma$ accordingly so that $\pi_\sigma\leq \pi$. Let us remind that in such a general setting of $A_V$ and $B_V$,   if $\{u,v\}$ is a block of $\sigma$ then $\pi_\sigma$ may have blocks $\{u,v\},\{-u,-v\}$ or $\{u,-v\},\{v,-u\}$. In both cases,  inequality (\ref{Aux: Equation 1}) is true. Then the remaining proof is the same as the special case of $A_V$ and $B_V$ discussed earlier.
\end{proof}

\section{Upper bounds for the cumulants}\label{Section: The order of GUE}

 With the aid of the upper bounds derived for $t(\cdot)$ in the last section, we can then prove our upper  bounds for the cumulants of the form (\ref{Equation: Cumulants of monomials}). This section is organized as follows. In Section \ref{Subsection: Nth order of GUE}, we derive a general upper bound for (\ref{Equation: Cumulants of monomials}) in the GUE case. In Section \ref{Subsection: uppper bound for epsilon cumulants}, we go one step further to derive the same upper bound, but allowing both GUE and its transpose in (\ref{Equation: Cumulants of monomials}). We then use this result to find upper bounds for the GOE case in Section \ref{Section: The order of GOE}. Finally, in Section \ref{Section: The order of Wigner}, we derive a general upper bound for the cumulants in the Wigner case. 
 
\subsection{General bound of the cumulants of the GUE case}\label{Subsection: Nth order of GUE}

From the definition of the high order moments and cumulants of a random matrix (cf. \cite[Section 2.3]{CMSS}),  we expect $\C_r(\text{Tr}(Y_1),\dots,\text{Tr}(Y_r))$ to be  of order $N^{2-r}$ when $r$ is fixed. This is known for a single GUE matrix. In this subsection we prove that the result still holds for an arbitrary polynomial of GUE's and deterministic matrices and further capture more precise $r$-dependence when $r$ could be $N$-dependent. 

Recall the definition of $\gamma$ from Notation \ref{def of gamma}, we have the following theorem.

\begin{lemma}\label{Corollary: Upper bound for cumulants of GUE}
Let $\mathcal{X}\sim GUE$. For any $r\geq 1$ and $m$ even,
\begin{equation}
|N^{r-2}\C_r(\text{Tr}(Y_1),\dots,\text{Tr}(Y_r))|\leq|\cP_2^i(m)| \prod_{k=1}^m ||D_{j_k}||,
\end{equation}
where
$$\cP_2^i(m)=\{\tau\in \cP_2(m): \tau\vee\gamma=1_m\text{ and }\tau\leq ker(i)\}.$$
In particular,
\begin{equation}
|N^{r-2}\C_r(\text{Tr}(Y_1),\dots,\text{Tr}(Y_r))|\leq m!! \prod_{k=1}^m ||D_{j_k}||.
\end{equation}
\end{lemma}
\begin{proof}
From (\ref{Equation: Cumulant expansion}), we have
\begin{eqnarray*}
&&N^{r-2}\C_r(\text{Tr}(Y_1),\dots,\text{Tr}(Y_r)) = \\
&& \sum_{\substack{\tau\in \cP_2(m) \\ \tau\vee\gamma=1_m \\ \tau \leq ker(i)}}\sum_{\substack{\pi\in \cP(\pm[m]) \\ \pi\geq \pi_\tau}}N^{r-2-m/2} \left[\sum_{\substack{\psi: \pm[m]\rightarrow [N] \\ ker(\psi)=\pi}}\mathbf{D}(\psi)\right]\C_\tau(x_{\psi^\pi(1),\psi^\pi(-1)}^{(i_1)},\dots,x_{\psi^\pi(m),\psi^\pi(-m)}^{(i_m)}),
\end{eqnarray*}
because if $\C_\tau(\cdot)\neq 0$ we must have $\tau \leq ker(i)$, $\tau\in \cP_2(m)$,  and $x_{\psi^\pi(u),\psi^\pi(-u)}^{(i_u)}=x_{\psi^\pi(-v),\psi^\pi(v)}^{(i_v)}$ for $\{u,v\}\in \tau$. Hence, $[u]_\pi=[-v]_\pi$ and $[v]_\pi=[-u]_\pi$, which implies $\pi\geq \pi_\tau$ with $\pi_\tau$ as in Notation \ref{Notation: definition of pi_tau}. Further, if $\C_\tau(\cdot)\neq 0$ we have $\C_\tau(\cdot)=1$.  Therefore, 
\begin{eqnarray*}
N^{r-2}\C_r(\text{Tr}(Y_1),\dots,\text{Tr}(Y_r)) & = & \sum_{\substack{\tau\in \cP_2(m) \\ \tau\vee\gamma=1_m \\ \tau\leq ker(i)}}\sum_{\substack{\pi\in \cP(\pm[m]) \\ \pi\geq \pi_\tau}}N^{r-2-m/2} \sum_{\substack{\psi: \pm[m]\rightarrow [N] \\ ker(\psi)=\pi}}\mathbf{D}(\psi) \\
& = & \sum_{\substack{\tau\in \cP_2(m) \\ \tau\vee\gamma=1_m \\ \tau\leq ker(i)}}N^{r-2-m/2} \sum_{\substack{\psi: \pm[m]\rightarrow [N] \\ ker(\psi)\geq \pi_\tau}}\mathbf{D}(\psi).
\end{eqnarray*}
By Theorem \ref{Theorem: Order of GUE case for t()} and Lemma \ref{Lemma: Order of sums associated to deterministic matrices}, we know 
$$\left|N^{r-2-m/2} \sum_{\substack{\psi: \pm[m]\rightarrow [N] \\ ker(\psi)\geq\pi_\tau}}\mathbf{D}(\psi)\right| \leq N^{r-2-m/2}N^{m/2+2-r}\prod_{k=1}^m ||D_{j_k}||=\prod_{k=1}^m ||D_{j_k}||.$$ 
Hence, we can derive the upper bound
\begin{eqnarray*}
|N^{r-2}\C_r(\text{Tr}(Y_1),\dots,\text{Tr}(Y_r))| & \leq & 
\sum_{\substack{\tau\in \cP_2(m) \\ \tau\vee\gamma=1_m \\ \tau\leq ker(i)}}\left|N^{r-2-m/2} \sum_{\substack{\psi: \pm[m]\rightarrow [N] \\ ker(\psi)\geq \pi_\tau}}\mathbf{D}(\psi)\right| \\
& \leq & \sum_{\substack{\tau\in \cP_2(m) \\ \tau\vee\gamma=1_m \\ \tau\leq ker(i)}} \prod_{k=1}^m ||D_{j_k}|| = |\cP_2^i(m)|\prod_{k=1}^m ||D_{j_k}||.
\end{eqnarray*}
This concludes the proof. 
\end{proof}

\subsection{General bound of the $\epsilon$-cumulants}\label{Subsection: uppper bound for epsilon cumulants}

In Notation \ref{Notation: Y_k}, we consider the alternating product of GUE and deterministic matrices. In this subsection we go one step further. We now consider the alternating product of GUE, or its transpose and deterministic matrices. As we mentioned earlier, we will call cumulants of the form (\ref{Equation: Cumulants of monomials}) but involving both $X_i$'s and the transposes of them the $\epsilon$-cumulants.  The motivation to consider the $\epsilon$-cumulants is to provide upper bounds as in Lemma \ref{Corollary: Upper bound for cumulants of GUE} for the GOE case. This type of cumulants will appear naturally in Section \ref{Section: The order of GOE}. Let us give a formal definition of the $\epsilon$-cumulants. 

\begin{definition}\label{Definition: epsilon cumulants}
For $1\leq k\leq r$ and $\epsilon\in \{1,{\top}\}^m$, we let
$$Y_{k}^\epsilon \vcentcolon=X_{i_{M_{k-1}+1}}^{\epsilon(M_{k-1})}D_{j_{M_{k-1}+1}}\cdots X_{i_{M_k}}^{\epsilon(M_{k})}D_{j_{M_k}},$$
with the convention $M_0=0$. Here $\epsilon=(\epsilon(1),\dots,\epsilon(m))\in \{1,{\top}\}^m$. We call $\C_r(Tr(Y_1^{\epsilon}),\dots, Tr(Y_r^\epsilon))$ an $\epsilon$-cumulant. 
\end{definition}

The following lemma generalizes (\ref{Equation: Cumulant expansion}).

\begin{lemma}\label{Lemma: Cumulant expansion for epsilon cumulants}
Let $\mathcal{X}\sim GUE$. For any $r\geq 1$ the following holds
\begin{eqnarray}\label{Equation: Cumulant expansion for epsilon cumulants}
&&\C_r(Tr(Y_1^{\epsilon}),\dots, Tr(Y_r^\epsilon)) = \nonumber\\
&& N^{-m/2}\sum_{\pi\in \cP(\pm[m])}\sum_{\substack{\tau\in \cP(m) \\ \tau\vee\gamma=1_m}}\left[\sum_{\substack{\psi: \pm[m]\rightarrow [N] \\ ker(\psi)=\pi}}\mathbf{D}(\psi)\right]\C_\tau(x_{\psi^\pi(1),\psi^\pi(-1)}^{(i_1)}(\epsilon(1)),\dots,x_{\psi^\pi(m),\psi^\pi(-m)}^{(i_m)}(\epsilon(m))),
\end{eqnarray}
where $x_{i,j}^{(k)}(1)=x_{i,j}^{(k)}$ and $x_{i,j}^{(k)}({\top})=x_{j,i}^{(k)}$.
\end{lemma}
\begin{proof}
The proof is nearly the same as (\ref{Equation: Cumulant expansion}). The only difference is that if there is a transpose then we take the $(i,j)$-entry of $X^{\top}$ which is the same as taking the $(j,i)$-entry of $X$.
\end{proof}

Our goal is to find an upper bound for the absolute value of the $\epsilon$-cumulants. The proof strategy to follow will be the same as the regular cumulants.

The following corollary is a direct consequence of Theorem \ref{Theorem: Order of GUE case for epsilon cumulants for t()} and (\ref{Equation: Cumulant expansion for epsilon cumulants}). Its proof is the same as that for Lemma \ref{Corollary: Upper bound for cumulants of GUE}, and is therefore omitted.

\begin{corollary}\label{Corollary: Upper bound for cumulants of GUE for epsilon cumulants}
Let $\mathcal{X}\sim GUE$. For any $r\geq 1$, $m$ even and $\epsilon\in \{1,{\top}\}^m$,
\begin{equation}
|N^{r-2}\C_r(\text{Tr}(Y_1^\epsilon),\dots,\text{Tr}(Y_r^\epsilon))|\leq|\cP_2^i(m)| \prod_{k=1}^m ||D_{j_k}||,
\end{equation}
where
$$\cP_2^i(m)=\{\tau\in \cP_2(m): \tau\vee\gamma=1_m\text{ and }\tau\leq ker(i)\}.$$
In particular,
\begin{equation}
|N^{r-2}\C_r(\text{Tr}(Y_1^\epsilon),\dots,\text{Tr}(Y_r^\epsilon))|\leq m!! \prod_{k=1}^m ||D_{j_k}||.
\end{equation}
\end{corollary}

\subsection{General bounds of the cumulants of the GOE case}\label{Section: The order of GOE}

In order to get upper bounds for the quantities (\ref{Equation: Cumulants of monomials}) where $\mathcal{X}\sim GOE$,  we write each GOE random matrix as $(Z_i+Z_i^{\top})/\sqrt{2}$,  where $Z_i\sim GUE$. Then we are reduced to the upper bounds obtained in Section \ref{Subsection: uppper bound for epsilon cumulants}, by multi-linearity of cumulants. More precisely, we get the following bounds.

\begin{lemma}\label{Corollary: Upper bound for cumulants of GOE}
Suppose $\mathcal{X}\sim GOE$. For any $r\geq 1$ and $m$ even, we have
\begin{equation}
|N^{r-2}\C_r(\text{Tr}(Y_1),\dots,\text{Tr}(Y_r))|\leq 2^{m/2}|\cP_2^i(m)| \prod_{k=1}^m ||D_{j_k}||.
\end{equation}
Here $\cP_2^i(m)$ is defined as in Lemma \ref{Corollary: Upper bound for cumulants of GUE}. In particular,
\begin{equation}
|N^{r-2}\C_r(\text{Tr}(Y_1),\dots,\text{Tr}(Y_r))|\leq 2^{m/2}m!! \prod_{k=1}^m ||D_{j_k}||.
\end{equation}
\end{lemma}
\begin{proof}
For each $i$ we write $X_i=\frac{1}{\sqrt{2}}(Z_i+Z_i^{\top})$ where $Z_i\sim GUE$. Hence for $1\leq k\leq r$,
\begin{eqnarray*}
Y_k & = & \frac{1}{2^{m_r/2}}(Z_{i_{M_{k-1}+1}}+Z_{i_{M_{k-1}+1}}^{\top})D_{j_{M_{k-1}+1}}\cdots (Z_{i_{M_k}}+Z_{i_{M_k}}^{\top})D_{j_{M_k}} \\
& = & \frac{1}{2^{m_r/2}}\sum_{\epsilon(M_{k-1}+1),\dots,\epsilon(M_k)\in \{1,{\top}\}}Z_{i_{M_{k-1}+1}}^{\epsilon(M_{k-1})}D_{j_{M_{k-1}+1}}\cdots Z_{i_{M_k}}^{\epsilon(M_{k})}D_{j_{M_k}}.
\end{eqnarray*}
Therefore
\begin{eqnarray*}
N^{r-2}\C_r(\text{Tr}(Y_1),\dots,\text{Tr}(Y_r)) & = & \frac{1}{2^{m/2}}\sum_{\epsilon\in \{1,{\top}\}^m} \C_r(\text{Tr}(Y_1^\epsilon(Z)),\dots,\text{Tr}(Y_r^\epsilon(Z))),
\end{eqnarray*}
where $Y_k^\epsilon(Z)$ is the same as $Y_k^\epsilon$ defined in Definition \ref{Definition: epsilon cumulants} with the $X_i$ matrices being distributed as $GUE$. From Corollary \ref{Corollary: Upper bound for cumulants of GUE for epsilon cumulants} it follows,
\begin{eqnarray*}
|N^{r-2}\C_r(\text{Tr}(Y_1),\dots,\text{Tr}(Y_r))| & \leq & \frac{1}{2^{m/2}}\sum_{\epsilon\in \{1,{\top}\}^m} |\C_r(\text{Tr}(Y_1^\epsilon(Z)),\dots,\text{Tr}(Y_r^\epsilon(Z)))| \\
& \leq & \frac{1}{2^{m/2}}\sum_{\epsilon\in \{1,{\top}\}^m}|\cP_2^i(m)| \prod_{k=1}^m ||D_{j_k}|| \\
& = & 2^{m/2}|\cP_2^i(m)| \prod_{k=1}^m ||D_{j_k}||.
\end{eqnarray*}
This concludes the proof. 
\end{proof}

\subsection{General bounds of the cumulants of the Wigner case}\label{Section: The order of Wigner}

In this section, we consider general Wigner matrices with subexponential entries. As shown in  Example \ref{Example: Order of Wigner case} at the end of the section, in the Wigner case the $N$-order of the cumulants is different from the Gaussian case in general. More specifically, we show that  the order of the $N$-factor in the cumulant bound is $N^{1-r/2}$, which is square root of the order in the Gaussian case.

\begin{theorem}\label{Theorem: Order of Wigner case}
Let $\mathcal{X}\sim Wigner$. Assume $r\geq 2$. Let $\pi\in \cP(\pm[m])$ and $\tau\in \cP(m)$ be such that $\tau\vee\gamma=1_m$ and
$$\C_\tau(x_{\psi^\pi(1),\psi^\pi(-1)}^{(i_1)},\dots,x_{\psi^\pi(m),\psi^\pi(-m)}^{(i_m)})\neq 0.$$
Then,
$$\left| \sum_{\substack{\psi: \pm[m]\rightarrow [N] \\ ker(\psi)=\pi}}\mathbf{D}(\psi)\right| \leq N^{m/2+1-r/2}\prod_{k=1}^m ||D_{j_k}|| \#(\pi)! |p(\#(\pi))|,$$
where $p(m)$ is the set of integer partitions of $m$. 
\end{theorem}
\begin{proof}
We denote
\begin{equation}\label{Equation: Sum associated to product of deterministic matrices with equality}
S^0_{\pi}(N) \vcentcolon= \sum_{\substack{\psi: \pm[m]\rightarrow [N] \\ ker(\psi)=\pi}}\mathbf{D}(\psi).
\end{equation}
First, let us prove a  result similar to Lemma \ref{Lemma: Order of sums associated to deterministic matrices}. We claim
\begin{equation}\label{Aux: Inequality for inductive trace}
|S_\pi^0(N)|\leq N^{t(D^\pi)}{{\prod_{k=1}^m ||D_{j_k}||}} \#(\pi)!|p(\#(\pi))|.
\end{equation}
Here $D^\pi$ is the quotient graph of $D$ defined as in Notation \ref{Notation: The graphs D and G} and $t(D^\pi)$ is defined as in Lemma \ref{Lemma: Order of sums associated to deterministic matrices}. From \cite[Lemma 2.6]{M}, we have
\begin{eqnarray*}
S_\pi^0(N) = \sum_{\pi^\prime\in \cP(\pi)}\prod_{B\in \pi^\prime} (-1)^{|B|-1}(|B|-1)!S_{\pi^\prime}(N),
\end{eqnarray*}
where $\cP(\pi)$ denotes the set of partitions of the blocks of $\pi$. From Lemma \ref{Lemma: Order of sums associated to deterministic matrices}, we know $|S_{\pi^\prime}(N)|\leq N^{t(D^{\pi^\prime})}\prod_{k=1}^m ||D_{j_k}||$ for any $\pi^\prime$. Moreover from Lemma \ref{Lemma: Order doesn't increase when doing quotient}, we have $t(D^{\pi^\prime})\leq t(D^\pi)$. Hence
\begin{eqnarray*}
|S_\pi^0(N)| &\leq & \sum_{\pi^\prime\in \cP(\pi)}\prod_{B\in \pi^\prime}(|B|-1)!|S_{\pi^\prime}(N)| \\
& \leq & \sum_{\pi^\prime\in \cP(\pi)}\prod_{B\in \pi^\prime}(|B|-1)!N^{t(D^\pi)}\prod_{k=1}^m ||D_{j_k}|| \\
&=& N^{t(D^\pi)}\prod_{k=1}^m ||D_{j_k}||\sum_{\pi^\prime\in \cP(\pi)}\prod_{B\in \pi^\prime}(|B|-1)!.
\end{eqnarray*}
To prove (\ref{Aux: Inequality for inductive trace}), it suffices to prove that for an integer $n$, $\sum_{\pi\in \cP(n)}\prod_{B\in \pi}(|B|-1)! \leq n!|p(n)|$. Indeed, the partitions $\cP(n)$ can be counted by $B_n(1,\dots,1)$ where $B_n(x_1,\dots,x_n)$ is the Bell polynomial (see \cite[Section 3.3]{Comtet}), which is defined as
$$B_n(x_1,\dots,x_n)=n!\sum_{\substack{j_1,\dots,j_n \\ j_1+\cdots +nj_n=n}}\prod_{i=1}^n \frac{x_i^{j_i}}{i!^{j_i}j_i!}.$$
In the above sum, the index $j_i$ determines how many partitions there are with $j_i$ blocks of size $i$. Hence
\begin{eqnarray*}
\sum_{\pi\in \cP(n)}\prod_{B\in \pi}(|B|-1)! & \leq & \sum_{\pi\in \cP(n)}\prod_{B\in \pi}|B|! = n!\sum_{\substack{j_1,\dots,j_n \\ j_1+\cdots +nj_n=n}}\prod_{i=1}^n \frac{1}{i!^{j_i}j_i!}\prod_{i=1}^n i!^{j_i} \\
& = & n! \sum_{\substack{j_1,\dots,j_n \\ j_1+\cdots +nj_n=n}}\prod_{i=1}^n \frac{1}{j_i!} \leq  n!\sum_{\substack{j_1,\dots,j_n \\ j_1+\cdots +nj_n=n}}1 = n! |p(n)|.
\end{eqnarray*}
We conclude the proof by recalling $t(D^\pi)\leq m/2+1-r/2$ from Theorem \ref{Theorem: Order of Wigner case for t()}. 
\end{proof}

\begin{corollary}\label{Corollary: Upper bound for cumulants of Wigner}
Let $\mathcal{X}\sim Wigner$. For any $r\geq 2$ and $m$ even, we have
\begin{equation}
|N^{r/2-1}\C_r(\text{Tr}(Y_1),\dots,\text{Tr}(Y_r))|\leq B_m (2\mathsf{C})^mm!^2 |p(m)||p(m/2)| \prod_{k=1}^m ||D_{j_k}||,
\end{equation}
where $B_m$ is the $m^{th}$-Bell number and $p(n)$ is the set of integer partitions of $n$.
\end{corollary}
\begin{proof}
First, from (\ref{Equation: Cumulant expansion}), we have 
\begin{eqnarray}\label{Aux: Aux of condition C1 assumed}
&&|N^{r/2-1}\C_r(\text{Tr}(Y_1),\dots,\text{Tr}(Y_r))| \leq \nonumber \\
&& \sum_{\pi\in \cP(\pm[m])}\sum_{\substack{\tau\in \cP(m) \\ \tau\vee\gamma=1_m}}\left|N^{r/2-1-m/2} \sum_{\substack{\psi: \pm[m]\rightarrow [N] \\ ker(\psi)=\pi}}\mathbf{D}(\psi)\right| |\C_\tau(x_{\psi^\pi(1),\psi^\pi(-1)}^{(i_1)},\dots,x_{\psi^\pi(m),\psi^\pi(-m)}^{(i_m)})|.
\end{eqnarray}
By Theorem \ref{Theorem: Order of Wigner case}, whenever
\begin{align}
\C_\tau(x_{\psi^\pi(1),\psi^\pi(-1)}^{(i_1)},\dots,x_{\psi^\pi(m),\psi^\pi(-m)}^{(i_m)})\neq 0, \label{111101}
\end{align}
we have $$\left|N^{r/2-1-m/2} \sum_{\substack{\psi: \pm[m]\rightarrow [N] \\ ker(\psi)=\pi}}\mathbf{D}(\psi)\right| \leq \prod_{k=1}^m ||D_{j_k}||\#(\pi)!|p(\#(\pi))|.$$
We also know that (\ref{111101}) holds
only when $\tau$ has all blocks of even size and $\tau\leq ker(i)$. Moreover, for a given partition $\tau$ with blocks $B_1,\dots,B_s$,  we know the partition $\pi$ is greater than or equal to the partition $\pi_\tau$.  Here $\pi_\tau$ is a partition with $2s$ blocks, and each block is determined by a block of $\tau$. Finally, observe that for a given $\tau$, by the assumption in Definition \ref{Definition: Wigner Matrix}, we have
$$|\C_\tau(x_{\psi^\pi(1),\psi^\pi(-1)}^{(i_1)},\dots,x_{\psi^\pi(m),\psi^\pi(-m)}^{(i_m)})|\leq \prod_{i} {\mathsf{C}^{|B_i|}}|B_i|!.$$
Hence, if we let $\cP_{even}(m)$ be the set of all partitions of $[m]$ with all blocks of even size, we have
\begin{equation}\label{Aux: 2}
\frac{|N^{r/2-1}\C_r(\text{Tr}(Y_1),\dots,\text{Tr}(Y_r))|}{\prod_{k=1}^m ||D_{j_k}||} \leq
\sum_{\substack{\tau\in \cP_{even}(m) \\ \tau\leq ker(i)}} B_{2\#(\tau)}2^{m-\#(\tau)}(2\#(\tau))!|p(2\#(\tau)|\prod_{\tau=\{B_1,\dots,B_s\}}{\mathsf{C}^{|B_s|}}|B_s|!.
\end{equation}
In the last expression the $2^{m-\#(\tau)}$ factor appears because for each $\tau$ and each block $B_i$ of $\tau$ there are $2^{|B_i|-1}$ possible distinct partitions $\pi_\tau$. Summing over $i$ gives $2^{m-\#(\tau)}$. Similarly, the $(2\#(\tau))!$ and $|p(2\#(\tau))|$ factors appear because $\pi_\tau \leq \pi$ and hence $\#(\pi)\leq \#(\pi_\tau)=2\#(\tau)$. Further, we notice that the number of partitions $\tau$ (without considering the restriction $\tau\leq ker(i)$) can be counted by $B_{m}(0,1,\dots,0,1)$, where $B_m(x_1,\dots,x_m)$ is the Bell polynomial. 
Since $x_i=0$ for $i$ odd and $x_i=1$ for $i$ even, we have
\begin{eqnarray*}
B_m(0,1,\dots,0,1) &=& \sum_{\substack{j_2, \dots ,j_m \\ 2j_2+\cdots+mj_m=m}} m!\prod_{i=1}^{m/2} \frac{1}{(2i)!^{j_{2i}}j_{2i}!}.
\end{eqnarray*}
Combining this with (\ref{Aux: 2})  yields
\begin{eqnarray*}
\frac{|N^{r/2-1}\C_r(\text{Tr}(Y_1),\dots,\text{Tr}(Y_r))|}{\prod_{k=1}^m ||D_{j_k}||} &\leq &
\sum_{\substack{j_2,\dots,j_m \\ 2j_2+\cdots+mj_m=m}} (2s)!|p(2s)| m! B_{2s}2^{m-s}\prod_{i=1}^{m/2} \frac{1}{(2i)!^{j_{2i}}j_{2i}!}\prod_{i=1}^{m/2}(2i)!^{j_{2i}}{\mathsf{C}^{2i j_{2i}}} \\
&& \text{ with }s=j_2+\cdots +j_m \\
&\leq & (2\mathsf{C})^mm!^2|p(m)| B_{m} \sum_{\substack{j_2,\dots ,j_m \\ 2j_2+\cdots+mj_m=m}} \prod_{i=1}^{m/2} \frac{1}{j_{2i}!},
\end{eqnarray*}
where in last inequality we used that $\frac{B_{2s}}{2^{s-1}}\leq B_{m}$ for any $1\leq s\leq m/2$ and $s\leq m/2$. Finally, we have
\begin{eqnarray*}
\frac{|N^{r/2-1}\C_r(\text{Tr}(Y_1),\dots,\text{Tr}(Y_r))|}{\prod_{k=1}^m ||D_{j_k}||} &\leq &(2\mathsf{C})^m m!^2 |p(m)| B_{m}\sum_{\substack{j_2,\dots ,j_m \\ 2j_2+\cdots+mj_m=m}} \prod_{i=1}^{m/2} \frac{1}{j_{2i}!} \\
&\leq & (2\mathsf{C})^m m!^2 |p(m)| B_{m}\sum_{\substack{j_2,\dots ,j_m \\ 2j_2+\cdots+mj_m=m}} 1. \\
\end{eqnarray*}
To conclude, it is enough to observe that the sum in the right hand side is the number of ways to partition a set of $m$ elements into integers of even size. This can be seen as the number of ways to partition a set of $m/2$ elements into integers of any size by taking the half of each type. For example $8$ has the type $2+2+2+2$, i.e. a set of $8$ elements can be partitioned into $4$ integers each one being $2$. This can be associated with $4$ having type $1+1+1+1$. Thus,
$$\sum_{\substack{j_2,\dots,j_m\geq 0 \\ 2j_2+\cdots+mj_m=m}} 1=|p(m/2)|.$$
\end{proof}

\begin{example}\label{Example: Order of Wigner case}
Let $x\sim Unif(-1/2,1/2)$. It can be verified that odd cumulants of $x$ are $0$ due to symmetry while the $2n$ cumulant is $\frac{B_{2n}}{{2n}}$, where $B_n$ is the $n$-Bell number. Let $D=(D_{i,j})_{i,j=1}^N$ be the matrix
with $D_{i,j}=1$ if $i=j$ or $j=i+1$, and $D_{i,j}=0$ otherwise. 
It is elementary to check $sup_N ||D||\leq 2$. Let $r\in\mathbb{N}$ be even and let us consider the quantity
$$\C_r(\text{Tr}(XD),\dots,\text{Tr}(XD)),$$
where $X=X_N=\frac{1}{\sqrt{N}}(x_{i,j})_{i,j=1}^N$ is a Wigner matrix with $x_{i,j}\overset{d}{=} x$ for any $i,j$. From (\ref{Equation: Cumulant expansion}) we get
\begin{eqnarray*}
\C_r(\text{Tr}(XD),\dots,\text{Tr}(XD))= N^{-r/2}\sum_{\pi\in \cP(\pm[r])}\sum_{\substack{\tau\in \cP(r)\\ \gamma\vee\tau=1_r}}\left[ \sum_{\substack{\psi: \pm[r]\rightarrow [N]\\ ker(\psi)=\pi}}D(\psi) \right]\C_{\tau}(x_{\psi^\pi(1),\psi^\pi(-1)}^{(1)},\dots,x_{\psi^\pi(r),\psi^\pi(-r)}^{(r)}),
\end{eqnarray*}
where $\gamma\in S_r$ is the identity permutation. From the condition $\tau\vee\gamma=1_r$ we conclude that $\tau=1_r$ is the only choice and therefore, 
$$\C_{\tau}(x_{\psi^\pi(1),\psi^\pi(-1)}^{(1)},\dots,x_{\psi^\pi(r),\psi^\pi(-r)}^{(r)})=\C_{r}(x_{\psi^\pi(1),\psi^\pi(-1)}^{(1)},\dots,x_{\psi^\pi(r),\psi^\pi(-r)}^{(r)}).$$
If this quantity is nonzero, $x_{\psi^\pi(1),\psi^\pi(-1)}^{(1)},\dots,x_{\psi^\pi(r),\psi^\pi(-r)}^{(r)}$ must be all the same random variable. This is possible only if either $\pi$ has only one block or $\pi$ consists of two blocks $A$ and $B$ where for any $1\leq k\leq r$ either $k\in A$ and $-k\in B$ or $k\in B$ and $-k\in A$. When $\pi$ has only one block the corresponding cumulant is $\C_r(x_{1,1})$ while in the second case one has a cumulant of the form $\C_r(x_{1,2}^{\epsilon_1},\dots,x_{1,2}^{\epsilon_r})$, for some $\epsilon_i\in \{1,-1\}$ that depends on how the blocks $A$ and $B$ are chosen. In either case we obtain $\C_r(x)=\frac{B_r}{r}$. Hence, in these two cases, we have the contributing terms as follows
\begin{align*}\sum_{\substack{j_1,j_{-1}\dots,j_r,j_{-r}=1 \\ j_1=j_{-1}=\cdots = j_r=j_{-r}}}^N\prod_{i=1}^r D_{j_{-i},j_{i}},\qquad {\text{and}} \qquad \sum_{\substack{j_1,j_{-1}\dots,j_r,j_{-r}=1 \\ j_1=\cdots = j_r \\ j_{-1}=\cdots = j_{-r} \\ j_1\neq j_{-1}}}^N\prod_{i=1}^r D_{j_{-i},j_{i}},
\end{align*}
which equals $N$ and $N-1$, respectively. 
Hence, $\C_r(\text{Tr}(XD),\dots,\text{Tr}(XD))$ has leading order $N^{1-r/2}$.
\end{example}

Nevertheless, we notice that in the special case when $D$ is diagonal and the diagonal entries of $X$ are Gaussian, we still obtain the order $N^{2-r}$.

\section{Statulevi$\check{c}$ius condition and quantitative laws}\label{Section: The Statulevicius condition}

In this section, we provide proofs to all our main results, based on the estimates obtained in previous sections. First, we derive the cumulant bounds in Theorems \ref{Theorem: Upper bound of cumulants GUE} and \ref{Theorem: Upper bound of cumulants Wigner}. Then we use these bounds to prove Theorems \ref{Theorem: Berry-Esseen bound GUE/GOE Main result 1} and \ref{Theorem: Berry-Esseen bound Wigner}.

Let us first introduce some facts from \cite{DJ}, regarding the cumulant method for quantitative CLTs. For a real random variable $Y$ centered and normalized, one says that it satisfies the {\it Statulevi$\check{c}$ius condition} $(S_\delta)$, if its cumulants can be upper bounded as
$$|K_j(Y)|\leq \frac{j!^{1+\delta}}{\Delta^{j-2}},$$
for all $j\geq 3$,  with some $\delta \geq 0$ and $\Delta >0$. Let us also denote by $Z$ the standard normal variable. First, we have \cite[Theorem 2.3]{DJ} which states the following CLT with Cram\'{e}r correction: If $Y$ satisfies $(S_\delta)$, then for all $x\in [0,\Delta_\delta)$,
\begin{equation}\label{Equation: Theorem of DJ for Cramer correction}
\mathbb{P}(Y\geq x)=e^{{L(x)}}\mathbb{P}(Z\geq x)\left(1+O\left(\frac{x+1}{\Delta^{1/(1+2\delta)}}\right)\right),
\end{equation}
where 
\begin{align}
|{L(x)}|= O(x^3/\Delta^{1/(1+2\delta)}), \qquad {\text{and}} \qquad \Delta_\delta=\frac{1}{6}\left(\frac{\Delta}{3\sqrt{2}}\right)^{1/(1+2\delta)}.
\end{align}
Next, we have \cite[Theorem 2.4]{DJ} which states the following Berry-Esseen type bound: If $Y$ satisfies $(S_\delta)$, then
\begin{equation}\label{Equation: Theorem of DJ for Berry Esseen}
\sup_{x\in\mathbb{R}}\big|P(Y\geq x)-P(Z\geq x)\big|\leq \frac{C_\delta}{\Delta^{1/(1+2\delta)}},
\end{equation}
for some constant $C_\delta>0$ that does not depend on $Y$ or $\Delta$. Finally, we have \cite[Theorem 2.5]{DJ} which states the following concentration inequality: If $Y$ is mean $0$ and for all $j$
\begin{equation}\label{Condition: Similar to Statulevivius}
|\C_j(Y)|\leq \frac{j!^{1+\delta}}{2}\frac{H}{\Bar{\Delta}^{j-2}},
\end{equation}
for some $\delta\geq 0$ and $H,\Bar{\Delta}>0$, there exists a constant $A>0$, such that for all $x\geq 0$,
\begin{equation}\label{Equation: Theorem of DJ for Concetration}
\mathbb{P}(Y\geq x)\leq A \exp\left(-\frac{1}{2}\frac{x^2}{H+x^{2-\alpha}/\Bar{\Delta}^{\alpha}}\right),
\end{equation}
where the constant $A$ does not depend on $Y, H, \Bar{\Delta}$ or $\delta$,  and here $\alpha=1/(1+\delta)$.

Note that, our cumulants bounds in Theorems \ref{Theorem: Upper bound of cumulants GUE}-\ref{Theorem: Upper bound of cumulants Wigner} can be regarded as 
 Statulevi$\check{c}$ius type upper bound for the random variable $\mathbf{X}$, which can also be regarded as the bounds for the normalized variable $\mathfrak{X}$ by adjusting the values of $\theta_1,\theta_2$ and $\theta_3$, under the assumption $\mathbf{C}_2$.

Let us now prove Theorems \ref{Theorem: Upper bound of cumulants GUE} and \ref{Theorem: Upper bound of cumulants Wigner}.

\begin{proof}[Proof of Theorem \ref{Theorem: Upper bound of cumulants GUE}]
Let us prove it in details for GUE,  the same proof applies to the GOE case. Let $r\geq 3$. By the multi-linearity of the cumulants and  translation invariance, we have
$$N^{r-2}\C_r(\mathbf{X})=N^{r-2}\sum_{\textbf{v}} \C_r(Y_1(\textbf{v}),\dots,Y_r(\textbf{v})).$$
Here $\textbf{v}$ is a vector of size $r$ with entries from the set $\{1,\dots,t\}$, allowing repetitions. And for $1\leq k\leq r$, the $k^{th}$ index of $\textbf{v}$, $\textbf{v}(k)$, determines the variable $Y_k(\textbf{v})$ via
$$Y_k(\textbf{v})=X_{i^{(\mathbf{v}(k))}_1}D_{j^{(\mathbf{v}(k))}_1}\cdots X_{i^{(\mathbf{v}(k))}_{m_{\mathbf{v}(k)}}}D_{j^{(\mathbf{v}(k))}_{m_{\mathbf{v}(k)}}}.$$
Now, we use Lemma \ref{Corollary: Upper bound for cumulants of GUE}. For each choice of indices $\mathbf{v}$, the corresponding cumulant satisfies the bound
$$|N^{r-2}\C_r(Y_1(\textbf{v}),\dots,Y_r(\textbf{v}))|\leq m(\mathbf{v})!! \prod_{k=1}^r \prod_{u=1}^{m_{\textbf{v}(k)}}||D_{j^{(k)}_{u}}||,$$
where $m(\mathbf{v})=\sum_{k=1}^r m_{\mathbf{v}(k)}$ is even, otherwise the last quantity is simply $0$. Note that regardless of the choice $\mathbf{v}$,  the quantity $m(\mathbf{v})$ is upper bounded by $Mr$. On the other hand, if $K\leq 1$ then $\prod_{k=1}^r \prod_{u=1}^{m_{\textbf{v}(k)}}||D_{j^{(k)}_{u}}||\leq 1$, otherwise $\prod_{k=1}^r \prod_{u=1}^{m_{\textbf{v}(k)}}||D_{j^{(k)}_{u}}||\leq K^{m(\mathbf{v})} \leq K^{Mr}$. Here $M$ and $K$ are defined as in (\ref{Definition: M}) and (\ref{Definition: K}) respectively.  Hence,
\begin{eqnarray*}
|N^{r-2}\C_r(\mathbf{X})| & \leq & \sum_{\mathbf{v}}|N^{r-2}\C_r(Y_1(\mathbf{v}),\dots,Y_r(\mathbf{v}))| \\
&\leq & \sum_{\mathbf{v}} \{Mr\}!! \max\{K,1\}^{Mr},
\end{eqnarray*} 
where in the last expression, the term $\{Mr\}$ means $Mr$ if $Mr$ is even or $Mr-1$ if $Mr$ is odd. It is  then elementary to bound
\begin{eqnarray}\label{Aux: Proof of Statelivicus equ 1}
|N^{r-2}\C_r(\mathbf{X})| \leq \frac{t^r \max\{1,K\}^{Mr} \{Mr\}!}{ (\{Mr\}/2)!2^{\{Mr\}/2}} =t^r \max\{1,K\}^{Mr} \times \frac{\{Mr\}!}{2^{\{Mr\}/2}(\{Mr\}/2)!}. 
\end{eqnarray}
 Let $\beta={2^{M/2}}/({t\max\{1,K\}^M})$. 
Applying basic  Stirling's approximation one can show that  there exists a constant $C \equiv C(K,M,t)>0$ such that
$$|\C_r(\mathbf{X})|\leq C\frac{r!^{M/2}}{(\beta N)^{r-2}}.$$
Absorbing $C$ into $\beta$ we can conclude that there exists a constant $\theta_1\equiv \theta_1(K,M,t)>0$ such that
$$|\C_r(\mathbf{X})|\leq \frac{r!^{M/2}}{(\theta_1 N)^{r-2}}.$$
To finish our proof, note that the cumulants of GUE and GOE are upper bounded by the same quantity up to a factor $2^{Mr/2}$. Hence, asymptotically the same reasoning applies to GOE by adjusting the parameter $\beta$ (and hence $\theta_1$) appropriately.
\end{proof}


\begin{proof}[Proof of Theorem \ref{Theorem: Upper bound of cumulants Wigner}]
It suffices to prove that there exists constants $C\equiv C(K, M, \mathsf{C}, t)$ and $\beta \equiv \beta(K,M,\mathsf{C}, t)$, such that
\begin{equation}
|\C_r(\mathbf{X})|\leq C\frac{r!^{3M}}{(\beta \sqrt{N})^{r-2}},
\end{equation}
for any $N$ and $r\geq 3$. The proof can be done similarly to that of Theorem \ref{Theorem: Upper bound of cumulants GUE}. From Corollary \ref{Corollary: Upper bound for cumulants of Wigner}, the cumulant $|N^{r/2-1}K_r(\mathbf{X})|$ is now upper bounded by $$B_{Mr}{(2\mathsf{C})^{Mr}}(Mr)!^2 |p(Mr)||p(Mr/2)| \max\{1,K\}^{Mr}t^r.$$ {{
First, it is known that $B_m\leq m!$. Further, it is also  well known that for sufficiently large $n$, 
$$|p(n)|\sim \frac{1}{4n\sqrt{3}}\exp\left(\pi \sqrt{\frac{2n}{3}}\right),$$
which further implies that  for any $n$,
$$|p(n)|\leq \frac{A}{n}\exp\left(\pi \sqrt{\frac{2n}{3}}\right)$$
for some constant $A>0$ uniformly in $n$. Therefore $|N^{r/2-1}K_r(\mathbf{X})|$ is now upper bounded by 
\begin{eqnarray*}
&& B_{Mr}{(2\mathsf{C})^{Mr}}(Mr)!^2 |p(Mr)|^2 \max\{1,K\}^{Mr}t^r \\
&\leq & (Mr)!{(2\mathsf{C})^{Mr}}(Mr)!^2 A^2\frac{1}{M^2r^2}\exp\left(2\pi \sqrt{\frac{2Mr}{3}}\right) \max\{1,K\}^{Mr}t^r \\
&\leq & (Mr)!^3{(2\mathsf{C})^{Mr}} A^2\frac{1}{M^2r}\exp\left(2\pi Mr\right) \max\{1,K\}^{Mr}t^r \\
\end{eqnarray*}
In the last expression the term after $(Mr)!^3$ is of the form $\frac{(C_1)^r}{C_2r} $ for some constants $C_1,C_2>0$. Therefore $|N^{r/2-1}K_r(\mathbf{X})|$ is upper bounded by $\frac{(C_1)^r}{C_2r} (Mr)!^3 $ for some constants $C_1,C_2>0$. 
}}
In case $\mathbf{C}_1'$ is additionally assumed,
we claim that
\begin{equation}\label{Aux: Claiming of last proof}
|N^{r/2-1}\C_r(\text{Tr}(Y_1),\dots,\text{Tr}(Y_r))|\leq (2\mathsf{C})^mm!|p(m/2)| \prod_{k=1}^m ||D_{j_k}||.
\end{equation}
Indeed, the proof follows similarly as Corollary \ref{Corollary: Upper bound for cumulants of Wigner}. The major difference is now
\begin{eqnarray*}
\left|N^{r/2-1-m/2}\sum_{\substack{\psi: \pm [m]\rightarrow [N] \\ ker(\psi)=\pi}}\mathbf{D}(\psi)\right| & \leq & N^{r/2-1-m/2}\sum_{\substack{\psi: \pm [m]\rightarrow [N] \\ ker(\psi)=\pi}}\mathbf{D}^{abs}(\psi).
\end{eqnarray*}
Here $\mathbf{D}^{abs}(\psi)$ is defined as $\mathbf{D}(\psi)$ via replacing the $d_{i,j}$'s by their absolute values. 
Hence (\ref{Aux: Aux of condition C1 assumed}) becomes,
\begin{eqnarray}\label{Aux: auxiliar of last proof}
&&|N^{r/2-1}\C_r(\text{Tr}(Y_1),\dots,\text{Tr}(Y_r))| \leq \nonumber \\
&& \sum_{\pi\in \cP(\pm[m])}\sum_{\substack{\tau\in \cP(m) \\ \tau\vee\gamma=1_m}}N^{r/2-1-m/2}\sum_{\substack{\psi: \pm [m]\rightarrow [N] \\ ker(\psi)=\pi}}\mathbf{D}^{abs}(\psi)|\C_\tau(x_{\psi^\pi(1),\psi^\pi(-1)}^{(i_1)},\dots,x_{\psi^\pi(m),\psi^\pi(-m)}^{(i_m)})|.
\end{eqnarray}
Let us remind that
$$\C_\tau(x_{\psi^\pi(1),\psi^\pi(-1)}^{(i_1)},\dots,x_{\psi^\pi(m),\psi^\pi(-m)}^{(i_m)}) \neq 0$$
only if $\tau\in \cP_{even}(m)$ and $\pi\geq \pi_\tau$ where $\pi_\tau$ is a partition with $2\#(\tau)$ blocks.{{ Here let us remind that $\cP_{even}(m)$ denotes the set of all partitions of $[m]$ with all blocks of even size}}. Further, the partition $\pi_\tau$ can be chosen in $2^{m-\#(\tau)}$ ways. In the sum (\ref{Aux: auxiliar of last proof}) all the terms are positive, hence
\begin{eqnarray*}
&&|N^{r/2-1}\C_r(\text{Tr}(Y_1),\dots,\text{Tr}(Y_r))| \leq \nonumber \\
&& \sum_{{\tau\in \cP_{even}(m)}}2^{m-\#(\tau)}\sum_{\substack{\pi\in \cP(\pm[m]) \\ \pi \geq \pi_\tau}}N^{r/2-1-m/2}\sum_{\substack{\psi: \pm [m]\rightarrow [N] \\ ker(\psi)=\pi}}\mathbf{D}^{abs}(\psi)|\C_\tau(x_{\psi^\pi(1),\psi^\pi(-1)}^{(i_1)},\dots,x_{\psi^\pi(m),\psi^\pi(-m)}^{(i_m)})|,
\end{eqnarray*}
because we might be over counting some partitions $\pi$ but as the terms are all positive we get the inequality. Consequently, we have 
\begin{eqnarray*}
&&|N^{r/2-1}\C_r(\text{Tr}(Y_1),\dots,\text{Tr}(Y_r))| \leq \nonumber \\
&& 2^{m-1}\sum_{{\tau\in \cP_{even}(m)}}N^{r/2-1-m/2}\sum_{\substack{\psi: \pm [m]\rightarrow [N] \\ ker(\psi)\geq \pi_\tau}}\mathbf{D}^{abs}(\psi)|\C_\tau(x_{\psi^\pi(1),\psi^\pi(-1)}^{(i_1)},\dots,x_{\psi^\pi(m),\psi^\pi(-m)}^{(i_m)})|.
\end{eqnarray*}
Further, note that 
$$N^{r/2-1-m/2}\sum_{\substack{\psi: \pm [m]\rightarrow [N] \\ ker(\psi)\geq \pi_\tau}}\mathbf{D}^{abs}(\psi)\leq \prod_{k=1}^m ||D^{abs}_{j_k}||$$
 according to Lemma \ref{Lemma: Order of sums associated to deterministic matrices} and the fact that $t(D^{\pi_\tau})\leq m/2+1-r/2$ as proved in Theorem \ref{Theorem: Order of Wigner case for t()}. The remaining proof of (\ref{Aux: Claiming of last proof}) is then the same as Corollary \ref{Corollary: Upper bound for cumulants of Wigner}. Correspondingly, the rest of the proof of the cumulant bound follows via replacing $(Mr)!^3$ by $(Mr)!$.
\end{proof}

We then use the conclusions in Theorems \ref{Theorem: Upper bound of cumulants GUE} and \ref{Theorem: Upper bound of cumulants Wigner}  together with (\ref{Equation: Theorem of DJ for Cramer correction}), (\ref{Equation: Theorem of DJ for Berry Esseen}) and (\ref{Equation: Theorem of DJ for Concetration}) to prove Theorems \ref{Theorem: Berry-Esseen bound GUE/GOE Main result 1} and \ref{Theorem: Berry-Esseen bound Wigner}.

\begin{proof}[Proof of Theorems \ref{Theorem: Berry-Esseen bound GUE/GOE Main result 1} and \ref{Theorem: Berry-Esseen bound Wigner}]
We first prove Theorem \ref{Theorem: Berry-Esseen bound GUE/GOE Main result 1}.  Observe that thanks to Theorem \ref{Theorem: Upper bound of cumulants GUE}, we know the cumulants of $\mathbf{X}$ satisfy the Statulevi$\check{c}$ius condition $(S_{\gamma})$ with $\Delta=\theta_1 N$ and $\gamma=M/2-1$, which further implies that the normalized random variable $\mathfrak{X}$ also satisfies this condition with slightly changed $\theta_1$, thanks to the condition $\mathbf{C}_2$. Thus $1+2\gamma=M-1$. Therefore from (\ref{Equation: Theorem of DJ for Cramer correction}), it follows part $(i)$, and from (\ref{Equation: Theorem of DJ for Berry Esseen}), it follows part $(ii)$. To prove part $(iii)$, observe that Theorem \ref{Theorem: Upper bound of cumulants GUE} says that $\mathbf{X}$ and also $\mathfrak{X}$ satisfy condition (\ref{Condition: Similar to Statulevivius}) with $\alpha=1/(1+\gamma)=1/(M/2)=2/M$. Therefore, from  (\ref{Equation: Theorem of DJ for Concetration}) it follows part $(iii)$.

Finally,  that the proof of Theorem \ref{Theorem: Berry-Esseen bound Wigner} follows the same as the proof of Theorem \ref{Theorem: Berry-Esseen bound GUE/GOE Main result 1}, by changing the corresponding $\gamma$ and $\Delta$.
\end{proof}


\section*{Acknowledgments} We would like to thank the reviewers for helpful comments and suggestions. 
Z. Bao is supported by
Hong Kong RGC Grant 16303922, NSFC12222121 and NSFC12271475. D. Munoz George is supported by Hong Kong RGC Grant 16304724 and  National Key R\&D Program of China (No. 2023YFA1010400).

\end{document}